\documentclass[a4paper,11pt]{article}
\usepackage{amsmath,amsxtra,amssymb,latexsym, amscd,amsthm}
\usepackage{indentfirst}
\usepackage[mathscr]{eucal}

\usepackage{hyperref}  
\newlength{\longarc}
\newcommand{\wtilde}[1]{\settowidth{\longarc}{$#1$}%
         \addtolength{\longarc}{-0.1em}%
         \unitlength \longarc%
         \ensuremath{%
         \stackrel{\begin{picture}(1,0.05)
                 \qbezier(0,0)(0.25,0.05)(0.5,0)
                 \qbezier(0.5,0)(0.75,-0.05)(1,0)
         \end{picture}}
         {#1}
}}

\begin{document}
\title{Carleman estimates for semi-discrete parabolic operators with a discontinuous diffusion coefficient and application to controllability}


\author{Thuy N.T. Nguyen\thanks{Universit\'{e} d'Orl\'{e}ans, Laboratoire Math\'{e}matiques et Applications, Physique Math\'{e}matiques d'Orl\'{e}ans (MAPMO), B\^{a}timent de Math\'{e}matiques, B.P. 6759, 45067 Orléans cedex 2, FRANCE. Email address : {\tt nguyentnthuy318@gmail.com}}}
\date{}

\maketitle
\renewcommand{\abstractname}{Abstract}
 \renewcommand{\proofname}{Proof}

\begin{abstract}In the discrete setting of one-dimensional finite-differences we prove a Carleman estimate for a semi-discretization of the parabolic operator $\partial_t-\partial_x (c\partial_x )$ where the diffusion coefficient $c$ has a jump. As a consequence of this Carleman estimate, we deduce consistent null-controllability results for classes of semi-linear parabolic equations.
\end{abstract}

\numberwithin{equation}{section}

\newtheorem{Theorem}{\bf Theorem}[section]
\newtheorem{Corollary}[Theorem]{\bf Corollary}
\newtheorem{Lemma}[Theorem]{\bf Lemma}
\newtheorem{Example}[Theorem]{\bf Example}
\newtheorem{Proposition}[Theorem]{\bf Proposition}
\newtheorem{Assumption}[Theorem]{\bf Assumption}
\newtheorem{Conjecture}[Theorem]{\bf Conjecture}
\newtheorem{Problem}[Theorem]{\bf Problem}
\newtheorem{Remark}[Theorem]{\bf Remark}
\newtheorem{Definition}[Theorem]{\bf Definition}
\newtheorem{Claim}{\bf Claim}
\newtheorem*{my2}{Theorem \ref{my2}}
\newtheorem*{my3}{Theorem \ref{my3}}

\section{Introduction and settings}
\label{sec: introduction}

\par Let $\Omega,\omega$ be connected non-empty open interval of $\mathbb{R}$ with $\omega\Subset \Omega$. We consider the following parabolic problem in $(0,T)\times \Omega$, with $T>0$,
\begin{equation}
  \label{eq: control heat equation}
  \partial_ty-\partial_x (c\partial_x y)=\textbf{1}_{\omega}v \ \ \textmd{in} \ (0,T)\times \Omega, \ y|_{\partial\Omega}=0, \ \textmd{and} \ y|_{t=0}=y_0, 
\end{equation}
where the diffusion coefficient $c=c(x)>0$.
\par System ~\eqref{eq: control heat equation} is said to be null controllable from $y_0\in L^2(\Omega)$ in time $T$ if there exists $v\in L^2((0,T)\times \Omega )$, such that $y(T)=0$. 

\par In the continuous framework, we refer to \cite{FI96} and \cite{LR95} who proved such a controllability result by means of a global/local Carleman observability estimates in the case the diffusion coefficient $c$ is smooth. The authors of \cite{BDL07} produced this controllability result in the case of a discontinuous coefficient in the one-dimensional case later extended to arbitrary dimension by \cite{LR10}. Additionally, a result of controllability in the case of a coefficient with bounded variation (BV) was shown in \cite{FCZ02, L07}. 

\par The authors of \cite{LZ98} show that uniform controllability holds in the one-dimensional case with constant diffusion coefficient $c$ and for a constant step size finite-difference scheme. Here, "uniform" is meant with respect to the discretization parameter $h$. The situation becomes more complex in higher dimension. In fact, a counter-example to null-controllability due to O. Kavian is provided in \cite{Zua06} for a finite-difference discretization scheme for the heat equation in a square. 

\par In recent works, by means of discrete Carleman estimate, the authors of \cite{BHL10a}, \cite{BHL10b} and \cite{BL12} obtained weak observability inequalities in the case of a smooth diffusion coefficient $c(x)$. Such observability estimates are charaterized by an additional term that vanishes exponentially fast. Morever, also with a constant diffusion coffiencient $c$, under the assumption that the discretized semigroup is uniformly analytic and that the degree of unboundedness of control operator is lower than 1/2, a uniform observability property of semi-discrete approximations for System ~\eqref{eq: control heat equation} is achieved in $L^2$ \cite{LT06}. Besides that, such a result continues to hold even with the condition that the degree of unboundedness of control operator is greater than 1/2 \cite{N12}.

\par In the case of a non-smooth coefficient, our aim is to investigate the uniform controllability of System \eqref{eq: control heat equation} after discretization. It is well known that controllability and observability are dual aspects of the same problem. We shall therefore focus on uniform observability which is shown to hold when the observability constant of the finite dimensional approximation systems does not depend on the step-size $h$.
\par In the present paper we prove a Carleman estimate for system ~\eqref{eq: control heat equation} in the case of:
\begin{itemize}
\item the heat equation in one space dimension;
\item a piecewise ${\textbf{C}}^1$ coefficient $c$ with jumps at a finite number of points in $\Omega$;
\item a finite-difference discretization in space.
\end{itemize}
The main idea of the proof is combination of the derivation of a discrete Carleman estimate as in \cite{BHL10a, BL12} and tecniques of \cite{BDL07} for operators with discontinuous coefficients in the one-dimensional case. A similar question in $n$-dimensional case, $n\ge 2$, remains open, to our knowledge. 
\par When considering a discontinuous coefficient $c$ the parabolic problem ~\eqref{eq: control heat equation} can be understood as a transmission problem. For instance, assume that $c$ exhibits a jump at $a\in\Omega$.  Then we write
\begin{center}$\quad\quad\begin{cases}
\partial_ty-\partial_x (c\partial_x y)=\textbf{1}_{\omega}v \quad \quad\quad \textmd{in} \ (0,T)\times \big((0,a)\cup(a,1)\big), &\\
c\partial_xy|_{a^+}=c\partial_xy|_{a^-},\quad\quad y|_{a^+}=y|_{a^-},&\\
 y|_{\partial\Omega}=0, \quad\quad \textmd{and}\quad\quad\quad \ y|_{t=0}=y_0.&
\end{cases}$
\end{center}
The second line is thus a transmission condition implying the continuity of the solution and of the flux at $x=a$.
\par When one gives a finite-difference version of this transmission problem, a similar condition can be given for the continuity of the solution. Yet, for the flux, it is only achieved up to a consistent term. In what follows, in the finite-difference approximation, we shall in fact write

\begin{center}$\begin{cases}
y(a^-)=y(a^+)=y_{n+1} ,&\\
(c_dDy)_{n+\frac{3}{2}}-(c_dDy)_{n+\frac{1}{2}}=h\big(\bar{D}(c_dDy)\big)_{n+1},&
\end{cases}$
\end{center} 
(the discrete notation will be given below). Note that the flux condition converges to the continuous one if $h\to 0$, $h$ being the discretization parameter. This difference between the continuous and the discrete case will be the source of several technical points. 

\par An important point in the proof of Carleman estimate is the construction of a suitable weight function $\psi$ whose gradient does not vanish in the complement of the observation region. The weight function is chosen smooth in the case of a smooth diffusion coefficient $c(x)$. In general, the technique to construct such a function is based on Morse functions (see some details in \cite{FI96}). In one space dimension, this construction is in fact straightforward. In the case of a discontinuous diffusion coefficient, authors of \cite{BDL07} introduced an \textit{ad hoc} transmission condition on the weight function: its derivative exhibits jumps at the singular points of the coefficient. In this paper, we construct a weight function based on these techniques in the one-dimentional discrete case.
\par From the semi-discrete Carleman we obtain, we give an observability inequality for semi-discrete parabolic problems with potential. As compared to the result in continuous case \cite{BDL07} the observability estimate we state here is weak because of an additional term that describes the obstruction to the null-controllability. This term is exponentially small in agreement with the results obtained in \cite{BHL10a,BHL10b} in the smooth coefficient case. A precise statement is given in Section~\ref{sec:Controllability results}.
\par Finally, the observability inequality allows one to obtain controllability results for semi-discrete parabolic with semi-linear terms. In continuous case, this was achieved in \cite{BDL07}. Taking advantage of one-dimensional situation, the results we state are uniform with respect to the discretization parameter $h$ (see Section~\ref{sec:Controllability results}).
\subsection{Discrete settings}\label{subsec:discrete setting}
\par We restrict our analysis to one dimension in space. Let us consider the operator formally defined by $\mathcal{A}=-\partial_x(c\partial_x)$ on the open interval $\Omega=(0,L)\subset \mathbb{R}$. We let $a'\in {\Omega}$ and set $\Omega_1:=(0,a')$ and $\Omega_2:=(a',L)$. The diffusion coefficient c is assumed to be piecewise regular such that 
\begin{equation}\label{eq:coefficient}0<c_{min}\le c\le c_{max}
\end{equation} 
\begin{center}$c=\begin{cases}
{c_0}~\textrm{in}~ \Omega_1, &\\
{c_1}~\textrm{in}~ {\Omega_2},&
\end{cases}$
\end{center}
with ${c_i}\in C^1(\overline{{\Omega_i}}), i=1,2$.
\par The domain of $\mathcal{A}$ is $D(\mathcal{A})=\left\{{u\in H^1_0(\Omega);\quad c\partial_x u\in H^1(\Omega)}\right\}.$
\par Let $T>0$. We shall use the following notation $\Omega'=\Omega_1\cup \Omega_2$, $Q=(0,T)\times \Omega$, $Q'=(0,T)\times \Omega'$, $Q_i=(0,T)\times \Omega_i,\ i=1,2$, $\Gamma=\left\{{0,L}\right\}$, and $\Sigma =(0,T)\times\Gamma$. We also set $S=\left\{{a'}\right\}$. We consider the following parabolic problem

\begin{center}$ \begin{cases}
\partial_t y+\mathcal{A}y=f~\textrm{in}~ Q' ,&\\
y(0,x)=y_0(x)~\textrm{in}~\Omega& .
\end{cases}$
\end{center}
(real valued coefficient and solution), for $y_0\in L^2(\Omega)$ and $f\in L^2(Q)$, with the following transmission conditions at $a'$
 \begin{center}$(TC) \begin{cases}
{y(a'^-)=y(a'^+)}, &\\
{c(a'^-)\partial_x y(a'^-)=c(a'^+)\partial_x y(a'^+)}.&
\end{cases}$
\end{center}
\par Now, we introduce finite-difference approximations of the operator $\mathcal{A}$. Let $0=x'_0<x'_1<\ldots<\underline{x'_{n+1}=a'}<\ldots<x'_{n+m+1}<x'_{n+m+2}=L$. We refer to this discretization as to the primal mesh $\mathfrak{M}:=(x'_i)_{1\le i\le n+m+1}$. We set $\left|{\mathfrak{M}}\right|:=n+m+1$. We set $h'_{i+\frac{1}{2}}=x'_{i+1}-x'_i$ and $x'_{i+\frac{1}{2}}=(x'_{i+1}+x'_i)/2$, $i=0,\ldots,n+m+1$, and $h'=\max_{0\le i\le n+m+1}h'_{i+\frac{1}{2}}$. We call $\overline{\mathfrak{M}}:=(x'_{i+\frac{1}{2}})_{0\le i\le n+m+1}$ the dual mesh and set $h'_i=x'_{i+\frac{1}{2}}-x'_{i-\frac{1}{2}}=(h'_{i+\frac{1}{2}}+h'_{i-\frac{1}{2}})/2$, $i=0,\ldots,n+m+1$.\\

\par In this paper, we shall address to some families of non uniform meshes, that will be precisely defined in Section~\ref{subsec:Families of non-uniform meshes}.
\par We introduce the following notation
\begin{eqnarray}&&\label{eq:notation1}[\rho_1\star]_a=\rho_1(a^+)-\rho_1(a^-),\\
&&\label{eq:notation2}[\star \rho_2]_a=\rho_2(n+\frac{3}{2})-\rho_2(n+\frac{1}{2}),\\
&&\label{eq:notation3}[\rho_1 \star \rho_2]_a=\rho_1(a^+)\rho_2(n+\frac{3}{2})-\rho_1(a^-)\rho_2(n+\frac{1}{2}).
\end{eqnarray}

\par We follow some notation of \cite{BHL10a} for discrete functions in the one-dimensional case. We denote by $\mathbb{C}^{\mathfrak{M}}$ and $\mathbb{C}^{\overline{\mathfrak{M}}}$ the sets of discrete functions defined on $\mathfrak{M}$ and $\overline{\mathfrak{M}}$ respectively. If $u\in\mathbb{C}^{\mathfrak{M}}$ (resp. $\mathbb{C}^{\overline{\mathfrak{M}}}$), we denote by $u_i$ (resp. $u_{i+\frac{1}{2}}$) its value corresponding to $x'_i$ (resp. $x'_{i+\frac{1}{2}}$). For $u\in \mathbb{C}^{{\mathfrak{M}}} $ we define
\begin{equation}u^{\mathfrak{M}}=\sum_{i=1}^{n+m+1}1|_{[x'_{i-\frac{1}{2}},x'_{i+\frac{1}{2}}]}u_i\in L^{\infty}(\Omega).\nonumber
\end{equation}

\par And for $u\in \mathbb{C}^{{\mathfrak{M}}} $ we define $\int_{\Omega}u:=\int_{\Omega}u^{\mathfrak{M}}(x)dx=\sum_{i=1}^{n+m+1}h'_iu_i.$
\par For $u\in \mathbb{C}^{\overline{\mathfrak{M}}} $ we define
\begin{equation}u^{\overline{\mathfrak{M}}}=\sum_{i=0}^{n+m+1}1|_{[x'_{i},x'_{i+1}]}u_{i+\frac{1}{2}}.\nonumber
\end{equation}

\par As above, for $u\in \mathbb{C}^{\overline{\mathfrak{M}}} $ , we define $\int_{\Omega}u:=\int_{\Omega}u^{\overline{\mathfrak{M}}}(x)dx=\sum_{i=0}^{n+m+1}h'_{i+\frac{1}{2}}u_{i+\frac{1}{2}}.$ In particular we define the following $L^2$ inner product on $\mathbb{C}^{\mathfrak{M}}$ (resp. $\mathbb{C}^{\overline{\mathfrak{M}}}$)
\begin{equation}(u,v)_{L^2}=\int_{\Omega}{u^{\mathfrak{M}}(x)v^{\mathfrak{M}}(x)dx},\ ~\textrm{resp.}~  (u,v)_{L^2}=\int_{\Omega}{u^{\overline{\mathfrak{M}}}(x)v^{\overline{\mathfrak{M}}}(x)dx}.\nonumber
\end{equation}
\par For some $u\in \mathbb{C}^{{\mathfrak{M}}} $, we shall need to associate boundary conditions $u^{\partial\mathfrak{M}}=\left\{{u_0,u_{n+m+2}}\right\}$. The set of such extended discrete functions is denoted by $\mathbb{C}^{\mathfrak{M}\cup \partial \mathfrak{M}}$. Homogeneous Dirichlet boundary conditions then consist in the choice $u_0=u_{n+m+2}=0$, in short $u^{\partial\mathfrak{M}}=0$. We can define translation operators $\tau^{\pm}$, a difference operator $D$ and an averaging operator as the map $\mathbb{C}^{\mathfrak{M}\cup \partial\mathfrak{M}}\to \mathbb{C}^{\overline{\mathfrak{M}}}$ given by
\begin{eqnarray}&&(\tau^{+}u)_{i+\frac{1}{2}}:=u_{i+1}, \ \ \ \ \  (\tau^{-}u)_{i+\frac{1}{2}}:=u_i, \ \ \ \  i=0,\ldots n+m+1,\nonumber\\
&&(Du)_{i+\frac{1}{2}}:=\frac{1}{h'_{i+\frac{1}{2}}}(\tau^{+}u- \tau^-u)_{i+\frac{1}{2}}, \ \ \ \   \tilde{u}:=\frac{1}{2}(\tau^{+} +\tau^-)u.\nonumber
\end{eqnarray}
\par We also define, on the dual mesh, translation operators $\tau^{\pm}$,  a difference operator $\bar{D}$ and an averaging operator as the map $\mathbb{C}^{\overline{\mathfrak{M}}}\to\mathbb{C}^{\mathfrak{M}}$ given by
\begin{eqnarray}&&(\tau^{+}u)_{i}:=u_{i+\frac{1}{2}}, \ \ \ \ \  (\tau^{-}u)_{i}:=u_{i-\frac{1}{2}},\ \ \ \  i=1,\ldots n+m+1,\nonumber\\
&&(\bar{D}u)_{i}:=\frac{1}{h'_{i}}(\tau^{+}u-\tau^-u)_{i}, \ \ \ \  \bar{u}:=\frac{1}{2}(\tau^{+} +\tau^{-})u.\nonumber
\end{eqnarray}

\subsection{Families of non-uniform meshes}\label{subsec:Families of non-uniform meshes}
\par In this paper, we address non-uniform meshes that are obtained as the smooth image of an uniform grid.
\par More precisely, let $\Omega_0=]0,1[$ and let $\vartheta: \mathbb{R}\to \mathbb{R}$ be an increasing map such that
\begin{equation} \label{eq:vartheta} \vartheta(\Omega_0)=\Omega, \quad \vartheta\in \mathcal{C}^{\infty}, \quad {\inf \vartheta'}>0 ~\textmd{and}\quad  \vartheta(a)=a'
\end{equation}
with $a$ to be kept fixed in what follows and chosen such that $a\in (0,1)\cap \mathbb{Q}$, i.e $a=\frac{p}{q}$ with $p, q \in \mathbb{N^*}$. Clearly, we have $q>p$. We impose the function $\vartheta$ to be affine on ${[a-\delta, a+\delta]}$ $\vartheta|_{[a-\delta, a+\delta]}$ (for some $\delta>0$).
\par Given $r\in \mathbb{N^*}$ and set $m=(q-p)r$ and $n= pr$. The parameter $r$ is used to refine the mesh when increased. Set $a=x_{n+1}=x_{pr+1}$. The interval $\overline{\Omega_{01}}=[0,a]$ is then discretized with $n=pr$ interior grid points (excluding $0$ and $a$). The interval $\overline{\Omega_{02}}=[a,1]$ is discretized with $m=(q-p)r$ exterior grid  points (excluding $a$ and $1$). Let $\mathfrak{M}_0=(ih)_{1\le i\le n+m+1}$ with $h=\frac{1}{n+m+2}$ be uniform mesh of $\Omega_0$ and $\overline{\mathfrak{M}_0}$ be the associated dual mesh. We define a non-uniform mesh $\mathfrak{M}$ of $\Omega$ as image of $\mathfrak{M}_0$ by the map $\vartheta$, settings

\begin{eqnarray}\label{eq:non uniform mesh}&&x'_i=\vartheta(ih), \quad \forall i\in \left\{{0,...,n}\right\}\cup \left\{{n+2,...,n+m+2}\right\}\nonumber\\
&&x'_{n+1}:=a'=\vartheta(a).
\end{eqnarray}
\par The dual mesh $\overline{\mathfrak{M}}$ and the general notation are then those of the previous section.
\subsection{Main results}
With the notation we have introduced, a consistent finite-difference approximation of $\mathcal{A}u$ with homogeneous boundary condition is 
\begin{equation}\mathcal{A}^{\mathfrak{M}}u=-\bar{D}(c_dDu)\nonumber
\end{equation}
for $u\in \mathbb{C}^{\mathfrak{M}\cup \partial \mathfrak{M}}$ satisfying $u|_{\partial \Omega}=u^{\partial \mathfrak{M}}=0$. We have
\begin{equation}(\mathcal{A}^{\mathfrak{M}}u)_i=-\frac{c_d(x_{i+\frac{1}{2}})\frac{u_{i+1}-u_i}{h_{i+\frac{1}{2}}}-c_d(x_{i-\frac{1}{2}})\frac{u_{i}-u_{i-1}}{h_{i-\frac{1}{2}}}}{h_i}, \ \ i=1,..,n+m+1.\nonumber
\end{equation}

\par For a suitable weight function $\varphi$ (to be defined below), the announced semi-discrete Carleman estimate for the operator $\mathcal{P}^{\mathfrak{M}}=-\partial_t+\mathcal{A}^{\mathfrak{M}}$ with a discontinuous diffusion coefficient $c$, for the non-uniform meshes we consider, is of the form
\begin{eqnarray}&&\tau^{-1}\left\|{\theta^{-\frac{1}{2}}e^{\tau \theta \varphi}\partial_t u}\right\|^2_{L^2(Q)}+\tau\left\|{\theta^{\frac{1}{2}}e^{\tau \theta\varphi}Du}\right\|^2_{L^2(Q)}+\tau^3\left\|{\theta^{\frac{3}{2}}e^{\tau\theta\varphi}u}\right\|^2_{L^2(Q)}\nonumber\\
&&\quad \quad \quad \le C_{\lambda,\mathfrak{K}}\left( \left\|{e^{\tau\theta\varphi}P^{\mathfrak{M}}u}\right\|^2_{L^2(Q)}+\tau^3\left\|{\theta^{\frac{3}{2}}e^{\tau \theta\varphi}u}\right\|^2_{L^2( (0,T) \times \omega)} \right.\nonumber\\
&&\left. \quad \quad \quad\quad\quad\quad \quad +h^{-2} \left|{ e^{\tau\theta\varphi}u|_{t=0}}\right|^2_{L^2(\Omega)}+ h^{-2}\left|{e^{\tau\theta\varphi}u|_{t=T}}\right|^2_{L^2(\Omega)} \right),
\end{eqnarray}
for properly chosen functions $\theta=\theta(t)$ and $\varphi=\varphi(x)$, for all $\tau\ge \tau_0(T+T^2)$, $0<h\le h_0$ and $\tau h(\alpha T)^{-1}\le  \epsilon_0$, $0<\alpha<T$ and for all $u\in {C}^\infty(0,T;\mathbb{C}^{\mathfrak{M}})$ satisfying the discrete transmission conditions, where $\tau_0, h_0, \epsilon_0$ only depend on the data. We refer to Theorem ~\ref{theo:Carleman estinmate uniform} (uniform mesh) and Theorem~\ref{theo:carleman non uniform} (non uniform mesh) below for a precise result. The proof of this estimate will be first carried out for piecewise uniform meshes, and then adapted to the case of the non-uniform meshes we introduced in Section~\ref{subsec:Families of non-uniform meshes}.\\

\par From the semi-discrete Carleman estimate we obtain allows we deduce following weak observability estimate 
\begin{equation}\left|{q(0)}\right|_{L^2(\Omega)}\le C_{obs}\left\|{q}\right\|^2_{L^2( (0,T)\times \omega)}+e^{-\frac{C}{h}}\left|{q(T)}\right|^2_{L^2(\Omega)},\nonumber
\end{equation}
for any $q$ solution to the adjoint system 
\begin{equation}\partial_t q+\mathcal{A}^{\mathfrak{M}}q+aq=0,\quad \quad q|_{\partial \Omega}=0.\nonumber
\end{equation}
A precise statement is given in Section~\ref{sec:Controllability results}.
\par Moreover, from the weak observability estimate given above we obtain a controllability result for the linear operator $P^{\mathfrak{M}}$. This result can be extended to classes of semi-linear equations
\begin{equation}\big(\partial_t +\mathcal{A}^{\mathfrak{M}}\big)y+{G}(y)=\textbf{1}_{\omega}v, \quad y\in (0,T) \quad\quad y|_{\partial \Omega}=0,\quad y(0)=y_0,\nonumber
\end{equation}
with $G(x)=xg(x)$, where $g\in L^{\infty}({\mathbb{R}})$ and
\begin{equation}\left|{g(x)}\right|\le K\ln^r(e+\left|{x}\right|), \quad x\in \mathbb{R},\quad\quad\textmd{with} \quad\quad 0\le r<\frac{3}{2}.\nonumber
\end{equation}
We shall state controllability results with a control that satisfies
\begin{equation}
\left\|v \right\| _{L^2(Q)}\le C\left|{y_0} \right|.\nonumber  
\end{equation}
Thanks to one space dimension the size of the control function is uniform with respect to the discretization parameter $h$.
\subsection{Sketch of proof of the Carleman estimate}
\par The main idea of the proof lays in the combination of the derivation of a discrete Carleman estimate as in \cite{BHL10a, BL12} and techniques used in \cite{BDL07} to achieve such estimates for operators with discontinuous coefficients in the one-dimensional case. 
\par We set $v=e^{-s\varphi}u$ yielding $e^{s\varphi}Pe^{-s\varphi}v=e^{s\varphi}f_1$ in $Q'_0$ if $Pu=f_1$
\par We obtain $g=Av+Bv$ in $Q'_0$, with $A$ and $iB$ 'essentially' selfadjoint.
\par We write $\left\| {g}\right\|^2_{L^2} =\left\| {Av}\right\|^2_{L^2}+\left\|{Bv} \right\|^2_{L^2}+2(Av,Bv)_{L^2} $ and the main part of the proof is dedicated to computing the inner product $(Av,Bv)_{L^2(Q'_0)}$, involving (discrete) integration by parts. 
\par We proceed with these computations separately in each domain $\Omega_{01}$, $\Omega_{02}$. As in \cite{BL12} we obtain terms involving boundary points $x=0$ and $ x=1$ such as $v(0)$,$ v(1)$,$ \partial_t v(0)$,$ \partial_t v(1)$, $(Dv)_{n+m+\frac{1}{2}}, (Dv)_{n+m+\frac{3}{2}}$. In our case we obtain additional terms involving the jump point $a$ such as $v(a), \partial_t v(a)$, $ \tilde{v}_{n+\frac{1}{2}}, \tilde{v}_{n+\frac{3}{2}}$, $(Dv)_{n+\frac{1}{2}}, (Dv)_{n+\frac{3}{2}}$. Main difficulties of our work come from dealing with these new terms. To reduce the number of terms to control, we find relations among connecting these various values at jump point allowing to focus our computations on terms only involving $v(a), \partial_t v(a)$ and $(Dv)_{n+\frac{1}{2}}$. Those relations are stated in \textbf{Lemma ~\ref{lem:transmission condition}}. In the limit $h\to 0$ they give back the transmission conditions for the function $v=e^{-s\varphi}u$ used crucial way in \cite{BDL07}. The idea of this technique comes from a similar technique shown in continuos case by  \cite{BDL07}.
\par The discrete setting could allow computation on the whole $\Omega$. Yet such computation would yield constant that would depend on discrete derivatives of the diffusions coefficient, yielding non-uniformity with respect to the discretization parameter $h$. This explains why we resort to working on both $\Omega_0$ and $\Omega_1$ separately and deal with the interface terms that appear. As in  \cite{BDL07} the weight function is chosen to obtain positive contributions for these terms.
\par Sketch of proof Theorem 
\begin{enumerate}
\item We compute the inner product $(Av,Bv)$ in a series of terms and collect them together in an estimate (see \textbf{Lemma ~\ref{lem: estimate I11}--Lemma ~\ref{lem:estimate I33}}). In that estimate, we need to tackle two parts: volume integrals, integrals involving boundary points and the jump point. Volume integrals and boundary terms are dealt with similar to \cite{BL12}. Terms at the jump point require special case.
\item Treatment of terms the jump point
\begin{itemize}
\item Terms at jump point involving $\partial_t v$ : when treating the term $Y_{13}$ we obtain a positive integral of $(\partial_t v(a))^2$ in the LHS of the estimate as shown in \textbf{Lemma ~\ref{lem:estimate Y13}}. We keep this term in the LHS of the estimate.
\item Other terms: We collect together the terms at the jump point that already exist in the continuous case. As in \cite{BDL07} we obtain a quadratic form because of the choice of the weight function (jump of its slope). This allows us to obtain positive two integrals involving $v^2(a)$, $(Dv)^2_{n+\frac{1}{2}}$ in the LHS of our estimate (see \textbf{Lemma ~\ref{lem:estimate Y11 plus Y21 again}}).
\item The remaining terms at the jump point are placed in the RHS of estimate. After that, we apply Young's inequality to them (as shown in \textbf{Lemma ~\ref{lem:Young ineq}}) and they then can be absorded by the positive integrals involving $v^2(a)$, $(Dv)^2_{n+\frac{1}{2}}$, $(\partial_t v(a))^2$ in the LHS of estimate as described above.  
\end{itemize}
\end{enumerate}
\subsection{Outline}In section ~\ref{sec: weight functions}, we construct the weight functions to be used in the Carleman estimate. In section~\ref{sec: prelim results} we have gathered some preliminary discrete calculus results and we present how transmission conditions can be expressed in the discretization scheme. Section~\ref{sec: Carleman estimate uniform mesh} is devoted to the proof the semi-discrete parabolic Carleman estimate in the case of a discontinuous diffusion cofficient for piecewise uniform meshes in the one-dimensional case. To ease the reading, a large number of proofs of intermediate estimates have been provided in Appendix A. This result is then extended to non-uniform meshes in Section~\ref{sec:non uniform meshes}. Finally, in Section~\ref{sec:Controllability results}, as consequences of the Carleman estimate, we present the weak observability estimate and associated some controllability results.

\section{Weight functions}
\label{sec: weight functions}
\par We shall first introduce a particular type of weight functions, which are constructed through the following lemma.
\par We enlarge the open intervals $\Omega_1, \Omega_2$ to large open sets $\tilde{\Omega}_1, \tilde{\Omega}_2$.
\begin{Lemma}\label{lem:weight function}Let $\tilde{\Omega}_1$, $\tilde{\Omega}_2$ be a smooth open and connected neighborhoods of intervals $\overline{\Omega}_1$, $\overline{\Omega}_2$ of $\mathbb{R}$ and let $\omega\subset \Omega_2$ be a non-empty open set. Then, there exists a function $\psi\in C({\bar{\Omega}})$ such that
\begin{center}$\psi(x)=\begin{cases}
{\psi_1}~\textrm{in}~ \overline{\Omega}_1 , &\\
{\psi_2}~\textrm{in}~ \overline{\Omega}_2,&
\end{cases}$
\end{center}
with ${\psi_i}\in C^{\infty}(\overline{\tilde{\Omega_i}}), i=1, 2$, ${\psi}>0~\textrm{in}~\Omega, {\psi}=0~\textrm{on}~\Gamma, {\psi '_2}\neq 0 ~\textrm{in}~\overline{\Omega}_2\setminus \omega, \ {\psi'_1}\neq 0 ~\textrm{in}~\overline{\Omega}_1$ and the function ${\psi}$ satisfies the following trace properties, for some $\alpha_0>0$,
\begin{equation}(\textbf{A}u,u)\ge \alpha_0\left|{u}\right|^2 \ \  u\in \mathbb{R}^2,\nonumber
\end{equation}
with the matrix $\textbf{A}$ defined by

\begin{displaymath}
\textbf{A} =
\left( \begin{array}{cc}
a_{11} & a_{12}\\
a_{21} & a_{22} \\
\end{array} \right),
\end{displaymath}
\ with 
\begin{eqnarray}&&a_{11}=[\psi'\star]_{a'},\nonumber\\
&&a_{22}=[c\psi'\star]^2_{a'}(\psi')(a'^+)+[c^2(\psi')^3\star]_{a'},\nonumber\\
&&a_{12}=a_{21}=[c\psi' \star]_{a'}(\psi')(a'^+),\nonumber
\end{eqnarray}
(see the notation ~\eqref{eq:notation1}~-~~\eqref{eq:notation3} introduced in Section~\ref{subsec:discrete setting}).
\end{Lemma}

\begin{Remark}\label{rek:observation} Here we choose a weight function that yields an observation in the region $\omega\subset \Omega_2$ in the Carleman estimate of Section~\ref{sec: Carleman estimate uniform mesh}. This choice is of course arbitrary.\end{Remark}

\par \begin{proof} We refer to Lemma 1.1 in \cite{BDL07} for a similar proof. 
\end{proof}

%
%
%

\par Choosing a function $\psi$, as in the previous lemma, for $\lambda>0$ and $K> \left\|{\psi}\right\|_{\infty}$, we define the following weight functions
\begin{eqnarray}\label{eq:weight functions}&&\varphi(x)=e^{\lambda\psi(x)}-e^{\lambda K}<0, \ \ \phi(x)=e^{\lambda\psi(x)}, \\
  &&r(t,x)=e^{s(t)\varphi(x)}, \ \ \rho(t,x)=(r(t,x))^{-1},\nonumber
\end{eqnarray}
with
\begin{equation}s(t)=\tau\theta(t), \ \ \tau>0, \ \  \theta(t)=( (t+\alpha)(T+\alpha-t))^{-1},\nonumber
\end{equation}
\ for $0<\alpha<T$.

\par We have 
\begin{equation}\label{eq:theta} \underset{[0,T]}{\max}\theta=\theta(0)=\theta(T)=\alpha^{-1}(T+\alpha)^{-1},\end{equation}
\ and $\underset{[0,T]}{\min}\theta\ge T^{-2}$. We note that
\begin{equation}\label{eq:derivative of theta}\partial_t \theta=(2t-T)\theta^2.\end{equation}
\par For the Carleman estimate and the observation/control results we choose here to treat the case of an distributed-observation in $\omega\subset \Omega$. The weight function is of the form $r=e^{s\varphi}$ with $\varphi=e^{\lambda\psi}$, with $\psi$ fulfilling the following assumption. Construction of such a weight function is classical (see e.g \cite{FI96}).
\begin{Assumption}Let $\omega\subset \Omega$ be an open set. Let $\tilde{\Omega}$ be a smooth open and connected neighborhood of $\bar{\Omega}$ in $R$. The function $\psi=\psi(x)$ is in $\textbf{C}^p\big({\overline{\tilde{\Omega}}, R}\big)$, $p$ sufficiently large, and satisfies, for some $c>0$,
\begin{eqnarray}&&\psi>0 \quad~\textrm{in}\quad \tilde{\Omega}, \quad\quad\quad\left|{\triangledown \psi} \right|\ge c \quad~\textrm{in}\quad\tilde{\Omega}\backslash \omega_0,\nonumber\\
&&\partial_{n}\psi(x)\le-c<0, \quad\quad \partial_x^2\psi(x)\le 0 \quad~\textrm{in}\quad V_{\partial\Omega}.\nonumber
\end{eqnarray}
where $V_{\partial\Omega}$ is a sufficiently small neighborhood of $\partial\Omega$ in $\tilde{\Omega}$, in which the outward unit normal $n$ to $\Omega$ is extended from $\partial\Omega$.
\end{Assumption}
\section{Some preliminary discrete calculus results for uniform meshes}
\label{sec: prelim results}

\par Here, to prepare for Section~\ref{sec: Carleman estimate uniform mesh}, we only consider constant-step discretizations, i.e., $h_{i+\frac{1}{2}}=h$, $i=0,\ldots,n+m+1$. 
\par We use here the following notation: $\Omega_0=(0,1)$, $\Omega_{01}=\left({0,a}\right)$, $\Omega_{02}=\left({a,1}\right)$, $\Omega'_{0}=\Omega_{01}\cup \Omega_{02}$, $Q_0=\left({0,T}\right)\times \Omega_0$, $Q'_0=\left({0,T}\right)\times \Omega'_0$, $Q_{0i}=\left({0,T}\right)\times \Omega_{0i}$ with $i=1,2$ and $\partial\Omega_0=\left\{{0,1}\right\}$.
\par This section aims to provide calculus rules for discrete operators such as $D_i, \bar{D}_i$ and also to provide estimates for the successive applications of  such operators on the weight functions. To avoid cumbersome notation we introduce the following continuous difference and averaging operators on continuous functions. For a function $f$ defined on $\Omega_0$ we set
\begin{align}\tau^+f(x):=f(x+{h}/2), &\quad \quad \tau^-f(x):=f(x-{h}/2),\nonumber\\
\mathtt{D}f(x):=(\tau^+-\tau^-)f(x)/{h}, &\quad \quad \hat{f}(x)=(\tau^++\tau^-)f(x)/2.\nonumber
\end{align}
\begin{Remark} To iterate averaging symbols we shall sometimes write $Af=\hat{f}$, and thus $A^2f=\hat{\hat{f}}.$
\end{Remark}
\subsection{Discrete calculus formulae} We present calculus results for finite-difference operators that were defined in the introductory section. Proofs can be found in Appendix of \cite{BHL10a} in the one-dimension case.

\begin{Lemma}
  \label{lemma: discrete leibnitz}
  Let the functions $f_1$ and $f_2$ be continuously defined in a neighborhood of $\bar{\Omega}$. We have:
\begin{equation}\mathtt{D}(f_1f_2)=\mathtt{D}(f_1)\hat{f_2}+\hat{f_1}\mathtt{D}(f_2).\nonumber
\end{equation}
\par Note that the immediate translation of the proposition to discrete functions $f_1, f_2\in \mathbb{C}^{\mathfrak{M}}$  and  $g_1, g_2\in \mathbb{C}^{\overline{\mathfrak{M}}}$ is
\begin{equation}D(f_1f_2)=D(f_1)\tilde{f_2}+\tilde{f_1}D(f_2),\ \bar{D}(g_1g_2)=\bar{D}(g_1)\bar{g_2}+\bar{g_1}\bar{D}(g_2).\nonumber
\end{equation}
\end{Lemma}

\begin{Lemma}\label{lem:average of two function} Let the functions $f_1$ and $f_2$ be continuously defined in a neighborhood of $\bar{\Omega}$. We have:
\begin{equation}\widehat{f_1f_2}=\hat{f_1}\hat{f_2}+\frac{{h}^2}{4}\mathtt{D}(f_1)\mathtt{D}(f_2).\nonumber
\end{equation}
\par Note that the immediate translation of the proposition to discrete functions $f_1, f_2\in \mathbb{C}^{\mathfrak{M}}$  and  $g_1, g_2\in \mathbb{C}^{\overline{\mathfrak{M}}}$ is
\begin{equation}\wtilde{f_1f_2}=\tilde{f_1}\tilde{f_2}+\frac{{h}^2}{4}D(f_1)D(f_2), \ \overline{g_1g_2}=\bar{g_1}\bar{g_2}+\frac{{h}^2}{4}\bar{D}(g_1)\bar{D}(g_2).\nonumber
\end{equation}
\end{Lemma}

\par Some of the following properties can be extended in such a manner to discrete functions. We shall not always write it explicitly.
\par Averaging a function twice gives the following formula.

\begin{Lemma}\label{lem:average double}Let the function f be continuously defined over $\mathbb{R}$. We then have
\begin{equation}A^2f:=\hat{\hat{f}}=f+\frac{{h}^2}{4}\mathtt{D}\mathtt{D}f .\nonumber
\end{equation}
\end{Lemma}

\par The following proposition covers discrete integrations by parts and related formula.
\begin{Proposition}\label{prop:IVP} Let $f\in \mathbb{C}^{\mathfrak{M}\cup \partial \mathfrak{M}}$ and $g\in \mathbb{C}^{\overline{{\mathfrak{M}}}}$. We have the following formulae:
\begin{eqnarray}
&&\int_{\Omega_0}{f(\bar{D}g)}=-\int_{\Omega_0}{(Df)g}+f_{n+m+2}g_{n+m+\frac{3}{2}}-f_0g_{\frac{1}{2}},\nonumber\\
&&\int_{\Omega_0}{f\bar{g}}=\int_{\Omega_0}{\tilde{f}g}-\frac{h}{2}f_{n+m+2}g_{n+m+\frac{3}{2}}-\frac{h}{2}f_0g_{\frac{1}{2}}.\nonumber
\end{eqnarray}
\end{Proposition}

\begin{Lemma} \label{lem:traslate function}Let f be a smooth function defined in a neighborhood of $\bar{\Omega}$. We have
\begin{eqnarray}&&\tau^\pm f=f\pm \frac{{h}}{2}\int_0^T{\partial_x f(.\pm\sigma {h}/2)d\sigma},\nonumber\\
&& A^j f=f+C_j{ h}^2\int_{-1}^1{(1-\left|{\sigma}\right|)\partial^2_x f(.+l_j\sigma {h})d\sigma}, \nonumber\\
&& D^j f=\partial^j_x f+C'_j {h^2}\int_{-1}^1{(1-\left|{\sigma}\right|)^{j+1}\partial_x^{j+2}f(.+l_j\sigma {h})d\sigma},\ j=1,2,\ l_1=\frac{1}{2}, \ l_2=1.\nonumber
\end{eqnarray}
\end{Lemma}

\subsection{Calculus results related to the weight functions} We now present some technical lemmata related to discrete operators performed on the Carleman weight functions that is of the form $e^{s\varphi}$, $\varphi=e^{\lambda\psi}-e^{\lambda K}$, where $\psi$ satisfies the properties listed in Section~\ref{sec: weight functions} in the domain $\Omega_0$. For concision, we set $r(t,x)=e^{s(t)\varphi(x)}$ and $\rho=r^{-1}$, with $s(t)=\tau\theta(t)$. From Section~\ref{sec: weight functions}, we have $\psi_{|_{\Omega_{01}}}=\psi_{1|_{\Omega_{01}}}$, $\psi_{|_{\Omega_{01}}}=\psi_{2|_{\Omega_{01}}}$ where $\psi_{i}\in \mathbb{C}^2(\overline{\tilde{\Omega}}_{0i})$. Then $\rho=e^{-s\varphi}$ can be replaced by
\begin{eqnarray}&&\rho_{1}=e^{-s\varphi_{1}}\quad ~\textmd{with}~\quad \varphi_{1}=e^{\lambda\psi_{1}}-e^{\lambda K}\quad ~\textmd{in domain}~\quad \Omega_{01}\nonumber\\
&&\rho_{1}=e^{-s\varphi_{2}}\quad ~\textmd{with}~\quad \varphi_{2}=e^{\lambda\psi_{2}}-e^{\lambda K}\quad ~\textmd{in domain }~\quad\Omega_{02}\nonumber
\end{eqnarray}
And $r=\rho^{-1}$ is also replaced by
\begin{eqnarray}&&r_{1}=\rho^{-1}_{1}~\textmd{in domain}~\quad \Omega_{01}\nonumber\\
&&r_{2}=\rho^{-1}_{2}~\textmd{in domain}~\quad \Omega_{02}\nonumber
\end{eqnarray} 
 The positive parameters $\tau$ and $h$ will be large and small respectively and we are particularly interested in the dependence on $\tau, h$ and $\lambda$ in the following basic estimates in each domain $\Omega_{01}$, $\Omega_{02}$.
\par We assume $\tau\ge 1$ and $\lambda\ge 1$.

\begin{Lemma} \label{lem:derivative wrt x}Let $\alpha,\ \beta\in \mathbb{N}$, i=1,2. We have
\begin{eqnarray}\partial^{\beta}_x(r_{i}\partial^{\alpha}_x\rho_{i})&=&{\alpha}^{\beta}(-s\phi_{i})^{\alpha}\lambda^{\alpha+\beta}(\bigtriangledown\psi_{i})^{\alpha+\beta}\nonumber\\
&+&{\alpha}{\beta}(s\phi_{i})^{\alpha}\lambda^{\alpha+\beta-1}\mathscr{O}_{\lambda}(1)+s^{\alpha-1}{\alpha}({\alpha}-1)\mathscr{O}_{\lambda}(1)=\mathscr{O}_{\lambda}(s^{\alpha}).\nonumber
\end{eqnarray}
\par Let $\sigma\in [-1,1]$, we have 
\begin{equation}\partial^{\beta}_x\big(r_{i}(t,.)(\partial^{\alpha}_x\rho_{i})(t,.+\sigma h)\big)=\mathscr{O}_{\lambda}(s^{\alpha}(1+(sh)^{\beta}))e^{\mathscr{O}_{\lambda}(sh)}.\nonumber
\end{equation}
\par Provided $0<\tau h(max_{[0,T]}\theta)\le \mathfrak{K}$ we have $\partial^{\beta}_x\big(r_{i}(t,.)(\partial^{\alpha}\rho_{i})(t,.+\sigma h)\big)=\mathcal{O_{\lambda, \mathfrak{K}}}(s^{\left|{\alpha}\right|})$. The same expressions hold with r and $\rho$ interchanged and with s changed into -s.
\end{Lemma}A proof is given in \cite[proof of Lemma 3.7]{BHL10a} in the time independent case. Additionally, we provide a result below to the time-dependent case whose proof is refered to \cite[proof of Lemma 2.8]{BL12}. Note that the condition $0<\tau h(max_{[0,T]}\theta)\le \mathfrak{K}$ implies that $s(t)h\le \mathfrak{K}$ for all $t\in[0,T]$.

\begin{Lemma}\label{lem:derivative wrt t}Let $\alpha\in \mathbb{\mathbb{N}}$, i=1,2. We have
\begin{equation}\partial_t(r_{i}\partial^{\alpha}_x\rho_{i})=s^{\alpha}T\theta \mathcal{O}_{\lambda}(1).\nonumber
\end{equation}
\end{Lemma}

\par With Leibniz formula we have the following estimates

\begin{Corollary}\label{cor:multi derivative} Let $\alpha,\beta, \delta\in \mathbb{N}$, i=1,2. We have
\begin{eqnarray}\partial^{\delta}_x(r_{i}^2(\partial^{\alpha}_x\rho_{i})\partial^{\beta}_x\rho_{i})&=&(\alpha+\beta)^{\delta}(-s\phi_{i})^{\alpha+\beta}\lambda^{\alpha+\beta+\delta}(\bigtriangledown\psi_{i})^{\alpha+\beta+\delta}\nonumber\\
&+&{\delta}(\alpha+\beta)(s\phi_{i})^{\alpha+\beta}\lambda^{\alpha+\beta+\delta-1}\mathcal{O}(1)\nonumber\\
&+&s^{\alpha+\beta-1}({\alpha}(\alpha-1)+{\beta}(\beta-1))\mathcal{O}_{\lambda}(1)=\mathcal{O}_{\lambda}(s^{\alpha+\beta}).\nonumber
\end{eqnarray}
\end{Corollary}
\par The proofs of the following properties can be found in Appendix A of \cite{BHL10a}.

\begin{Proposition}\label{prop:property 1} Let $\alpha\in \mathbb{N}$, i=1,2. Provided $0<\tau h(max_{[0,T]}\theta)\le \mathfrak{K}$, we have
\begin{eqnarray}&&r_{i}\tau^{\pm}\partial^{\alpha}_x\rho_{i}=r_{i}\partial^{\alpha}_x\rho_{i}+s^{\alpha}\mathcal{O_{\lambda, \mathfrak{K}}}(sh)=s^{\alpha}\mathcal{O_{\lambda, \mathfrak{K}}}(1),\nonumber\\
&&r_{i}A^k\partial^{\alpha}_x\rho_{i}=r_{i}\partial^{\alpha}_x\rho_{i}+s^{\alpha}\mathcal{O_{\lambda, \mathfrak{K}}}(sh)^2=s^{\alpha}\mathcal{O_{\lambda, \mathfrak{K}}}(1),\ k=0,1,2, \nonumber\\
&&r_{i}A^kD\rho_{i}=r_{i}\partial_x \rho_{i}+s \mathcal{O_{\lambda, \mathfrak{K}}}(sh)^2=s\mathcal{O_{\lambda, \mathfrak{K}}}(1), \ k=0,1,\nonumber\\
&&r_{i}D^2\rho_{i}=r_{i}\partial^2_x\rho_{i}+s^2\mathcal{O_{\lambda, \mathfrak{K}}}(sh)^2=s^2\mathcal{O_{\lambda, \mathfrak{K}}}(1). \nonumber
\end{eqnarray}
The same estimates hold with $\rho_i$ and $r_i$ interchanged.
\end{Proposition}

\begin{Lemma}\label{lem:property 2}Let $\alpha, \beta\in \mathbb{N}$ and $k=1,2; j=1,2; i=1,2$. Provided $0<\tau h(\max_{[0,T]}\theta)\le \mathfrak{K}$, we have
\begin{eqnarray}&&D^{k}(\partial^{\beta}_x(r_{i}\partial^{\alpha}_x\rho_{i}))=\partial^{k+\beta}_x(r_{i}\partial^{\alpha}_x\rho_{i})+h^2\mathcal{O_{\lambda, \mathfrak{K}}}(s^{\alpha}), \nonumber\\
&&A^j\partial^{\beta}_x(r_{i}\partial^{\alpha}_x\rho_{i})=\partial^{\beta}(r_{i}\partial^{\alpha}_x\rho_{i})+h^2\mathcal{O_{\lambda, \mathfrak{K}}}(s^{\alpha}).\nonumber
\end{eqnarray}
\par Let $\sigma \in[-1,1]$, we have $D^k\partial^{\beta}(r_{i}(t,.)\partial^{\alpha}\rho_{i}(t,.+\sigma h))=\mathcal{O_{\lambda, \mathfrak{K}}}(s^{\left|{\alpha}\right|})$. The same estimates hold with $r_i$ and $\rho_i$ interchanged.
\end{Lemma}

\begin{Lemma}\label{lem:property 3} Let $\alpha$, $\beta, \delta\in \mathbb{N}$  and $k=1,2; j=1,2; i=1,2$. Provided  $0<\tau h(max_{[0,T]}\theta)\le \mathfrak{K}$, we have
\begin{equation}A^j\partial^{\delta}_x(r_{i}^2(\partial^{\alpha}_x\rho_{i})\partial^{\beta}_x\rho_{i})=\partial^{\delta}_x(r_{i}^2(\partial^{\alpha}_x\rho_i)\partial^{\beta}_x\rho_{i})+h^2\mathcal{O_{\lambda, \mathfrak{K}}}(s^{\alpha+\beta})=\mathcal{O_{\lambda, \mathfrak{K}}}(s^{\alpha+\beta}),\nonumber
\end{equation}
\begin{eqnarray} D^k\partial^{\delta}_x(r_{i}^2(\partial^{\alpha}_x\rho_{i})\partial^{\beta}_x\rho_{i})&=&\partial^{k+\delta}_x(r_{i}^2(\partial^{\alpha+\beta}_x\rho_{i})+h^2\mathcal{O_{\lambda, \mathfrak{K}}}(s^{\alpha+\beta})=\mathcal{O_{\lambda, \mathfrak{K}}}(s^{\alpha+\beta}).\nonumber
\end{eqnarray}
\par Let $\sigma,\sigma' \in[-1,1]$. We have
\begin{eqnarray}&&A^j\partial^{\delta}(r_{i}(t,.)^2(\partial^{\alpha}\rho_{i}(t,.+\sigma h))\partial^{\beta}\rho_{i}(y,.+\sigma' h))=\mathcal{O_{\lambda, \mathfrak{K}}}(s^{\alpha+\beta}),\nonumber\\
&&D^k\partial^{\delta}(r_{i}(t,.)^2(\partial^{\alpha}\rho_{i}(t,.+\sigma h))\partial^{\beta}\rho_{i}(t,. +\sigma' h))=\mathcal{O_{\lambda, \mathfrak{K}}}(s^{\alpha+\beta}). \nonumber
\end{eqnarray}
The same estimates hold with $r_i$ and $\rho_i$ interchanged.
\end{Lemma}

\begin{Proposition}\label{prop:property 4} Let $\alpha\in \mathbb{N}$ and $k=0,1,2; j=0,1,2; i=1,2$. Provided $0<sh\le \mathfrak{K}$, we have
\begin{eqnarray}&&D^kA^j\partial^{\alpha}_x(r_{i}\widehat{D\rho_{i}})=\partial^{k+\alpha}_x(r_{i} \partial_x \rho_{i})+s\mathcal{O_{\lambda, \mathfrak{K}}}(sh)^2=s\mathcal{O_{\lambda, \mathfrak{K}}}(1),\nonumber\\
&&D^k(r_{i}D^2\rho_{i})=\partial^k_x(r_{i}\partial^2\rho_{i})+s^2\mathcal{O_{\lambda, \mathfrak{K}}}(sh)^2=s^2\mathcal{O_{\lambda, \mathfrak{K}}}(1), \nonumber\\
&&r_{i}A^2\rho_{i}=1+\mathcal{O_{\lambda, \mathfrak{K}}}(sh)^2, \ D^k(r_{i}A^2\rho_{i})=\mathcal{O_{\lambda, \mathfrak{K}}}(sh)^2.\nonumber
\end{eqnarray} 
The same estimates hold with $r_{i}$ and $\rho_{i}$ interchanged.
\end{Proposition}

\begin{Proposition}\label{prop:property 5} Provided $0<\tau h(max_{[0,T]}\theta)\le \mathfrak{K}$ and $\sigma$ is bounded, we have
\begin{eqnarray}&&\partial_t(r_{i}(.,x)(\partial^{\alpha}\rho_{i})(.,x+\sigma h))=Ts^{\alpha}\theta(t)\mathcal{O_{\lambda, \mathfrak{K}}}(1),\nonumber\\
&& \partial_t(r_{i}A^2\rho_{i})=T(sh)^2\theta(t)\mathcal{O_{\lambda, \mathfrak{K}}}(1),\nonumber\\
&&\partial_t(r_iD^2\rho_i)=Ts^2\theta(t)\mathcal{O_{\lambda, \mathfrak{K}}}(1).\nonumber
\end{eqnarray} 
The same estimates hold with $r_{i}$ and $\rho_i$ interchanged.
\end{Proposition}

\begin{Proposition}\label{prop:property 6} Let $\alpha$, $\beta\in \mathbb{N}$ and $k=0,1,2; j=0,1,2; i=1,2$, provided $0<sh\le \mathfrak{K}$, we have
\begin{eqnarray}A^jD^k\partial^{\beta}(r_{i}^2(\partial^{\alpha})\widehat{D\rho_{i}})&=&\partial^{k+\beta}_x(r_{i}^2(\partial^{\alpha}\rho)\partial\rho_{i})+s^{\alpha+1}\mathcal{O_{\lambda, \mathfrak{K}}}(sh)^2=s^{\alpha+1}\mathcal{O_{\lambda, \mathfrak{K}}}(1),\nonumber
\end{eqnarray}
\begin{eqnarray}A^jD^k\partial^{\beta}(r_{i}^2(\partial^{\alpha})A^2\rho_{i})&=&\partial^{k+\beta}_x(r_{i}(\partial^{\alpha}\rho_{i}))+s^{\alpha}\mathcal{O_{\lambda, \mathfrak{K}}}((sh)^2)=s^{\alpha}\mathcal{O_{\lambda, \mathfrak{K}}}(1),\nonumber
\end{eqnarray}
\begin{eqnarray}A^jD^k\partial^{\beta}(r_{i}^2(\partial^{\alpha})D^2\rho_{i})&=&\partial^{k+\beta}_x(r_{i}^2(\partial^{\alpha}\rho)\partial^2\rho_{i})+s^{\alpha+2}\mathcal{O_{\lambda, \mathfrak{K}}}(sh)^2=s^{\alpha+2}\mathcal{O_{\lambda, \mathfrak{K}}}(1),\nonumber
\end{eqnarray}
\ and we have
\begin{eqnarray}&&A^jD^k\partial^{\alpha}(r_{i}^2\widehat{D\rho_{i}}D^2\rho_{0i})=\partial^{k+\alpha}_x(r_{i}^2(\partial\rho_{i})\partial^2\rho_{i})+s^3\mathcal{O_{\lambda, \mathfrak{K}}}(sh)^2=s^3\mathcal{O_{\lambda, \mathfrak{K}}}(1),\nonumber\\
&&A^jD^k\partial^{\alpha}(r_{i}^2\widehat{D\rho_{i}}A^2\rho_{i})=\partial^{k+\alpha}_x(r_{i}\partial\rho_{i})+s\mathcal{O_{\lambda, \mathfrak{K}}}(sh)^2=s\mathcal{O_{\lambda, \mathfrak{K}}}(1).\nonumber
\end{eqnarray}
\end{Proposition}
\subsection{Transmission conditions}

\par We consider here discrete version of the transmission conditions (TC) at the point $a$. For $u\in \mathbb{C}^\mathfrak{M}$ we set $f:=\bar{D}(c_dDu)$ we then have
\begin{center}$\begin{cases}
u(a^-)=u(a^+)=u_{n+1} ,&\\
(c_dDu)_{n+\frac{3}{2}}-(c_dDu)_{n+\frac{1}{2}}=hf_{n+1}&.
\end{cases}$
\end{center}
\begin{Remark} These transmission conditions provide the continuity for $u$ and the discrete flux at the singular point of coefficient up to a consistent factor.
\end{Remark}

\par From these conditions, we obtain the following lemma whose proof is given in Appendix A
\begin{Lemma}\label{lem:transmission condition}For the parameter $\lambda$ chosen sufficiently large and $sh$ sufficiently small and with $u=\rho v$ we have
\small{
\begin{equation}\label{eq:transmission eqs1}[\star c_dDv]_a=(c_dDv)_{n+\frac{3}{2}}-(c_dDv)_{n+\frac{1}{2}}=J_1v_{n+1}+J_2(c_dDv)_{n+\frac{1}{2}}+J_3h(rf)_{n+1}
\end{equation}
}
where 
\begin{eqnarray}&&J_1=\big(1+\mathcal{O_{\lambda, \mathfrak{K}}}(sh)\big)\lambda s[\star c\phi \psi']_a +s\mathcal{O_{\lambda, \mathfrak{K}}}(sh),\nonumber\\
&&J_2=\mathcal{O_{\lambda, \mathfrak{K}}}(sh),\quad\quad\quad\quad\quad\quad\quad\quad\quad\quad\quad\quad\quad\quad\quad J_3=\big(1+\mathcal{O_{\lambda, \mathfrak{K}}}(sh)\big).\nonumber
\end{eqnarray}
\par Furthermore, we have
\begin{eqnarray}&&\partial_t J_1=sT\theta(t)\mathcal{O_{\lambda, \mathfrak{K}}}(sh),\nonumber\\
&&\partial_t J_2=T\theta(t)\mathcal{O_{\lambda, \mathfrak{K}}}(sh),\quad\quad\quad\quad\quad\quad\quad\quad\partial_t J_3=T\theta(t)\mathcal{O_{\lambda, \mathfrak{K}}}(sh).\nonumber 
\end{eqnarray}
\par For simplicity, ~\eqref{eq:transmission eqs1} can be written in form
\begin{equation}\label{eq:transmission eqs2}[\star c_dDv]_a=\lambda s[\star c\phi \psi']_a v_{n+1}+r_0,
\end{equation}
where $r_0=\lambda s\mathcal{O_{\lambda, \mathfrak{K}}}(sh)v_{n+1}+\mathcal{O_{\lambda, \mathfrak{K}}}(sh)(c_dDv)_{n+\frac{1}{2}}+h\Big(1+\mathcal{O_{\lambda, \mathfrak{K}}}(sh)\Big)(rf)_{n+1}.$
\end{Lemma}





\section{Carleman estimate for uniform meshes}  
\label{sec: Carleman estimate uniform mesh}
\par In this section, we prove a Carleman estimate in case of picewise uniform meshes, i.e, constant-step discretizations in each subinterval $(0, a)$ and $(a,1)$. The case of non-uniform meshes is treated in Section~\ref{sec:non uniform meshes}. 
\par We let $\omega_0\subset \Omega_{02}$ be a nonempty open subset. We set the operator $P^{\mathfrak{M}}$ to be $\mathcal{P}^{\mathfrak{M}}=-\partial_t+\mathcal{A}^{\mathfrak{M}}=-\partial_t-\bar{D}(c_dD)$, continuous in the variable $t\in (0,T)$ with $T>0$, and discrete in the variable $x\in \Omega_0$.

\par The Carleman weight function is of the form $r=e^{s\varphi}$ with $\varphi=e^{\lambda\psi}-e^{\lambda K}$ where $\psi$ satisfies the properties listed in Section~\ref{sec: weight functions} in the domain $\Omega_0$. Here, to treat the semi-discrete case, we use the enlarged neighborhoods $\tilde{\Omega}_{01}$, $\tilde{\Omega}_{02}$ of  $\Omega_{01}$, $\Omega_{02}$ as introduced in Lemma~\ref{lem:weight function}. This allows one to apply multiple discrete operators such as $D$ and $A$ on the weight functions. In particular, we take $\psi$ such that $\partial_x \psi\ge 0$ in $V_{0}$ and $\partial_x \psi\le 0$ in $V_{1}$ where $V_{0}$ and $V_{1}$ are neighborhoods of $0$ and $1$ respectively. This then yields on $\partial\Omega_0$
\begin{equation}\label{constrainted boundary}(r\overline{D\rho})_0\le 0, \ \  (r\overline{D\rho})_{n+m+2}\ge 0
\end{equation}
%
%
\begin{Theorem} \label{theo:Carleman estinmate uniform}Let $\omega_0 \subset \Omega_{02}$ be a non-empty open set and we set $f:=\bar{D}(c_dDu)$. For the parameter $\lambda>1$ sufficiently large, there exists C, $\tau_0\ge 1$, $h_0>0$, $\epsilon_0>0$, depending on $\omega_0$ so that the following estimate holds
\begin{eqnarray}\label{Carleman estinmate uniform }&&\tau^{-1}\left\|{\theta^{-\frac{1}{2}}e^{\tau \theta \varphi}\partial_t u}\right\|^2_{L^2(Q_0)}+\tau\left\|{\theta^{\frac{1}{2}}e^{\tau \theta\varphi}Du}\right\|^2_{L^2(Q_0)}+\tau^3\left\|{\theta^{\frac{3}{2}}e^{\tau\theta\varphi}u}\right\|^2_{L^2(Q_0)}\nonumber\\
&&\ \ \le C_{\lambda,\mathfrak{K}}\left( \left\|{e^{\tau\theta\varphi}P^{\mathfrak{M}}u}\right\|^2_{L^2(Q_0)}+\tau^3\left\|{\theta^{\frac{3}{2}}e^{\tau \theta\varphi}u}\right\|^2_{L^2( (0,T) \times \omega_0)} \right.\nonumber\\
&&\ \ \ \ \ \ \ \ \ \ \ \ \ \ \ \ \ \ \ \ \ \left. +h^{-2} \left|{ e^{\tau\theta\varphi}u|_{t=0}}\right|^2_{L^2(\Omega_0)}+ h^{-2}\left|{e^{\tau\theta\varphi}u|_{t=T}}\right|^2_{L^2(\Omega_0)} \right),
\end{eqnarray}
for all $\tau\ge \tau_0(T+T^2)$, $0<h\le h_0$ and $\tau h(\alpha T)^{-1}\le  \epsilon_0$ and for all $u\in C^\infty(0,T;\mathbb{C}^{\mathfrak{M}})$ satisfying $u|_{\partial\Omega_0}=0$.
\end{Theorem}

\begin{Remark} Observation was chosen in $\Omega_{02}$ here. This is an arbitrary choice (see Remark~\ref{rek:observation}).
\end{Remark}

\par {Proof}. We set $f_1:=-P^{\mathfrak{M}}=\partial_t u+\bar{D}(c_dDu)$ and $f=\bar{D}(c_dDu)$. At first, we shall work with the function $v=ru$, i.e.,\  $u=\rho v$, that satisfies 
\begin{equation}\label{eq:expanse}r \Big(\partial_t (\rho v)+ \bar{D}\big(c_d D(\rho v)\big)\Big)=rf _1 ~\textrm{in}~ Q'_0.
\end{equation}
\par We have 
\begin{equation} r\partial_t(\rho v) =\partial_t v+ r(\partial_t \rho)v=\partial_t v -\tau(\partial_t \theta)\varphi v.\nonumber\end{equation}
\par We write:  $g=Av+Bv$,
\par where $Av=A_1v+A_2v+A_3v$, $Bv=B_1v+B_2v+B_3v$ with
\begin{eqnarray} &&A_1v=r\overline{\tilde{\rho}} \bar{D}(c_dDv) ,\   A_2v=cr(\bar{D}D\rho)\overline{\tilde{v}}, \   A_3v=-\tau(\partial_t \theta)\varphi v,\nonumber\\
&&B_1v=2cr\overline{D\rho}\ \overline{Dv}, \ B_2v= -2sc\phi'' v, \   B_3v=\partial_t v,\nonumber\\
&&g=rf_1-\frac{{h}}{4} r\overline{D\rho}(\bar{D}c_{d})(\tau^+Dv-\tau^-Dv)-\frac{{h}^2}{4}(\bar{D}c_d)r(\bar{D}D\rho)\overline{Dv}\nonumber\\
&& \  \ \ \ \ \ \ \ \ \ - {h}\mathcal{O}(1)r\overline{D\rho}\ \overline{Dv}-\big(r(\bar{D}c_d)\ \overline{D\rho}+{h}\mathcal{O}(1)r(\bar{D}D\rho)\big)\overline{\tilde{v}}-2sc(\phi'')v,\nonumber\end{eqnarray}
as derived in \cite{BL12}.
\par Equation ~\eqref{eq:expanse} now reads $Av+Bv=g$ and we write
\begin{equation}\label{eq:4.2}\left\|{Av}\right\|^2_{L^2(Q'_0)}+\left\|{Bv}\right\|^2_{L^2(Q'_0)}+2(Av,Bv)_{L^2(Q'_0)}=\left\|{g}\right\|^2_{L^2(Q'_0)}.
\end{equation}
\par First we need an estimation of $\left\|{g}\right\|^2_{L^2(Q'_0)}$. The proof can be adapted from \cite{BHL10a}.

\begin{Lemma}For $\tau h(\max_{[0,T]}\theta)\le \mathfrak{K}$ we have
\begin{equation}\label{eq:4.3}\left\|{g}\right\|^2_{L^2(Q'_0)}\le C_{\lambda, \mathfrak{K}}(\left\|{rf_1}\right\|^2_{L^2(Q'_0)})+\left\|{sv}\right\|^2_{L^2(Q'_0)}+h^2\left\|{sDv}\right\|^2_{L^2(Q'_0)}.
\end{equation}
\end{Lemma}

\par Most of the remaining of the proof will be dedicated to computing the inner product $(Av,Bv)_{L^2(Q'_0)}$. Developing the inner-product $(Av,Bv)_{L^2(Q'_0)}$, we set $I_{ij}= (A_iv,B_jv)_{L^2(Q'_0)}$. The proofs of the following lemmata are provided in Appendix A.

\begin{Lemma}[Estimate of $I_{11}$]
  \label{lem: estimate I11}
For $\tau h(max_{[0,T]}\theta)\le \mathfrak{K}$ we have
\begin{equation}I_{11}\ge -\int_{Q'}{s\lambda^2}(c^2\phi(\psi')^2)_d (Dv)^2-X_{11}+Y_{11},\nonumber\end{equation}
where $X_{11}=\int_{Q'_0}{\nu_{11}(Dv)^2}$ with $\nu_{11}$ of the form $s\lambda\phi \mathcal{O}(1)+s\mathcal{O_{\lambda, \mathfrak{K}}}(sh)$
and
\begin{eqnarray} &&Y_{11}=Y_{11}^{(1)}+Y_{11}^{(2,1)}+Y_{11}^{(2,2)},\nonumber\\
&& Y_{11}^{(1)}=\int_0^T\big(1+\mathcal{O_{\lambda, \mathfrak{K}}}(sh)\big)(c\bar{c}_d)(1)(r\overline{D\rho})(1)(Dv)^2_{n+m+\frac{3}{2}}\nonumber\\
&&\quad \quad-\int_0^T \big(1+\mathcal{O_{\lambda, \mathfrak{K}}}(sh)\big)(c\bar{c}_d)(0)(r\overline{D\rho})(0)(Dv)^2_{\frac{1}{2}},\nonumber\\
&&Y_{11}^{(2,1)}=\int_0^T{s\lambda\phi(a)\bar{c}_d(a)\Big((c\psi')(a^+)(Dv)^2_{n+\frac{3}{2}}-(c\psi')(a^-)(Dv)^2_{n+\frac{1}{2}}\Big)},\nonumber\\
&&Y_{11}^{(2,2)}=\int_0^T{s\mathcal{O_{\lambda, \mathfrak{K}}}(sh)^2(Dv)^2_{n+\frac{1}{2}}}-\int_0^T{s\mathcal{O_{\lambda, \mathfrak{K}}}(sh)^2(Dv)^2_{n+\frac{3}{2}}}.\nonumber\end{eqnarray}


\end{Lemma}

\begin{Lemma}[Estimate of $I_{12}$]\label{lem: estimate I12} For $\tau h(max_{[0,T]}\theta)\le \mathfrak{K}$, the term $I_{12}$ is of the following form
\begin{equation}I_{12}= 2\int_{Q'_0}{s\lambda^2(c^2\phi(\psi')^2)_d(Dv)^2}-X_{12}+Y_{12},\nonumber
\end{equation}
with 
\small{
\begin{eqnarray} Y_{12}&=&\int_0^T{s \lambda^2 \phi(a)v(a)[c(\psi')^2\star c_dDv]_a}\nonumber\\
&+&\int_0^T{\delta_{12}v(a)(cDv)_{n+\frac{3}{2}}}+\int_0^T{\bar{\delta}_{12}v(a)(cDv)_{n+\frac{1}{2}}},\nonumber
\end{eqnarray}
}where $\delta_{12}, \bar{\delta}_{12}$ are of the form $s\big(\lambda\phi(a)\mathcal{O}(1)+\mathcal{O_{\lambda, \mathfrak{K}}}(sh)^2\big)$ and
\begin{equation}X_{12}=\int_{Q'_0}{\nu_{12}(Dv)^2}+\int_{Q'_0}{s\mathcal{O_{\lambda, \mathfrak{K}}}(1)\tilde{v}Dv},\nonumber
\end{equation}
\ where
\begin{equation}\nu_{12}=s\lambda \phi \mathcal{O}(1)+s\mathcal{O_{\lambda, \mathfrak{K}}}(h+(sh)^2).\nonumber\end{equation}
\end{Lemma}

\begin{Lemma}[Estimate of $I_{13}$]\label{lem:estimate I13} There exists $\epsilon_1(\lambda)>0$ such that, for $0<\tau h(max_{[0,T]}\theta)\le \epsilon_1(\lambda)$, the term $I_{13}$ can be estimated from below in following way:
\begin{eqnarray} I_{13} \ge -\int_{\Omega'_0}\mathcal{C_{\lambda, \mathfrak{K}}}(1)(Dv(T))^2-X_{13}+Y_{13}.\nonumber\end{eqnarray}
with
\begin{eqnarray}X_{13}=\int_{Q'_0}{\big(s(sh)+T(sh)^2\theta\big)\mathcal{O_{\lambda, \mathfrak{K}}}(1)(Dv)^2}+\int_{Q_0}{s^{-1}\mathcal{O_{\lambda, \mathfrak{K}}}(sh)(\partial_t v)^2},\nonumber\end{eqnarray}
\small{
\begin{equation} Y_{13}=-\int_0^T {r\bar{\tilde{\rho}}(a^+) \partial_t v(a)(c_dDv)_{n+\frac{3}{2}}}+\int_0^T {r\bar{\tilde{\rho}}(a^-) \partial_t v(a)(c_dDv)_{n+\frac{1}{2}}}.\nonumber
\end{equation}
}
\end{Lemma}

\begin{Lemma}[Estimate of $I_{21}$]\label{lem:estimate I21} For $\tau h(max_{[0,T]}\theta)\le \mathfrak{K}$, the term $I_{21}$ can be estimated as
\begin{equation} I_{21} \ge 3\int_{Q'_0}{\lambda^4s^3\phi^3c^2(\psi')^4v^2}-X_{21}+Y_{21},\nonumber\end{equation}
with
\begin{equation}X_{21}=\int_{Q'_0}{\mu_{21}v^2}+\int_{Q'_0}{\nu_{21}(Dv)^2},\nonumber
\end{equation}
where 
\begin{equation}\mu_{21}=(s\lambda\phi)^3\mathcal{O}(1)+s^2\mathcal{O_{\lambda, \mathfrak{K}}}(1)+s^3\mathcal{O_{\lambda, \mathfrak{K}}}(sh)^2, \ \ \nu_{21}=s\mathcal{O_{\lambda, \mathfrak{K}}}(sh)^2,\nonumber
\end{equation}
and

\begin{eqnarray}&&Y_{21}=Y_{21}^{(1,1)}+Y_{21}^{(1,21)}+Y_{21}^{(1,22)}+Y_{21}^{(2)},\nonumber\\
&&Y_{21}^{(1,1)}=\int_0^T{\mathcal{O_{\lambda, \mathfrak{K}}}(sh)^2(r\overline{D\rho})(1)(Dv)^2_{n+m+\frac{3}{2}}}+\int_0^T{\mathcal{O_{\lambda, \mathfrak{K}}}(sh)^2(r\overline{D\rho})(0)(Dv)^2_{\frac{1}{2}}},\nonumber\\
&&Y_{21}^{(1,21)}=\int_0^T{s^3\lambda^3\phi^3(a)[c^2(\psi')^3\star(\tilde{v})^2]_a},\nonumber\\
&&Y_{21}^{(1,22)}=\int_0^T\big(s^2\mathcal{O}_{\lambda}(1)+s^3\mathcal{O_{\lambda, \mathfrak{K}}}(sh)^2\big)\big((\tilde{v})^2_{n+\frac{1}{2}}+(\tilde{v})^2_{n+\frac{3}{2}}\big),\nonumber\\
&&Y_{21}^{(2)}=\int_0^T{s^2\mathcal{O_{\lambda, \mathfrak{K}}}(sh)v^2(a)}.\nonumber
\end{eqnarray}
\end{Lemma}

\begin{Lemma}[Estimate of $I_{22}$]\label{lem:estimate I22} For $sh\le \mathfrak{K}$, we have
\begin{equation}I_{22}=-2\int_{Q'_0}{c^2s^3\lambda^4\phi^3(\psi')^4 v^2}-X_{22}+Y_{22},\nonumber\end{equation}
\ with
\begin{eqnarray}&&Y_{22}=Y_{22}^{(1)}+Y_{22}^{(2)},\nonumber\\
&&Y_{22}^{(1)}=\int_0^T{s^3\mathcal{O_{\lambda, \mathfrak{K}}}(1)v(a)\frac{h^{2}}{2}(Dv)_{n+\frac{1}{2}}+s^3\mathcal{O_{\lambda, \mathfrak{K}}}(1)v(a)\frac{h^{2}}{2}(Dv)_{n+\frac{3}{2}}},\nonumber\\
&&Y_{22}^{(2)}=\int_0^T{s\mathcal{O_{\lambda, \mathfrak{K}}}(sh)^2v^2(a)},\nonumber
\end{eqnarray}
and
\begin{equation}X_{22}=\int_{Q'_0}{\mu_{22}v^2}+\int_{Q'_0}{\nu_{22}(Dv)^2},\nonumber
\end{equation}
\ where
\begin{equation}\mu_{22}=(s\lambda\phi)^3\mathcal{O}(1)+s^2\mathcal{O_{\lambda, \mathfrak{K}}}(1)+s^3\mathcal{O_{\lambda, \mathfrak{K}}}(sh)^2, \ \ \nu_{22}=s\mathcal{O_{\lambda, \mathfrak{K}}}(sh)^2.\nonumber
\end{equation}
\end{Lemma}

\begin{Lemma}[Estimate of $I_{23}$]\label{lem:estimate I23} For $\tau h(max_{[0,T]}\theta)\le \mathfrak{K}$, the term $I_{23}$ can be estimated from below in the following way
\begin{equation}I_{23}\ge \int_{\Omega'_0}s^2\big({\mathcal{O_{\lambda, \mathfrak{K}}}(1)v^2_{|_{t=0}}+\mathcal{O_{\lambda, \mathfrak{K}}}(1)v^2_{|_{t=T}}}\big)-X_{23}+Y_{23},\nonumber
\end{equation}
with
\begin{eqnarray}X_{23}&=&\int_{Q_0}{Ts^2\theta\mathcal{O_{\lambda, \mathfrak{K}}}(1)v^2}+\int_{Q_0}{s^{-1}\mathcal{O_{\lambda, \mathfrak{K}}}(sh)^2(\partial_t v)^2}\nonumber\\
&&+\int_{Q'_0}{(sh)^2s\mathcal{O_{\lambda, \mathfrak{K}}}(1)(D v)^2},\nonumber
\end{eqnarray}
and
\begin{eqnarray}&&Y_{23}=Y_{23}^{(1)}+Y_{23}^{(2)}+Y_{23}^{(3)},\nonumber\\
&&Y_{23}^{(1)}=\int_0^T{s^2\mathcal{O_{\lambda, \mathfrak{K}}}(1)\partial_t v(a) \frac{h}{2}(\tilde{v}_{n+\frac{1}{2}})+s^2\mathcal{O_{\lambda, \mathfrak{K}}}(1)\partial_t v(a)\frac{h}{2}(\tilde{v}_{n+\frac{3}{2}})},\nonumber\\
&&Y_{23}^{(2)}=\int_0^T{sT\theta \mathcal{O_{\lambda, \mathfrak{K}}}(sh)v^2(a)},\nonumber\\
&&Y_{23}^{(3)}=\mathcal{O_{\lambda, \mathfrak{K}}}(sh)^2v^2(a)|_{t=0}^{t=T}.\nonumber
\end{eqnarray}
\end{Lemma}

\begin{Lemma}[Estimate of $I_{31}$]\label{lem:estimate I31}  For $\tau h(max_{[0,T]}\theta)\le \mathfrak{K}$, we have

\begin{equation}I_{31}=-X_{31}+Y_{31},\nonumber\
\end{equation}
\ with 
\begin{equation}X_{31}=\int_{Q'_0}{T\theta s^2\mathcal{O_{\lambda, \mathfrak{K}}}(1){v}^2}+\int_{Q'_0}{T\theta\mathcal{O_{\lambda, \mathfrak{K}}}(sh)^2 (Dv)^2},\nonumber
\end{equation}
\ and 
\begin{eqnarray}&&Y_{31}=Y_{31}^{(1)}+Y_{31}^{(2)},\nonumber\\
&&Y_{31}^{(1)}=\int_0^T{T\theta s^2\mathcal{O_{\lambda, \mathfrak{K}}}(1)v(a)\frac{h}{2}(Dv)_{n+\frac{1}{2}}}+\int_0^T{T\theta s^2\mathcal{O_{\lambda, \mathfrak{K}}}(1)v(a)\frac{h}{2}(Dv)_{n+\frac{3}{2}}},\nonumber\\
&&Y_{31}^{(2)}= \int_0^T{T\theta s^2\mathcal{O_{\lambda, \mathfrak{K}}}(1)v^2(a)}.\nonumber\
\end{eqnarray}
\end{Lemma}

\begin{Lemma}[Estimate of $I_{32}$]\label{lem:estimate I32} \cite{BL12} For $\tau h(max_{[0,T]}\theta)\le \mathfrak{K}$, the term $I_{32}$ can be estimated from below in the following way
\begin{equation}I_{32}=-X_{32}=\int_{Q'_0}{Ts^2\theta\mathcal{O_{\lambda, \mathfrak{K}}}(1)v^2}.\nonumber
\end{equation}
\end{Lemma}

\begin{Lemma}[Estimate of $I_{33}$]\label{lem:estimate I33} \cite[proof of Lemma 3.9]{BL12} For $\tau h(max_{[0,T]}\theta)\le \mathfrak{K}$, the term $I_{33}$ can be estimated from below in the following way
\begin{equation}I_{33}\ge-X_{33}=\frac{1}{2}\tau\int_{Q'_0}{c'\varphi(\partial^2_t \theta)v^2}.\nonumber
\end{equation}
\end{Lemma}

\par \underline{Continuation of the proof of Theorem 4.1}. Collecting the terms we have obtained in the previous lemmata, from ~\eqref{eq:4.2} and ~\eqref{eq:4.3} for $0<\tau h(max_{[0,T]}\theta)\le \epsilon_1(\lambda)$ we find
\small{
\begin{eqnarray}&&\left\|{Av}\right\|^2_{L^2(Q'_0)}+\left\|{Bv}\right\|^2_{L^2(Q'_0)}+2\int_{Q'_0}{s\lambda^2(c^2\phi(\psi')^2)_d(Dv)^2}+2\int_{Q'_0}{c^2s^3\lambda^4\phi^3(\psi')^4v^2}\nonumber\\
&&\ \ \ \  \ +2\big(Y_{11}^{(1)}+Y_{21}^{(1,1)}\big)+2\big(Y_{11}^{(2,1)}+Y_{21}^{(1,21)}\big)+2Y_{13}\nonumber\\
&\le&C_{{\lambda, \mathfrak{K}}}\Big(\left\|{rf_1}\right\|^2_{L^2(Q'_0)}+\int_{\Omega'_0}s^2\big(v^2_{|_{t=0}}+v^2_{|_{t=T}}\big)+\int_{\Omega'_0}(Dv(T))^2\Big)+2X+2\underline{Y},\nonumber
\end{eqnarray}
}
with
\begin{eqnarray}&&\underline{Y}=-\Big(Y_{11}^{(2,2)}+Y_{12}+Y_{21}^{(1,22)}+Y_{21}^{(2)}+Y_{22}+Y_{23}+Y_{31}\Big),\nonumber\\ &&X=X_{11}+X_{12}+X_{13}+X_{21}+X_{22}+X_{23}+X_{31}+X_{32}+X_{33}\nonumber\\
&&\quad\quad +C_{\lambda,\mathfrak{K}}\big(\left\|{sv}\right\|^2_{L^2(Q'_0)}+h^2\left\|{sDv}\right\|^2_{L^2(Q'_0)}\big).\nonumber
\end{eqnarray}

\ With the following lemma, we may in fact ignore the term $Y_{11}^{(1)}+Y_{21}^{(1,1)}$ in the previous inequality.
\begin{Lemma}\label{lem:estimate Y11 plus Y21}For all $\lambda$ there exists $0<\epsilon_2(\lambda)<\epsilon_1(\lambda)$ such that for $0<\tau h(max_{[0,T]}\theta)\le \epsilon_2(\lambda)$ we have $Y_{11}^{(1)}+Y_{21}^{(1,1)}\ge 0.$
\end{Lemma}
\par Recalling that $\triangledown \psi\ge C>0$ in $\Omega\backslash \omega_0$ we may thus write
\small{
\begin{eqnarray}\label{eq:4.4}&&\left\|{Av}\right\|^2_{L^2(Q'_0)}+\left\|{Bv}\right\|^2_{L^2(Q'_0)}+\int_{Q'_0}{s(Dv)^2}+\int_{Q'_0}{s^3v^2}\nonumber\\
&&\quad\quad\quad\quad\quad\quad +2\big(Y_{11}^{(2,1)}+Y_{21}^{(1,21)}\big)+2Y_{13}\nonumber\\
&\le&C_{{\lambda, \mathfrak{K}}}\left(\left\|{rf_1}\right\|^2_{L^2(Q'_0)}+2\int^T_0{\int_{\omega_0}{s(Dv)^2}}+2\int^T_0{\int_{\omega_0}{s^3v^2}}\right.\nonumber\\
&& \left. +\int_{\Omega'_0}s^2(v^2_{|_{t=0}}+v^2_{|_{t=T}})+\int_{\Omega'_0}(Dv(T))^2\right)+2X+2\underline{Y}.
\end{eqnarray}
}

\begin{Lemma} \label{lem:estimate Y11 plus Y21 again}With the function $\psi$ satisfing the properties of Lemma~\ref{lem:weight function} and for $\tau h(max_{[0,T]}\theta)\le \mathfrak{K}$,  we have
\begin{equation}Y_{11}^{(2,1)}+Y_{21}^{(1,21)}\ge C\alpha_0\int_0^T{s\lambda\phi(a)(c_dDv)^2_{n+\frac{1}{2}}}+C\alpha_0\int_0^T{s^3\lambda^3\phi^3(a)v^2_{n+1}}+\mu_1+\mu_r,\nonumber
\end{equation}
with $\alpha_0$ as given in Lemma~\ref{lem:weight function} and where
\small{
\begin{eqnarray}
\mu_r&=&\int_0^T{s\mathcal{O_{\lambda}}(1)r_0^2}+\int_0^T{s^2\mathcal{O_{\lambda}}(1)r_0v_{n+1}}+\int_0^T{s\mathcal{O_{\lambda}}(1)r_0(c_dDv)_{n+\frac{1}{2}}}\nonumber\\
&+&\int_0^T{s^2\mathcal{O_{\lambda, \mathfrak{K}}}(sh)v^2_{n+1}}+\int_0^T{s\mathcal{O_{\lambda, \mathfrak{K}}}(sh)v_{n+1}(c_dDv)_{n+\frac{1}{2}}}+\int_0^T{s\mathcal{O_{\lambda, \mathfrak{K}}}(sh)r_0v_{n+1}},\nonumber
\end{eqnarray}
}
with $r_0$ as given in Lemma~\ref{lem:transmission condition} and
\begin{equation}
\mu_1=\mu_1^{(1)}+\mu_1^{(2)},\nonumber
\end{equation} where
\small{
\begin{eqnarray}
&&\mu_1^{(1)}=\int_0^T{s\mathcal{O_{\lambda, \mathfrak{K}}}(sh)(c_dDv)^2_{n+\frac{3}{2}}}+\int_0^T{s\mathcal{O_{\lambda, \mathfrak{K}}}(sh)(c_dDv)^2_{n+\frac{1}{2}}},\nonumber\\
&&\mu_1^{(2)}=\int_0^T{s^2\mathcal{O_{\lambda, \mathfrak{K}}}(sh)(c_dDv)_{n+\frac{3}{2}}v_{n+1}}+\int_0^T{s^2\mathcal{O_{\lambda, \mathfrak{K}}}(sh)(c_dDv)_{n+\frac{1}{2}}v_{n+1}}.\nonumber
\end{eqnarray}
}
\end{Lemma}

\par For a proof  see Appendix A.
\begin{Lemma} \label{lem:estimate Y13}With $0<\epsilon_3(\lambda)<\epsilon_2(\lambda)$ sufficiently small we obtain
\small{
\begin{eqnarray}&&Y_{13}\ge\int_0^T{\mathcal{C}_{\lambda, \mathfrak{K}}h(\partial_t v(a))^2}+ \int_0^T{\big(s T\theta\mathcal{O_{\lambda, \mathfrak{K}}}(sh)+T^2\theta^2\mathcal{O_{\lambda, \mathfrak{K}}}(sh) \big)v^2(a)}\nonumber\\
&+&s\mathcal{O_{\lambda, \mathfrak{K}}}(1)v^2(a)|_{t=0}^{t=T}+\int_0^T{\mathcal{O_{\lambda, \mathfrak{K}}}(sh)\partial_t v(a)(c_dDv)_{n+\frac{1}{2}}}+\int_0^T{\mathcal{O_{\lambda, \mathfrak{K}}}(1)\partial_t v(a)h(rf_1)_{n+1}}.\nonumber
\end{eqnarray}
}
where $C_{\lambda, \mathfrak{K}}$ is positive constant whose value depends on $\lambda$ and $sh$.
\end{Lemma}
\par For a proof  see Appendix A.
\par If we choose $\lambda_2\ge \lambda_1$ sufficiently large, then for $\lambda=\lambda_2$ (fixed for the rest of the proof) and $0<\tau h(\max_{[0,T]}\theta)\le \epsilon_3$, from ~\eqref{eq:4.4} and Lemma~\ref{lem:estimate Y11 plus Y21 again} and Lemma~\ref{lem:estimate Y13}, we can thus achieve the following inequality

\begin{eqnarray}\label{eq:4.5}&&\left\|{Av}\right\|^2_{L^2(Q'_0)}+\left\|{Bv}\right\|^2_{L^2(Q'_0)}+\int_{Q'_0}{s\left|{Dv}\right|^2dt}+\int_{Q'_0}{s^3v^2dt}\nonumber\\
&+&C\alpha_0\int_0^T{s(c_dDv)^2_{n+\frac{1}{2}}}+C\alpha_0\int_0^T{s^3v^2(a)}+\int_0^T{C_{{\lambda, \mathfrak{K}}}h(\partial_t v(a))^2}\nonumber\\
&\le&C_{{\lambda, \mathfrak{K}}}\left(\left\|{rf_1}\right\|^2_{L^2(Q'_0)}+2\int^T_0{\int_{\omega_0}{s(Dv)^2}}+2\int^T_0{\int_{\omega_0}{s^3v^2}}\right.\nonumber\\
&& \left. +\int_{\Omega'_0}s^2(v^2_{|_{t=0}}+v^2_{|_{t=T}})+\int_{\Omega'_0}(Dv(T))^2+sv^2(a)|_{t=0}^{t=T}\right)\nonumber\\
&+&\int_0^T{\big(s T\theta\mathcal{O_{\lambda, \mathfrak{K}}}(sh)+T^2\theta^2\mathcal{O_{\lambda, \mathfrak{K}}}(sh) \big)v^2(a)} +\int_0^T{\mathcal{O_{\lambda, \mathfrak{K}}}(sh)\partial_t v(a)(c_dDv)_{n+\frac{1}{2}}}\nonumber\\
&+&\int_0^T{\mathcal{O}_{{\lambda, \mathfrak{K}}}(1)\partial_t v(a) h(rf_1)_{n+1}}+2\underline{X}+2\underline{Y}+2\underline{Z}\quad\quad\quad\quad,
\end{eqnarray}
where $\underline{Z}=\mu_r+\mu_1$ with $\mu_r$ and $\mu_1$ are given as in Lemma~\ref{lem:estimate Y11 plus Y21 again} and where
\begin{eqnarray}\underline{X}&=&\quad\int_{Q'_0}{\bar{\mu}v^2}\quad+\quad\int_{Q'_0}{\bar{\nu}(Dv)^2}\nonumber\\
&+&X_{12}+X_{13}+X_{23}+X_{31}+X_{32}+X_{33}, \nonumber
\end{eqnarray}
 with $\bar{\mu}=s^2\mathcal{O_{\lambda, \mathfrak{K}}}(1)+s^3\mathcal{O_{\lambda, \mathfrak{K}}}(sh)$ and $\bar{\nu}$ of the form $s\mathcal{O_{\lambda, \mathfrak{K}}}(sh)$.

\par By using the Young's inequality, we estimate in turn all the terms of $\underline{Y}$, $\underline{Z}$ and the two terms at the RHS of~\eqref{eq:4.5} through the following Lemma whose proof can be found in Appendix A

\begin{Lemma}\label{lem:Young ineq}For $sh\le\mathfrak{K}$, we have

\small{
\begin{equation}\int_0^T{\mathcal{O}_{{\lambda, \mathfrak{K}}}(1)\partial_t v(a)h(rf_1)_{n+1}}\le \epsilon\int_0^T{\mathcal{O}_{{\lambda, \mathfrak{K}}}(1)h(\partial_t v(a))^2}+C_{\epsilon}\int_0^T{\mathcal{O}_{{\lambda, \mathfrak{K}}}(1)h(rf_1)^2_{n+1}},\nonumber
\end{equation}
}
\small{
\begin{equation}\int_0^T{\mathcal{O}_{{\lambda, \mathfrak{K}}}(sh)\partial_tv(a)(c_dDv)_{n+\frac{1}{2}}} \le \epsilon\int_0^T{\mathcal{O}_{{\lambda, \mathfrak{K}}}(1)h(\partial_t v(a))^2}+C_{\epsilon}\int_0^T{s\mathcal{O}_{{\lambda, \mathfrak{K}}}(sh)(c_dDv)^2_{n+\frac{1}{2}}}.\nonumber
\end{equation}
}


\begin{equation}\left|{Y_{11}^{(2,2)}}\right|\le\int_0^T{\alpha_{11}v^2_{n+1}}+\int_0^T{ \beta_{11}h (\partial_t v)^2_{n+1}}+\int_0^T{\gamma_{11}(c_dDv)^2_{n+\frac{1}{2}}}+\int_0^T{\eta_{11}h(rf_1)^2_{n+1}},\nonumber
\end{equation}

\begin{align}\alpha_{11}=\Big(s^3\mathcal{O_{\lambda, \mathfrak{K}}}(sh)^2+sT^2\theta^2\mathcal{O_{\lambda, \mathfrak{K}}}(sh)^4\Big) &\quad \quad \beta_{11}=\mathcal{O_{\lambda, \mathfrak{K}}}(sh)^3,\nonumber\\
\gamma_{11}=s\mathcal{O_{\lambda, \mathfrak{K}}}(sh)^2 &\quad \quad \eta_{11}=\mathcal{O_{\lambda, \mathfrak{K}}}(sh)^3.\nonumber
\end{align}

\begin{equation}\left|{Y_{12}}\right|\le\int_0^T{\alpha_{12}v^2_{n+1}}+\int_0^T{\beta_{12}h (\partial_t v)^2_{n+1}}+\int_0^T{\gamma_{12}(c_dDv)^2_{n+\frac{1}{2}}}+\int_0^T{\eta_{12}h(rf_1)^2_{n+1}},\nonumber
\end{equation}
\begin{align}\alpha_{12}=\Big(s^2\mathcal{O_{\lambda, \mathfrak{K}}}(1)+sT^2\theta^2\mathcal{O_{\lambda, \mathfrak{K}}}(sh)^2\Big)&\quad \quad \beta_{12}=\mathcal{O_{\lambda, \mathfrak{K}}}(sh),\nonumber\\
\gamma_{12}=\mathcal{O_{\lambda, \mathfrak{K}}}(1) &\quad \quad \eta_{12}= \mathcal{O_{\lambda, \mathfrak{K}}}(sh).\nonumber
\end{align}

\begin{equation}\left|{Y_{21}^{(1,22)}}\right|\le\int_0^T{\alpha_{21}v^2_{n+1}}+\int_0^T{\beta_{21}h (\partial_t v)^2_{n+1}}+\int_0^T{\gamma_{21}(c_dDv)^2_{n+\frac{1}{2}}}+\int_0^T{ \eta_{21}h(rf_1)^2_{n+1}},\nonumber
\end{equation}

\begin{align}\alpha_{21}=\Big(s^3\mathcal{O_{\lambda, \mathfrak{K}}}(1)+sT^2\theta^2\mathcal{O_{\lambda, \mathfrak{K}}}(sh)^4\Big)&\quad \quad\beta_{21}=\mathcal{O_{\lambda, \mathfrak{K}}}(sh)^3,\nonumber\\
\gamma_{21}=s\mathcal{O_{\lambda, \mathfrak{K}}}(sh)^3 &\quad \quad \eta_{21}=\mathcal{O_{\lambda, \mathfrak{K}}}(sh)^3.\nonumber
\end{align}

\begin{equation}\left|{Y_{22}^{(1)}}\right|\le\int_0^T{\alpha_{22}v^2_{n+1}}+\int_0^T{\beta_{22}h (\partial_t v)^2_{n+1}}+\int_0^T{\gamma_{22}(c_dDv)^2_{n+\frac{1}{2}}}+\int_0^T{\eta_{22}h(rf_1)^2_{n+1}},\nonumber
\end{equation}

\begin{align}\alpha_{22}=\Big(s^2\mathcal{O_{\lambda, \mathfrak{K}}}(sh)^2+sT^2\theta^2\mathcal{O_{\lambda, \mathfrak{K}}}(sh)^4\Big)&\quad\quad \beta_{22}=\mathcal{O_{\lambda, \mathfrak{K}}}(sh)^3,\nonumber\\
\gamma_{22}=\mathcal{O_{\lambda, \mathfrak{K}}}(sh)^3 &\quad\quad \eta_{22}=\mathcal{O_{\lambda, \mathfrak{K}}}(sh)^3.\nonumber
\end{align}

\begin{equation}\left|{Y_{23}^{(1)}}\right|\le \int_0^T{\alpha_{23}v^2_{n+1}}+ \int_0^T{\beta_{23}h(\partial_t v(a))^2}+\int_0^T{\gamma_{23}(c_dDv)^2_{n+\frac{1}{2}}}+\int_0^T{\eta_{22}h(rf_1)^2_{n+1}},\nonumber
\end{equation}

\begin{align}\alpha_{23}=\Big(s^3\mathcal{O_{\lambda, \mathfrak{K}}}(1)+sT^2\theta^2\mathcal{O_{\lambda, \mathfrak{K}}}(sh)^3\Big)& \quad\quad\beta_{23}=\mathcal{O_{\lambda, \mathfrak{K}}}(sh),\nonumber\\
\gamma_{23}=s\mathcal{O_{\lambda, \mathfrak{K}}}(sh)^2 &\quad \quad \eta_{22}=\mathcal{O_{\lambda, \mathfrak{K}}}(sh)^2.\nonumber
\end{align}

\begin{equation}\left|{Y_{31}^{(1)}}\right|\le \int_0^T{\alpha_{31}v^2_{n+1}}+\int_0^T{\beta_{23}h(\partial_t v)^2_{n+1}}+\int_0^T{\gamma_{23}(c_dDv)^2_{n+\frac{1}{2}}}+\int_0^T{\eta_{22}h(rf_1)^2_{n+1}},\nonumber
\end{equation}

\begin{align}\alpha_{31}=\Big( s^2T\theta\mathcal{O_{\lambda, \mathfrak{K}}}(sh)+sT^2\theta^2\mathcal{O_{\lambda, \mathfrak{K}}}(sh)\Big) &\quad\quad \beta_{23}= \mathcal{O_{\lambda, \mathfrak{K}}}(sh),\nonumber\\
\gamma_{23}=s\mathcal{O_{\lambda, \mathfrak{K}}}(sh) &\quad \quad \eta_{22}=\mathcal{O_{\lambda, \mathfrak{K}}}(sh).\nonumber
\end{align}

\begin{equation}\mu_1\le \int_0^T{\alpha_{1}v^2_{n+1}}\nonumber+\int_0^T{\beta_{1}h(\partial_t v)^2_{n+1}}+\int_0^T{\gamma_{1}(c_dDv)^2_{n+\frac{1}{2}}}+\int_0^T{\eta_{1}h(rf_1)^2_{n+1}},\nonumber
\end{equation}

\begin{align}\alpha_{1}=\Big(s^3\mathcal{O_{\lambda, \mathfrak{K}}}(sh)+sT^2\theta^2\mathcal{O_{\lambda, \mathfrak{K}}}(sh)^3\Big) &\quad\quad \beta_{1}=\mathcal{O_{\lambda, \mathfrak{K}}}(sh)^2,\nonumber\\
\gamma_{1}=s\mathcal{O_{\lambda, \mathfrak{K}}}(sh) &\quad \quad \eta_{1}=\mathcal{O_{\lambda, \mathfrak{K}}}(sh)^2.\nonumber
\end{align}

\begin{equation}\mu_r\le \int_0^T{\alpha_{r}v^2_{n+1}}+\int_0^T{\beta_{r}h(\partial_t v)^2_{n+1}}+\int_0^T{\gamma_{r}(c_dDv)^2_{n+\frac{1}{2}}}+\int_0^T{\eta_{r}h(rf_1)^2_{n+1}},\nonumber
\end{equation}
\end{Lemma}

\begin{align}\alpha_{r}=\Big(s^3\mathcal{O_{\lambda, \mathfrak{K}}}(sh)+sT^2\theta^2\mathcal{O_{\epsilon, \lambda, \mathfrak{K}}}(sh)^2+\epsilon s^3\mathcal{O_{\lambda, \mathfrak{K}}}(1)\Big) &\quad\quad \beta_{r}=\mathcal{O_{\epsilon, \lambda, \mathfrak{K}}}(sh),\nonumber\\
\gamma_{r}=\Big(s\mathcal{O_{\lambda, \mathfrak{K}}}(sh)+\epsilon s\mathcal{O_{\lambda, \mathfrak{K}}}(1)\Big)&\quad \quad \eta_{r}=\mathcal{O_{\epsilon, \lambda, \mathfrak{K}}}(sh).\nonumber
\end{align}

\par Futhermore, we can estimate the term in $X_{12}$ as follows
\begin{eqnarray}\int_{Q'_0}{s\mathcal{O_{\lambda, \mathfrak{K}}}(1)\tilde{v}Dv}&\le& \int_{Q'_0}{s\mathcal{O_{\lambda, \mathfrak{K}}}(1)(\tilde{v})^2}+\int_{Q'_0}{s\mathcal{O_{\lambda, \mathfrak{K}}}(1)(Dv)^2}\nonumber\\
&\le& \int_{Q'_0}{s\mathcal{O_{\lambda, \mathfrak{K}}}(1)\wtilde{\left|{v}\right|}^2}+\int_{Q'_0}{s\mathcal{O_{\lambda, \mathfrak{K}}}(1)(Dv)^2}\nonumber\\
&=& \int_{Q_0}{s\mathcal{O_{\lambda, \mathfrak{K}}}(1)v^2}+\int_{Q'_0}{s\mathcal{O_{\lambda, \mathfrak{K}}}(1)(Dv)^2},\nonumber
\end{eqnarray}
by Lemma~\ref{lem:average of two function} and as $\int_{\Omega'_0}{\mathcal{O_{\lambda, \mathfrak{K}}}(1)\wtilde{\left|{v}\right|}^2}=\int_{\Omega_0}{\mathcal{O_{\lambda, \mathfrak{K}}}(1){v}^2}$.


\par Observe that 
\begin{equation} 1\le T^2\theta ~\textrm{and}~ \left|{\partial^2_t \theta}\right|\le CT^2\theta^3.\nonumber 
\end{equation}

\par We can now choose $\epsilon_4$ and $h_0$ sufficiently small, with $0<\epsilon_4\le \epsilon_3(\lambda_2)$, $0<h_0\le h_1(\lambda_2)$, and $\tau_2\ge 1$ sufficiently large, such that for $\tau\ge \tau_2(T+T^2)$, $0<h\le h_0$, and $\tau h(max_{[0,T]}\theta)\le \epsilon_4$, from ~\eqref{eq:4.5} and Lemma~\ref{lem:Young ineq} we get
\small{
\begin{eqnarray}\label{eq:4.6}&&\left\|{Av}\right\|^2_{L^2(Q'_0)}+\left\|{Bv}\right\|^2_{L^2(Q'_0)}+\int_{Q'_0}{s\left|{Dv}\right|^2}+\int_{Q'_0}{s^3v^2}\nonumber\\
&+&C\alpha_0\int_0^T{s(c_dDv)^2_{n+\frac{1}{2}}}+C\alpha_0\int_0^T{s^3v^2_{n+1}}+C_{\lambda, \mathfrak{K}}\int_0^T{h(\partial_t v(a))^2}\nonumber\\
&\le&C_{{\lambda, \mathfrak{K}}}\left(\left\|{rf_1}\right\|^2_{L^2(Q'_0)}+\int^T_0{\int_{\omega_0}{s(Dv)^2}}+\int^T_0{\int_{\omega_0}{s^3v^2}}\right. \nonumber\\
&&\left.+ h^{-2}\big(\int_{\Omega'_0}v^2_{|_{t=0}}+\int_{\Omega'_0}v^2_{|_{t=T}}\big)+sv^2(a)|_{t=0}^{t=T}\right)+\int_0^T{\mathcal{O_{\lambda, \mathfrak{K}}}(1)h(rf_1)^2_{n+1}}\nonumber\\
&+&\int_{Q_0}{s\mathcal{O_{\lambda, \mathfrak{K}}}(1)v^2}+\int_{Q_0}{s^{-1}\mathcal{O_{\lambda, \mathfrak{K}}}(sh)(\partial_t v)^2}+\int_{Q_0}s^2T\theta \mathcal{O_{\lambda, \mathfrak{K}}}(1)v^2.\nonumber\\
\end{eqnarray}
}
where we used that $(Dv)^2\le Ch^{-2}((\tau^+ v)^2+(\tau^- v)^2)$ and the last three terms whose integral taken on domain $Q_0$ come from the term in $X_{12}$, $X_{13}$ and $X_{23}$ respectively.

\par As $\tau\ge\tau_2(T+T^2)$ then $s\ge \tau_2>0$ and furthermore we observe that
\begin{eqnarray}\left\|{s^{-\frac{1}{2}}\partial_t v}\right\|^2_{L^2(Q'_0)}&\le& \mathcal{C_{\lambda, \mathfrak{K}}}\Big(\left\|{s^{-\frac{1}{2}}Bv}\right\|^2_{L^2(Q'_0)}+\left\|{s^{\frac{1}{2}}v}\right\|^2_{L^2(Q'_0)}+\left\|{s^{\frac{1}{2}}Dv}\right\|^2_{L^2(Q'_0)}\Big)\nonumber\\
&\le& \mathcal{C_{\lambda, \tau,\mathfrak{K}}}\Big(\left\|{Bv}\right\|^2_{L^2(Q'_0)}+\left\|{s^{\frac{3}{2}}v}\right\|^2_{L^2(Q'_0)}+\left\|{s^{\frac{1}{2}}Dv}\right\|^2_{L^2(Q'_0)}\Big).\nonumber
\end{eqnarray}
\par We then add the following terms $\int_0^T hs^3v^2_{n+1}$ and $\int_0^T hs^{-1}(\partial_t v(a))^2$ on both the right hand side and the left hand side of ~\eqref{eq:4.6}. This allows us to change the domain of integration from $Q'_0$ to $Q_0$ for the discrete integrals on the primal mesh. No additional term is required for discrete integrals on the dual mesh. For $sh$ sufficiently small and $s\ge 1$ sufficiently large, these terms at the right hand side are then absorbed by the terms at the left hand side. More precisely, with $0<\epsilon_0\le \epsilon_4$ sufficiently small and for $\tau\ge \tau_2 (T+T^2)$, $0<h\le h_0$, and $0<\tau h (max_{[0,T]}\theta)\le \epsilon_0$ we thus obtain


\begin{eqnarray}\label{eq:4.7}&&\left\|{s^{-\frac{1}{2}}\partial_t v}\right\|^2_{L^2(Q_0)}\quad+\quad\int_{Q_0}{s(Dv)^2}\quad+\quad\int_{Q_0}{s^3v^2}\nonumber\\
&\le&C_{{\lambda, \mathfrak{K}}}\left( \left\|{rf_1}\right\|^2_{L^2(Q_0)}+\int^T_0{\int_{\omega_0}{s(Dv)^2}}+\int^T_0{\int_{\omega_0}{s^3v^2}}\right.\nonumber\\
&&\left.+ h^{-2}\big(\int_{\Omega_0}v^2_{|_{t=0}}+\int_{\Omega_0}v^2_{|_{t=T}}\big)+s\mathcal{O_{\lambda, \mathfrak{K}}}(1)v^2(a)|_{t=0}^{t=T}\right).
\end{eqnarray}
\par Now we shall estimate the term $s\mathcal{O_{\lambda, \mathfrak{K}}}(1)v^2(a)_{|_{t=T}}$. We have
\begin{equation}
\left\| {v_{|_{t=T}}}\right\|^2_{L^2(\Omega_0)}=\sum_{j=1}^{n+m+1}hv^2_{j|_{t=T}}\ge h \left\| {v_{|_{t=T}}}\right\|^2_{L^{\infty}(\Omega_0)}.\nonumber
\end{equation}
It follows that, as $sh$ is bounded
\small{
\begin{eqnarray}
\left| s\mathcal{O_{\lambda, \mathfrak{K}}}(1)v^2(a)_{|_{t=T}}\right|\le C_{{\lambda, \mathfrak{K}}}\ s\left\| {v_{|_{t=T}}}\right\|^2_{L^{\infty}(\Omega_0)}&\le& C_{{\lambda, \mathfrak{K}}}\ sh^{-1}\left\| {v_{|_{t=T}}}\right\|^2_{L^{2}(\Omega_0)}\nonumber\\
&\le& C_{{\lambda, \mathfrak{K}}}\ h^{-2}\left\| {v_{|_{t=T}}}\right\|^2_{L^{2}(\Omega_0)}.\nonumber
\end{eqnarray}
}
\par Similarly, we treat the term $s\mathcal{O_{\lambda, \mathfrak{K}}}(1)v^2(a)_{|_{t=0}}$ as
\begin{equation}
\left| s\mathcal{O_{\lambda, \mathfrak{K}}}(1)v^2(a)_{|_{t=0}}\right|\le C_{{\lambda, \mathfrak{K}}}\ h^{-2}\left\| {v_{|_{t=0}}}\right\|^2_{L^{2}(\Omega_0)}.\nonumber
\end{equation}
 Therefore, ~\eqref{eq:4.7} can be written as
\small{
 \begin{eqnarray}&&\left\|{s^{-\frac{1}{2}}\partial_t v}\right\|^2_{L^2(Q_0)}\quad+\quad\left\|{s^{\frac{1}{2}}Dv} \right\|^2_{L^2(Q_0)} +\quad\left\|{s^{\frac{3}{2}}v} \right\|^2_{L^2(Q_0)}\nonumber\\
 &\le&C_{{\lambda, \mathfrak{K}}}\left( \left\|{rf_1}\right\|^2_{L^2(Q_0)}+\left\|{s^{\frac{1}{2}}Dv} \right\|^2_{L^2\big({(0,T)\times \omega_0}\big)}+\left\|{s^{\frac{3}{2}}v} \right\|^2_{L^2\big({(0,T)\times \omega_0}\big)}\right.\nonumber\\
 &&\left. \quad\quad\quad\quad\quad\quad\quad\quad\quad\quad\quad\quad\quad\quad+ h^{-2}\big(\int_{\Omega_0}v^2_{|_{t=0}}+\int_{\Omega_0}v^2_{|_{t=T}}\big)\right).\nonumber
 \end{eqnarray}
 }
\par We next remove the volume norm $\left\|{s^{\frac{1}{2}}Dv}\right\|^2_{L^2( (0,T)\times \omega_0)}$ by proceeding as in the proof of Theorem 4.1 in \cite{BHL10a} we thus write
\begin{eqnarray}\label{eq:Carleman est}&&\tau^{-1}\left\|{\theta^{-\frac{1}{2}}e^{\tau \theta \varphi}\partial_t u}\right\|^2_{L^2(Q_0)}+\tau\left\|{\theta^{\frac{1}{2}}e^{\tau \theta\varphi}Du}\right\|^2_{L^2(Q_0)}+\tau^3\left\|{\theta^{\frac{3}{2}}e^{\tau\theta\varphi}u}\right\|^2_{L^2(Q_0)}\nonumber\\
&&\ \ \le C_{\lambda,\mathfrak{K}}\left( \left\|{e^{\tau\theta\varphi}P^{\mathfrak{M}}u}\right\|^2_{L^2(Q_0)}+\tau^3\left\|{\theta^{\frac{3}{2}}e^{\tau \theta\varphi}u}\right\|^2_{L^2( (0,T) \times \omega_0)} \right.\nonumber\\
&&\ \ \ \ \ \ \ \ \ \ \ \ \ \ \ \ \ \ \ \ \ \left. +h^{-2} \left|{ e^{\tau\theta\varphi}u|_{t=0}}\right|^2_{L^2(\Omega_0)}+ h^{-2}\left|{e^{\tau\theta\varphi}u|_{t=T}}\right|^2_{L^2(\Omega_0)} \right),
\end{eqnarray}

\par As we have $\underset{[0,T]}{\max}\theta\le \frac{1}{T\alpha}$, we see that a sufficient condition for $\tau h\big({\underset{[0,T]}{\max}\theta}\big)\le \epsilon_0$ then becomes $\tau h(T\alpha)^{-1}\le \epsilon_0$. To finish the proof, we need to express all the terms in the estimate above in terms of the original function $u$. We can proceed exactly as in the end of proof of Theorem 4.1 in \cite{BHL10a}.
\section{Carleman estimates for regular non uniform meshes}

\label{sec:non uniform meshes}

\par In this section we focus on extending the above result to the class of non piecewise uniform meshes introduced in Section~\ref{subsec:Families of non-uniform meshes}. We choose a function $\vartheta$ satisfying ~\eqref{eq:vartheta} and further $\vartheta|_{[a-\delta, a+\delta]}$ is chosen affine (for some $\delta>0$ to remain fixed in the sequel). The way we proceed here is similar to what is done in \cite{BHL10a}. In this framework, we shall prove a non-uniform Carleman estimate for the parabolic operator $\mathcal{P}^{\mathfrak{M}}=-\partial_t+\mathcal{A}^{\mathfrak{M}}$ on the mesh $\mathfrak{M}$ by using the result on uniform meshes of Section~\ref{sec: Carleman estimate uniform mesh}.
\par By using first-order Taylor formulae we obtain the following result.
\begin{Lemma}\label{lem:define zeta}Let us define $\zeta\in \mathbb{R}^{\overline{\mathfrak{M}}}$ and $\overline{\zeta}\in \mathbb{R}^{\mathfrak{M}}$ as follows
\begin{align}
 \zeta_{i+\frac{1}{2}}=\frac{h'_{i+\frac{1}{2}}}{h}, i\in \left\{{0,\ldots, n+m+1}\right\}& \quad\quad \overline{\zeta_{i}}=\frac{h'_{i}}{h}, i\in \left\{{1,\ldots, n+m+1}\right\}\nonumber
\end{align}
These two discrete functions are connected to the geometry of the primal and dual meshes ${\mathfrak{M}}$ and $\overline{{\mathfrak{M}}}$ and we have
\begin{eqnarray}&&0<{\underset{\Omega_0}{\inf \vartheta'}}\le \zeta_{i+\frac{1}{2}}\le {\underset{\Omega_0}{\sup \vartheta'}}, \quad\quad \forall i\in{0,\ldots,n+m+1}\nonumber\\
&&0<{\underset{\Omega_0}{\inf \vartheta'}}\le\overline{ \zeta}_{i}\le {\underset{\Omega_0}{\sup \vartheta'}}, \quad\quad \forall i\in{1,\ldots,n+m+1}\nonumber
\end{eqnarray}
\begin{equation} \left|{\bar{D}\zeta}\right|_{L_{\infty}(\Omega)}\le \frac{\left\|{\vartheta''}\right\|_{L_{\infty}}}{\underset{\Omega_0}{\inf}\vartheta'}, \ \ \ \  \left|{{D}\bar{\zeta}}\right|_{L_{\infty}(\Omega)}\le \frac{\left\|{\vartheta''}\right\|_{L_{\infty}}}{\underset{\Omega_0}{\inf}\vartheta'}.\nonumber
\end{equation}
\end{Lemma}
\par We introduce some notation. To any $u\in \mathbb{C}^{\mathfrak{M}\cup \partial \mathfrak{M}}$, we associate the discrete function denoted by $\mathcal{Q}^{\mathfrak{M}_0}_{\mathfrak{M}}u\in \mathbb{C}^{\mathfrak{M}_0\cup \partial \mathfrak{M}_0}$ defined on the uniform mesh $\mathfrak{M}_0$ which takes the same values as $u$ at the corresponding nodes. More precisely, if $u=\sum_{i=1}^{n+m+1}\textbf{1}_{[x'_{i-\frac{1}{2}},x'_{i+\frac{1}{2}}]}u_i$, we let 
\begin{equation} \mathcal{Q}^{\mathfrak{M}_0}_{\mathfrak{M}}u=\sum_{i=1}^{n+m+1}\textbf{1}_{[(i-\frac{1}{2})h,(i+\frac{1}{2})h]}u_i\nonumber\end{equation}
and  $(\mathcal{Q}^{\mathfrak{M}_0}_{\mathfrak{M}}u)_0=u_0$, $({\mathcal{Q}}^{\mathfrak{M}_0}_{\mathfrak{M}}u)_{n+m+2}=u_{n+m+2}$.
Similarly, for $u\in \mathbb{C}^{\overline{\mathfrak{M}}}$, $u=\sum_{i=1}^{n+m+1}\textbf{1}_{[x'_{i},x'_{i+1}]}u_{i+\frac{1}{2}}$, we set 
\begin{equation}{\mathcal{Q}}^{\overline{\mathfrak{M}_0}}_{\overline{\mathfrak{M}}}u=\sum_{i=0}^{n+m+1}\textbf{1}_{[ih,(i+1)h]}u_{i+\frac{1}{2}}.\nonumber\end{equation} 
The operators ${\mathcal{Q}}^{\mathfrak{M}_0}_{{\mathfrak{M}}}$ and ${\mathcal{Q}}^{\overline{\mathfrak{M}_0}}_{\overline{\mathfrak{M}}}$ are invertible and we denote by ${\mathcal{Q}}^{\mathfrak{M}}_{{\mathfrak{M}_0}}$ and ${\mathcal{Q}}^{\overline{\mathfrak{M}}}_{\overline{\mathfrak{M}_0}}$ their respective inverses. We give commutation properties between these operators and discrete-difference operators through the following Lemmata whose proofs can be found in \cite{BHL10a}.   

\begin{Lemma}\label{lem:property of zeta}\cite[see the proof of  Lemma 5.2]{BHL10a}
\begin{enumerate}
\item For any $u\in \mathbb{C}^{\mathfrak{M}\cup \partial \mathfrak{M}}$ and any $v\in\mathbb{C}^{\overline{\mathfrak{M}}}$, we have
\begin{equation}D({\mathcal{Q}}^{\mathfrak{M}_0}_{\mathfrak{M}}u)={\mathcal{Q}}^{\overline{\mathfrak{M}_0}}_{\overline{\mathfrak{M}}}(\zeta Du), \quad \quad\bar{D}{\mathcal{Q}}^{\overline{\mathfrak{M}_0}}_{\overline{\mathfrak{M}}}v={\mathcal{Q}}^{\mathfrak{M}_0}_{\mathfrak{M}}(\overline{\zeta}\ \bar{ D}v)\nonumber
\end{equation}
\item For any $u\in \mathbb{C}^{\mathfrak{M}\cup \partial \mathfrak{M}}$ we have
\begin{equation}\bar{D}(c_dDu)=(\overline{\zeta})^{-1}{\mathcal{Q}}^{\mathfrak{M}}_{\mathfrak{M}_0}\Big(\bar{D}\big({({\mathcal{Q}}^{\overline{\mathfrak{M}_0}}_{\overline{\mathfrak{M}}}\frac{c_d}{\zeta})D({\mathcal{Q}}^{\mathfrak{M}_0}_{\mathfrak{M}}u)}\big)\Big).\nonumber
\end{equation}
\end{enumerate}
\end{Lemma}

\begin{Lemma}\cite[see proof of  Lemma 5.3]{BHL10a}\label{lem:property of zeta1}
\par For any $u\in \mathbb{C}^{\mathfrak{M}}$ and any $v\in\mathbb{C}^{\overline{\mathfrak{M}}}$, we have
\begin{eqnarray}&&(\underset{\Omega_0}{\sup\vartheta'})^{-1}\left|{u}\right|^2_{L^2(\Omega)}\le\left|{{\mathcal{Q}}^{\mathfrak{M}_0}_{\mathfrak{M}}u}\right|^2_{L^2(\Omega_0)}\le(\underset{\Omega_0}{\inf\vartheta'})^{-1}\left|{u}\right|^2_{L^2(\Omega)}\nonumber\\
&&(\underset{\Omega_0}{\sup\vartheta'})^{-1}\left|{v}\right|^2_{L^2(\Omega)}\le\left|{{\mathcal{Q}}^{\overline{\mathfrak{M}_0}}_{\overline{\mathfrak{M}}}u}\right|^2_{L^2(\Omega_0)}\le(\underset{\Omega_0}{\inf\vartheta'})^{-1}\left|{v}\right|^2_{L^2(\Omega)}\nonumber
\end{eqnarray}
Futhermore, the same inequalities hold by replacing $\Omega$ by $\omega$ and $\Omega_0$ by $\omega_0$, respectively.
\end{Lemma}
\par For any continuous function $f$ defined on $\overline{\Omega}$ (resp. on $\overline{\Omega_0}$) we denote by $\Pi_{\mathfrak{M}}f= (f(x'_i))_{0\le i\le n+m+2}\in \mathbb{C}^{\mathfrak{M}\cup \partial \mathfrak{M}}$ the sampling of $f$ on $\mathfrak{M}$ (resp. $\Pi_{\mathfrak{M}_0}f= (f(ih))_{0\le i\le n+m+2}\in \mathbb{C}^{{\mathfrak{M}_0}\cup \partial \mathfrak{M}_0}$ the sampling of $f$ on $\mathfrak{M}_0$).
\begin{Lemma}\cite[see the proof of  Lemma 5.4]{BHL10a}\label{lem:property of zeta2}
\par Let $f$ be a continuous function defined on $\Omega$
\begin{equation}{\mathcal{Q}}^{\mathfrak{M}_0}_{\mathfrak{M}}\Pi_{\mathfrak{M}}f=\Pi_{\mathfrak{M}_0}(f\circ\vartheta).\nonumber
\end{equation}
\par In particular, for  $u\in \mathbb{C}^{\mathfrak{M}\cup \partial \mathfrak{M}}$ we have
\begin{equation}{\mathcal{Q}}^{\mathfrak{M}_0}_{\mathfrak{M}}\Big({(\Pi_{\mathfrak{M}}f)u}\Big)=\Pi_{\mathfrak{M}_0}(f\circ\vartheta)({\mathcal{Q}}^{\mathfrak{M}_0}_{\mathfrak{M}}u).\nonumber
\end{equation}
\end{Lemma}
\par Moreover, by making use of Taylor formulae we get the following result
\begin{Lemma}\label{lem:bounded property}With $\zeta$ defined as in Lemma~\ref{lem:define zeta} we have
\begin{equation}\left\|{\bar{D}D\nu}\right\|_{\infty}< \infty, \quad\quad  \left\|{\bar{D}\tilde{\nu}}\right\|_{\infty}< \infty,\quad\quad 0<\left\|{\nu} \right\|_{\infty}, \left\|\tilde{\nu} \right\|_{\infty}< \infty \nonumber
\end{equation}
where $\nu:=\frac{1}{{\mathcal{Q}}^{\mathfrak{M}_0}_{\mathfrak{M}}\overline{\zeta}}$.
\end{Lemma}
 \begin{proof} 
\par From the definition of $\zeta$, ${\mathcal{Q}}^{\mathfrak{M}_0}_{\mathfrak{M}}$ and $D$ acting on $ \mathbb{C}^{\mathfrak{M}_0}$, $\bar{D}$ acting on $\mathbb{C}^{\overline{\mathfrak{M}_0}}$ we have
\begin{eqnarray}\label{eq:5.1}\Big(\bar{D}D(\frac{1}{\mathcal{Q}^{\mathfrak{M}_0}_{\mathfrak{M}}\overline{\zeta}})\Big)_i&:=&(\bar{D}D\nu)_i\nonumber\\
&=&\frac{\nu_{i+1}-2\nu_i+\nu_{i-1}}{h^2}\nonumber\\
&=&\frac{1}{h}\frac{(h'_i-h'_{i+1})h'_{i-1}-(h'_{i-1}-h'_i)h'_{i+1}}{h'_{i-1}h'_ih'_{i+1}}.
\end{eqnarray}
\par We find
\begin{eqnarray}&&h'_i=x'_{i+\frac{1}{2}}-x'_{i-\frac{1}{2}}=\frac{x'_{i+1}-x'_{i-1}}{2}=\frac{\vartheta\big({(i+1)h}\big)-\vartheta\big((i-1)h\big)}{2}=\frac{\vartheta_{i+1}-\vartheta_{i-1}}{2},\nonumber\\
&&h'_{i+1}=\frac{\vartheta\big({(i+2)h}\big)-\vartheta\big(ih\big)}{2}=\frac{\vartheta_{i+2}-\vartheta_{i}}{2},\nonumber\\
&&h'_{i-1}=\frac{\vartheta(ih)-\vartheta\big((i-2)h\big)}{2}=\frac{\vartheta_{i}-\vartheta_{i-2}}{2}.\nonumber
\end{eqnarray}

\par By using Taylor formulae we write
\begin{eqnarray}\vartheta_{i+2}&=&\vartheta_i+(2h)\vartheta'_i+\frac{(2h)^2}{2}\vartheta''_i+\frac{(2h)^3}{6}\vartheta'''_i+\frac{(2h)^4}{24}\vartheta^{(4)}_i+\mathcal{O}(h^5),\nonumber\\
\vartheta_{i-2}&=&\vartheta_i-(2h)\vartheta'_i+\frac{(2h)^2}{2}\vartheta''_i-\frac{(2h)^3}{6}\vartheta'''_i+\frac{(2h)^4}{24}\vartheta^{(4)}_i+\mathcal{O}(h^5),\nonumber\\
\vartheta_{i+1}&=&\vartheta_i+h\vartheta'_i+\frac{h^2}{2}\vartheta''_i+\frac{h^3}{6}\vartheta'''_i+\frac{(h)^4}{24}\vartheta^{(4)}_i+\mathcal{O}(h^5),\nonumber\\
\vartheta_{i-1}&=&\vartheta_i-h\vartheta'_i+\frac{h^2}{2}\vartheta''_i-\frac{h^3}{6}\vartheta'''_i+\frac{(h)^4}{24}\vartheta^{(4)}_i+\mathcal{O}(h^5).\nonumber
\end{eqnarray}
\par Thus we have
\begin{eqnarray}h'_i&=&{2h\vartheta'_i+\frac{2h^3}{6}\vartheta'''_i+\mathcal{O}(h^5)},\nonumber\\
h'_{i+1}&=&{2h\vartheta'_i+\frac{(2h)^2}{2}\vartheta''_i+\frac{(2h)^3}{6}\vartheta'''+\frac{(2h)^4}{24}\vartheta^{(4)}+\mathcal{O}(h^5)},\nonumber\\
h'_{i-1}&=&{2h\vartheta'_i-\frac{(2h)^2}{2}\vartheta''_i+\frac{(2h)^3}{6}\vartheta'''-\frac{(2h)^4}{24}\vartheta^{(4)}+\mathcal{O}(h^5)}.\nonumber
\end{eqnarray}
\par From ~\eqref{eq:5.1} we obtain
\begin{equation}
\bar{D}D(\frac{1}{{\mathcal{Q}}^{\mathfrak{M}_0}_{\mathfrak{M}}\overline{\zeta}})_i=\frac{N}{D},\nonumber\\
\end{equation}
where
\begin{eqnarray}
N&=&(h'_i-h'_{i+1})h'_{i-1}-(h'_{i-1}-h'_i)h'_{i+1}\nonumber\\
&=&\Big({-\frac{(2h)^2}{2}\vartheta''_i-h^3\vartheta'''_i-\frac{(2h)^4}{24}\vartheta^{(4)}_i+\mathcal{O}(h^5)}\Big)\big({(2h)\vartheta'_i-\frac{(2h)^2}{2}\vartheta''_i+\mathcal{O}(h^3)}\big)\nonumber\\
&-&\Big({-\frac{(2h)^2}{2}\vartheta''_i-h^3\vartheta'''_i-\frac{(2h)^4}{24}\vartheta^{(4)}_i+\mathcal{O}(h^5)}\Big)\big({(2h)\vartheta'_i+\frac{(2h)^2}{2}\vartheta''_i+\mathcal{O}(h^3)}\big)\nonumber\\
&=&\frac{(2h)^4}{2}(\vartheta''_i)^2+\mathcal{O}(h^5),\nonumber
\end{eqnarray}
and
\begin{equation}
D=h\times h'_{i-1}\times h'_i\times h'_{i+1}=(2h)^4(\vartheta'_i)^3+\mathcal{O}(h^5).\nonumber
\end{equation}
Thus, we have
\begin{equation}
\left| \bar{D}D(\frac{1}{{\mathcal{Q}}^{\mathfrak{M}_0}_{\mathfrak{M}}\overline{\zeta}})_i\right|  \lesssim (\inf \vartheta')^{-3}<\infty,\nonumber
\end{equation}
which proves the first result. Next, we proceed with the second result in the same manner as above. We have
\begin{equation}
(\bar{D}\tilde{\nu})_i=\frac{\tilde{\nu}_{i+\frac{1}{2}}-\tilde{\nu}_{i+\frac{1}{2}}}{h}=\frac{\nu_{i+1}-\nu_{i-1}}{2h}=\frac{1}{2h}\big(\frac{h}{h'_{i+1}}-\frac{h}{h'_{i-1}}\big)=\frac{h'_{i-1}-h'_{i+1}}{h'_{i+1}h'_{i-1}}.\nonumber
\end{equation}
\par By using the computations of $h'_{i-1}, h'_{i+1}$ above we find
\begin{equation}
\left|(\bar{D}\tilde{\nu})_i \right|=\left|\frac{-(2h)^2{\vartheta''}+\mathcal{O}(h^4)}{(2h)^2{\vartheta'}^2+\mathcal{O}(h^3)} \right| \lesssim \frac{\left\|{\vartheta''}\right\|_{\infty}}{(\inf \vartheta')^2}<\infty,\nonumber 
\end{equation}
which yields the second result.
\par Moreover, with the properties of $\zeta$ shown as in Lemma~\ref{lem:define zeta} we can assert
\begin{equation}
0<\left\|{\nu} \right\|_{\infty}, \left\|\tilde{\nu} \right\|_{\infty}< \infty. \nonumber
\end{equation}
\end{proof}
\par From Lemmata~\ref{lem:property of zeta}~--~\ref{lem:property of zeta2} we thus obtain the following discrete Carleman estimate for the operator $P^{\mathfrak{M}}=-\partial_t -D(c_dD.)$ on the mesh $\mathfrak{M}$.
\begin{Theorem}\label{theo:carleman non uniform} Let $\omega \subset \Omega_2$ be a non-empty open set and we set  $f:=\bar{D}(c_dDu)$. For the parameter $\lambda>1$ sufficiently large, there exists C, $\tau_0\ge 1$, $h_0>0$, $\epsilon_0>0$, depending on $\omega$ such that for any mesh $\mathfrak{M}$ obtained from $\vartheta$ by ~\eqref{eq:vartheta}~--~\eqref{eq:non uniform mesh}, we have
\begin{eqnarray}\label{eq:5.3}&&\tau^{-1}\left\|{\theta^{-\frac{1}{2}}e^{\tau \theta \varphi}\partial_t u}\right\|^2_{L^2(Q)}+\tau\left\|{\theta^{\frac{1}{2}}e^{\tau \theta\varphi}Du}\right\|^2_{L^2(Q)}+\tau^3\left\|{\theta^{\frac{3}{2}}e^{\tau\theta\varphi}u}\right\|^2_{L^2(Q)}\nonumber\\
&&\le C_{\lambda,\mathfrak{K}}\left( \left\|{e^{\tau\theta\varphi}P^{\mathfrak{M}}u}\right\|^2_{L^2(Q)}+\tau^3\left\|{\theta^{\frac{3}{2}}e^{\tau \theta\varphi}u}\right\|^2_{L^2( (0,T) \times \omega)} \right.\nonumber\\
&&\left. \quad\quad\quad\quad\quad\quad\quad\quad +h^{-2} \left|{ e^{\tau\theta\varphi}u|_{t=0}}\right|^2_{L^2(\Omega)}+ h^{-2}\left|{e^{\tau\theta\varphi}u|_{t=T}}\right|^2_{L^2(\Omega)} \right),
\end{eqnarray}
for all $\tau\ge \tau_0(T+T^2)$, $0<h\le h_0$ and $\tau h(\alpha T)^{-1}\le  \epsilon_0$ and for all $u\in C^\infty(0,T;\mathbb{C}^{\mathfrak{M}})$ satisfying $u|_{\partial\Omega}=0$.
\end{Theorem}
\begin{proof} We set $w={\mathcal{Q}}^{\mathfrak{M}_0}_{\mathfrak{M}}u$ defined on the uniform mesh ${\mathfrak{M}_0}$. By using Lemma~\ref{lem:property of zeta} we have
\begin{equation}\label{eq:5.4}{\mathcal{Q}}^{\mathfrak{M}_0}_{\mathfrak{M}}(\bar{\zeta}P^{\mathfrak{M}}u)=-({\mathcal{Q}}^{\mathfrak{M}_0}_{\mathfrak{M}}\bar{\zeta})\partial_t w-\bar{D}\Big(\big({\mathcal{Q}}^{\overline{\mathfrak{M}_0}}_{\overline{\mathfrak{M}}}\frac{c_d}{\zeta}\big)Dw\Big).
\end{equation}
We observe that the right-hand side of ~\eqref{eq:5.4} is a semi-discrete parabolic operator of the form $P^{\mathfrak{M}_0}=\xi'(-\partial_t -\frac{1}{\xi'}D(\xi_d D.))$ applied to $w$, where 
\begin{equation}\label{eq:5.5}\xi'={\mathcal{Q}}^{\mathfrak{M}_0}_{\mathfrak{M}}\bar{\zeta},\quad \xi_d={\mathcal{Q}}^{\overline{\mathfrak{M}_0}}_{\overline{\mathfrak{M}}}\frac{c_d}{\zeta}.
\end{equation}

\par We set $\nu:=\frac{1}{\xi'}=\frac{1}{{\mathcal{Q}}^{\mathfrak{M}_0}_{\mathfrak{M}}\bar{\zeta}}$ and we find
\begin{equation}
\bar{\tilde{\nu}}=\nu+h^2\bar{D}D\nu=\nu+h^2\mathcal{O}(1),\nonumber
\end{equation}
by using Lemma~\ref{lem:average of two function} and Lemma~\ref{lem:bounded property}.
\par Thus, the operator ${P}^{\mathfrak{M}_0}$ can be written in form as 
\begin{equation}{P}^{\mathfrak{M}_0}w=\xi'\big(-\partial_t w-\bar{\tilde{\nu}}\bar{D}(\xi_dDw)+h^2\mathcal{O}(1)\bar{D}(\xi_dDw)\big).\nonumber
\end{equation}

\par Moreover, using Lemma~\ref{lemma: discrete leibnitz} we have
\begin{equation}\bar{\tilde{\nu}}\bar{D}(\xi_dDw)=\bar{D}(\tilde{\nu}\xi_dDw)-\bar{D}(\tilde{\nu})\overline{\xi_dDw}.\nonumber
\end{equation}
\par We set ${P}^{\mathfrak{M}_0}_0w:=-\partial_tw-\bar{D}(\tilde{\nu}\xi_dDw)=-\partial_tw-\bar{D}(b_dDw)$ with $b_d=\tilde{\nu}\xi_d$. From the properties of $\tilde{\nu}$ and $\xi_d$ it follows that 
\begin{equation}0<b_{min}\le b\le b_{max}\quad\quad ~\textrm{and}\quad\quad \left\| \bar{D}(b_d)\right\| _{\infty}<\infty.\nonumber
\end{equation} 
\par First, we shall obtain a Carleman estimate for ${P}^{\mathfrak{M}_0}_0$. Then we shall deduce a Carleman estimate for the operator 
\begin{equation}\label{eq:5.6}{P}^{\mathfrak{M}_0}w=\xi'\big(\ {P}^{\mathfrak{M}_0}_0w+\bar{D}(\tilde{\nu})\overline{\xi_dDw}+h^2\mathcal{O}(1)\bar{D}(\xi_dDw) \ \big)
\end{equation}

\par Now, we consider the function $\psi\circ\vartheta: (t,x)\mapsto \psi\big(t,\vartheta(x)\big)$. By using the properties listed in Lemma~\ref{lem:weight function} and ~\eqref{eq:vartheta}, we shall see that $\psi\circ\vartheta$ is a suitable weight function associated to the control domain $\omega_0=\vartheta^{-1}(w)$ in $\Omega_0$, i.e., that $\psi\circ\vartheta$ satisfies Lemma~\ref{lem:weight function} for the domaims $\Omega_0$ and $\omega_0$.

\par The important property to checking is the trace property. The remaining properties are left to the reader. We set
\begin{displaymath}
\textbf{B} =
\left( \begin{array}{cc}
b_{11} & b_{12}\\
b_{21} & b_{22} \\
\end{array} \right),
\end{displaymath}
with
\begin{eqnarray}&&b_{11}=[(\psi\circ\vartheta)'\star]_a\nonumber\\
&&b_{22}=[b(\psi\circ\vartheta)'\star]^2_a(\psi\circ\vartheta)'(a^+)+[b^2(\psi\circ\vartheta)'^3\star]_a\nonumber\\
&&b_{12}=b_{21}=[b(\psi\circ\vartheta)'\star]_a(\psi\circ\vartheta)'(a^+)\nonumber
\end{eqnarray}
where $b=\frac{1}{{\mathcal{Q}}^{\mathfrak{M}_0}_{\mathfrak{M}}\bar{\zeta}}{\mathcal{Q}}^{\mathfrak{M}_0}_{\mathfrak{M}}\frac{c}{\bar{\zeta}}=\frac{c\circ\vartheta}{\bar{\zeta}^2}$.  Morever, we have $\vartheta'_+(a)=\vartheta'_-(a)$ and $\bar{\zeta}_{n+1}=\vartheta'_{n+1}$ (recall that $\vartheta|_{[a-\delta, a+\delta]}$ is an affine function). It follows that
\begin{eqnarray}&&b_{11}=[\psi'(\vartheta)\star]_a\vartheta'(a)=[\psi'\star]_{a'}\vartheta'(a),\nonumber\\
&&b_{22}=[\frac{c\circ\vartheta}{\bar{\zeta}^2}\psi'(\vartheta)\vartheta'\star]^2_a\psi'(\vartheta)(a^+)\vartheta'(a^+)+[(\frac{c\circ\vartheta}{\bar{\zeta}^2})^2\big({\psi'(\vartheta)\vartheta'}\big)^3\star]_a\nonumber\\
&&\quad \ \ =[c\psi'\star]_{a'}\psi'(a^+)\frac{1}{\vartheta'(a)}+[c^2(\psi')^3\star]_{a'}\frac{1}{\vartheta'(a)},\nonumber\\
&&b_{12}=b_{21}=[\frac{c\circ\vartheta}{\bar{\zeta}^2}\psi'(\vartheta)\vartheta'\star]_a\psi'(\vartheta(a^+))\vartheta'(a^+)\nonumber\\
&&\quad\quad\quad\quad=[c\psi'\star]_{a'}\psi'(a'^+).\nonumber
\end{eqnarray}

\par We can see that $(Bw,w)=(Aw,w)\ge \alpha_0 \left\|{w}\right\|^2$. This means that $\psi\circ\vartheta$ satisfies the trace property. 
\par Through Theorem ~\ref{theo:Carleman estinmate uniform}, we obtained a discrete uniform Carleman estimate for $P^{\mathfrak{M}_0}_0$ and the Carleman weight function is of the form $r_0=e^{s\varphi_0}$, with $\varphi_0=\varphi\circ\vartheta=e^{\lambda\psi_0}-e^{\lambda K}$ where $\psi_0=\psi\circ\vartheta$ on the uniform mesh ${\mathfrak{M}_0}$ . We can deduce the same result on the non-uniform mesh ${\mathfrak{M}}$. Namely, through ~\eqref{Carleman estinmate uniform } we see that the following estimate holds
\small{
\begin{eqnarray}\label{eq:5.7}&&\tau^{-1}\left\|{\theta^{-\frac{1}{2}}e^{\tau \theta \varphi_0}\partial_t w}\right\|^2_{L^2(Q_0)}+\tau\left\|{\theta^{\frac{1}{2}}e^{\tau \theta\varphi_0}Dw}\right\|^2_{L^2(Q_0)}+\tau^3\left\|{\theta^{\frac{3}{2}}e^{\tau\theta\varphi_0}w}\right\|^2_{L^2(Q_0)}\nonumber\\
&&\le C\left( \left\|{e^{\tau\theta\varphi_0}P^{\mathfrak{M}_0}_0w}\right\|^2_{L^2(Q_0)}+\tau^3\left\|{\theta^{\frac{3}{2}}e^{\tau \theta\varphi_0}w}\right\|^2_{L^2( (0,T) \times \omega_0)} \right.\nonumber\\
&&\left. \quad\quad\quad\quad\quad\quad+h^{-2}\left|{ e^{\tau\theta\varphi_0}w|_{t=0}}\right|^2_{L^2(\Omega_0)}+ h^{-2}\left|{e^{\tau\theta\varphi_0}w|_{t=T}}\right|^2_{L^2(\Omega_0)} \right),
\end{eqnarray}
}
and the constant $C$ is uniform in $h$ for $\tau$ sufficiently large and with $\tau h(\alpha T)^{-1}\le \epsilon_0$, for $\epsilon_0$ sufficiently small. Note that, setting $\tilde{\epsilon}_0=(\inf_{\Omega_0}\vartheta')\epsilon_0$, we see that the condition $\tau h'(\alpha T)^{-1}\le\tilde{ \epsilon}_0$ on the size of the non-uniform mesh ${\mathfrak{M}}$ implies the condition $\tau h(\alpha T)^{-1}\le \epsilon_0$ for the uniform mesh ${\mathfrak{M}_0}$.
\par From ~\eqref{eq:5.6}~--~\eqref{eq:5.7} we deduce the following Carleman estimate for $P^{\mathfrak{M}_0}$ 
\small{
\begin{eqnarray}\label{eq:5.8}&&\tau^{-1}\left\|{\theta^{-\frac{1}{2}}e^{\tau \theta \varphi_0}\partial_t w}\right\|^2_{L^2(Q_0)}+\tau\left\|{\theta^{\frac{1}{2}}e^{\tau \theta\varphi_0}Dw}\right\|^2_{L^2(Q_0)}+\tau^3\left\|{\theta^{\frac{3}{2}}e^{\tau\theta\varphi_0}w}\right\|^2_{L^2(Q_0)}\nonumber\\
&&\le C\left( \left\|{e^{\tau\theta\varphi_0}P^{\mathfrak{M}_0}w}\right\|^2_{L^2(Q_0)}+\left\|{e^{\tau\theta\varphi_0}\bar{D}(\tilde{\nu})\overline{\xi_dDw}}\right\|^2_{L^2(Q_0)}+ h^4\left\| e^{\tau\theta\varphi_0}\bar{D}(\xi_dDw)\right\|^2_{L^2(Q_0)} \right.\nonumber\\
&&\left. +\tau^3\left\|{\theta^{\frac{3}{2}}e^{\tau \theta\varphi_0}w}\right\|^2_{L^2( (0,T) \times \omega_0)} +h^{-2} \left|{ e^{\tau\theta\varphi_0}w|_{t=0}}\right|^2_{L^2(\Omega_0)}+ h^{-2}\left|{e^{\tau\theta\varphi_0}w|_{t=T}}\right|^2_{L^2(\Omega_0)} \right).\nonumber\\
\end{eqnarray}
}

\par Now, by using Lemma~\ref{lem:bounded property} we estimate $\left\|{e^{\tau\theta\varphi_0}\bar{D}(\tilde{\nu})\overline{\xi_dDw}}\right\|^2_{L^2(Q_0)}$ in the RHS of ~\eqref{eq:5.8} as
\begin{equation}\left\|{e^{\tau\theta\varphi_0}\bar{D}(\tilde{\nu})\overline{\xi_dDw}}\right\|^2_{L^2(Q_0)}\le C\left\|{e^{s\varphi_0}\overline{\xi_dDw}}\right\|^2_{L^2(Q_0)}.\nonumber
\end{equation}
\par We see that
\begin{equation}\overline{\xi_dDw}=\frac{1}{2}\Big({\bar{\tau}_+(\xi_dDw)+\bar{\tau}_-(\xi_dDw)}\Big).\nonumber
\end{equation}
\par Hence we find
\begin{eqnarray}&&\left\|{e^{s\varphi_0}\overline{\xi_dDw}}\right\|^2_{L^2(Q_0)}\nonumber\\
&\le& C\Big({\left\|{e^{s\varphi_0}{\bar{\tau}_+(\xi_dDw)}}\right\|^2_{L^2(Q_0)}+\left\|{e^{s\varphi_0}{\bar{\tau}_-(\xi_dDw)}}\right\|^2_{L^2(Q_0)}}\Big)\nonumber\\
&\le&C\Big({\left\|{e^{s\varphi_0}{\bar{\tau}_+\big(\xi_dD({\mathcal{Q}}^{\mathfrak{M}_0}_{\mathfrak{M}}u)\big)}}\right\|^2_{L^2(Q_0)}+\left\|{e^{s\varphi_0}{\bar{\tau}_-\big(\xi_dD({\mathcal{Q}}^{\mathfrak{M}_0}_{\mathfrak{M}}u)\big)}}\right\|^2_{L^2(Q_0)}}\Big)\nonumber\\
&\le&C\Big({\left\|{e^{s\varphi_0}{\bar{\tau}_+\big(\xi_d{\mathcal{Q}}^{\overline{\mathfrak{M}_0}}_{\overline{\mathfrak{M}}}(\zeta Du)\big)}}\right\|^2_{L^2(Q_0)}+\left\|{e^{s\varphi_0}{\bar{\tau}_-\big(\xi_d{\mathcal{Q}}^{\overline{\mathfrak{M}_0}}_{\overline{\mathfrak{M}}}(\zeta Du)\big)}}\right\|^2_{L^2(Q_0)}}\Big)\nonumber\\
&\le&C\Big({\left\|{e^{s\varphi_0}{\bar{\tau}_+{\mathcal{Q}}^{\overline{\mathfrak{M}_0}}_{\overline{\mathfrak{M}}}\big(c_dDu\big)}}\right\|^2_{L^2(Q_0)}+\left\|{e^{s\varphi_0}{\bar{\tau}_-{\mathcal{Q}}^{\overline{\mathfrak{M}_0}}_{\overline{\mathfrak{M}}}\big(c_dDu\big)}}\right\|^2_{L^2(Q_0)}}\Big)\nonumber\\
&\le&C\Big(\left\|{{\mathcal{Q}}^{\mathfrak{M}_0}_{\mathfrak{M}}\Big({e^{s\varphi}\bar{\tau}_+(c_dDu)}\Big)}\right\|^2_{L^2(Q_0)}+\left\|{{\mathcal{Q}}^{\mathfrak{M}_0}_{\mathfrak{M}}\Big({e^{s\varphi}\bar{\tau}_-(c_dDu)}\Big)}\right\|^2_{L^2(Q_0)}\Big)\nonumber\\
&\le&C({\inf\vartheta'})^{-1}\Big(\left\|{{e^{s\varphi}\bar{\tau}_+(c_dDu)}}\right\|^2_{L^2(Q)}+\left\|{{e^{s\varphi}\bar{\tau}_-(c_dDu)}}\right\|^2_{L^2(Q)}\Big),\nonumber
\end{eqnarray}
by using ~\eqref{eq:5.5} and Lemmata~\ref{lem:property of zeta}~--~\ref{lem:property of zeta2}.
\par We treat $\left\|{{e^{s\varphi}\bar{\tau}_+(c_dDu)}}\right\|_{L^2(Q)}$ (the term $\left\|{{e^{s\varphi}\bar{\tau}_-(c_dDu)}}\right\|_{L^2(Q)}$ can be treated similarly). We find
\begin{equation}\label{eq:5.9}\left\|{{e^{s\varphi}\bar{\tau}_+(c_dDu)}}\right\|_{L^2(Q)}=\left\|{{r\bar{\tau}_+(c_dDu)}}\right\|_{L^2(Q)}\le \left\|{{(\tau_-r)(c_dDu)}}\right\|_{L^2(Q)}\le C\left\|{{(\tau_-r)Du}}\right\|_{L^2(Q)}.
\end{equation}
\par We have $\tau_-r=r(\rho\tau_-r)=r\big({1+\mathcal{O_{\lambda, \mathfrak{K}}}(sh)}\big)$ (due to Proposition~\ref{prop:property 1}). From that we can write
\begin{equation}\left\|{e^{\tau\theta\varphi_0}\bar{D}(\tilde{\nu})\overline{\xi_dDw}}\right\|^2_{L^2(Q_0)}\le C({\inf\vartheta'})^{-1} \left\|{e^{s\varphi}Du}\right\|_{L^2(Q)},\nonumber
\end{equation}
which allows one to absorb by the term at the LHS of the Carleman estimate by choosing $\tau$ sufficiently large.
\par Next, we estimate $h^4\left\| e^{\tau\theta\varphi_0}\bar{D}(\xi_dDw)\right\|^2_{L^2(Q_0)}$ in the RHS of ~\eqref{eq:5.8} as
\begin{eqnarray}
&&h^4\left\| e^{\tau\theta\varphi_0}\bar{D}(\xi_dDw)\right\|^2_{L^2(Q_0)}\nonumber\\
&=&h^2\left\|{e^{s\varphi_0}{\bar{\tau}_+(\xi_dDw)}}-{e^{s\varphi_0}{\bar{\tau}_-(\xi_dDw)}} \right\|^2_{L^2(Q_0)} \nonumber\\
&\le& Ch^2\Big({\left\|{e^{s\varphi_0}{\bar{\tau}_+(\xi_dDw)}}\right\|^2_{L^2(Q_0)}+\left\|{e^{s\varphi_0}{\bar{\tau}_-(\xi_dDw)}}\right\|^2_{L^2(Q_0)}}\Big)\nonumber\\
&\le&Ch^2\Big({\left\|{e^{s\varphi_0}{\bar{\tau}_+\big(\xi_dD({\mathcal{Q}}^{\mathfrak{M}_0}_{\mathfrak{M}}u)\big)}}\right\|^2_{L^2(Q_0)}+\left\|{e^{s\varphi_0}{\bar{\tau}_-\big(\xi_dD({\mathcal{Q}}^{\mathfrak{M}_0}_{\mathfrak{M}}u)\big)}}\right\|^2_{L^2(Q_0)}}\Big)\nonumber\\
&\le&Ch^2\Big({\left\|{e^{s\varphi_0}{\bar{\tau}_+\big(\xi_d{\mathcal{Q}}^{\overline{\mathfrak{M}_0}}_{\overline{\mathfrak{M}}}(\zeta Du)\big)}}\right\|^2_{L^2(Q_0)}+\left\|{e^{s\varphi_0}{\bar{\tau}_-\big(\xi_d{\mathcal{Q}}^{\overline{\mathfrak{M}_0}}_{\overline{\mathfrak{M}}}(\zeta Du)\big)}}\right\|^2_{L^2(Q_0)}}\Big)\nonumber\\
&\le&Ch^2\Big({\left\|{e^{s\varphi_0}{\bar{\tau}_+{\mathcal{Q}}^{\overline{\mathfrak{M}_0}}_{\overline{\mathfrak{M}}}\big(c_dDu\big)}}\right\|^2_{L^2(Q_0)}+\left\|{e^{s\varphi_0}{\bar{\tau}_-{\mathcal{Q}}^{\overline{\mathfrak{M}_0}}_{\overline{\mathfrak{M}}}\big(c_dDu\big)}}\right\|^2_{L^2(Q_0)}}\Big)\nonumber\\
&\le&Ch^2\Big(\left\|{{\mathcal{Q}}^{\mathfrak{M}_0}_{\mathfrak{M}}\Big({e^{s\varphi}\bar{\tau}_+(c_dDu)}\Big)}\right\|^2_{L^2(Q_0)}+\left\|{{\mathcal{Q}}^{\mathfrak{M}_0}_{\mathfrak{M}}\Big({e^{s\varphi}\bar{\tau}_-(c_dDu)}\Big)}\right\|^2_{L^2(Q_0)}\Big)\nonumber\\
&\le&Ch^2({\inf\vartheta'})^{-1}\Big(\left\|{{e^{s\varphi}\bar{\tau}_+(c_dDu)}}\right\|^2_{L^2(Q)}+\left\|{{e^{s\varphi}\bar{\tau}_-(c_dDu)}}\right\|^2_{L^2(Q)}\Big),\nonumber
\end{eqnarray}
by using ~\eqref{eq:5.5} and Lemmata~\ref{lem:property of zeta}~--~\ref{lem:property of zeta2}.
We proceed with an estimate as in ~\eqref{eq:5.9}. We thus obtain
\begin{equation}
h^4\left\| e^{\tau\theta\varphi_0}\bar{D}(\xi_dDw)\right\|^2_{L^2(Q_0)}\le Ch^2({\inf\vartheta'})^{-1} \left\|{e^{s\varphi}Du}\right\|_{L^2(Q)},\nonumber 
\end{equation}
which allows one to absorb by the term in the LHS of Carleman estimate by choosing $\tau$ sufficiently large.

\par Futhermore, by using the previous Lemmata ~\ref{lem:define zeta}~--~\ref{lem:property of zeta2} and considering each term in ~\eqref{eq:5.8} separately, we see that we have the following estimates
 \begin{itemize}
\item For the first term in LHS of ~\eqref{eq:5.8}
\begin{eqnarray}&&\left\|{\theta^{-\frac{1}{2}}e^{\tau \theta \varphi_0}\partial_t w}\right\|^2_{L^2(Q_0)}=\left\|{{\mathcal{Q}}^{\mathfrak{M}_0}_{\mathfrak{M}}(\theta^{-\frac{1}{2}}e^{\tau \theta \varphi}\partial_t u})\right\|^2_{L^2(Q_0)}\nonumber\\
&&\quad \quad \quad \quad \quad\quad\quad\quad\quad\ge (\underset{\Omega_0}{\sup\vartheta'})^{-1}\left\|{\theta^{-\frac{1}{2}}e^{\tau \theta \varphi} \partial_tu}\right\|^2_{L^2(Q)},\nonumber
\end{eqnarray}
and a similar inequality holds for $\left\|{\theta^{\frac{3}{2}}e^{\tau\theta\varphi_0}w}\right\|^2_{L^2(Q_0)}$.
\item For the second term of LHS of ~\eqref{eq:5.8} we use Lemma~\ref{lem:property of zeta} and Lemma ~\ref{lem:property of zeta1} as follows
\begin{eqnarray}\left\|{\theta^{\frac{1}{2}}e^{\tau \theta\varphi_0}Dw}\right\|^2_{L^2(Q_0)}&=&\left\|{\theta^{\frac{1}{2}}e^{\tau \theta\varphi_0}D({\mathcal{Q}}^{\mathfrak{M}_0}_{\mathfrak{M}}u)}\right\|^2_{L^2(Q_0)}=\left\|{\theta^{\frac{1}{2}}\mathcal{Q}}^{\overline{\mathfrak{M}_0}}_{\overline{\mathfrak{M}}}(e^{\tau \theta\varphi}){\mathcal{Q}}^{\overline{\mathfrak{M}_0}}_{\overline{\mathfrak{M}}}(\zeta Du)\right\|^2_{L^2(Q_0)}\nonumber\\
&=&\left\|{\theta^{\frac{1}{2}}\mathcal{Q}}^{\overline{\mathfrak{M}_0}}_{\overline{\mathfrak{M}}}(\zeta e^{\tau \theta\varphi}Du)\right\|^2_{L^2(Q_0)}\gtrsim \left\|{\theta^{\frac{1}{2}}e^{\tau \theta\varphi}Du}\right\|^2_{L^2(Q)}.\nonumber
\end{eqnarray} 

\item By using ~\eqref{eq:5.4} and Lemma ~\ref{lem:property of zeta1} we have 
\begin{eqnarray} \left\|{e^{\tau\theta\varphi_0}P^{\mathfrak{M}_0}w}\right\|^2_{L^2(Q_0)}&=& \left\|{e^{\tau\theta\varphi_0}{\mathcal{Q}}^{\mathfrak{M}_0}_{\mathfrak{M}}(\bar{\zeta}P^{\mathfrak{M}}u)}\right\|^2_{L^2(Q_0)}=\left\|{{\mathcal{Q}}^{\mathfrak{M}_0}_{\mathfrak{M}}(e^{\tau\theta\varphi}\bar{\zeta}P^{\mathfrak{M}}u)}\right\|^2_{L^2(Q_0)}\nonumber\\
&&\lesssim  \left\|{e^{\tau\theta\varphi}\bar{\zeta}P^{\mathfrak{M}}u}\right\|^2_{L^2(Q)}\lesssim\left\|{e^{\tau\theta\varphi}P^{\mathfrak{M}}u}\right\|^2_{L^2(Q)}.\nonumber
\end{eqnarray}
\item For the third term of RHS of ~\eqref{eq:5.8}
\begin{eqnarray}\left|{ e^{\tau\theta\varphi_0}w|_{t=0}}\right|^2_{L^2(\Omega_0)}&=&\left|{{\mathcal{Q}}^{\mathfrak{M}_0}_{\mathfrak{M}}( e^{\tau\theta\varphi}w|_{t=0}})\right|^2_{L^2(\Omega_0)}\nonumber\\
&\lesssim& \left|{ e^{\tau\theta\varphi}w|_{t=0}}\right|^2_{L^2(\Omega)}\nonumber
\end{eqnarray}
and a similar inequality holds for $\left|{e^{\tau\theta\varphi_0}w|_{t=T}}\right|^2_{L^2(\Omega_0)}$, $\left\|{\theta^{\frac{3}{2}}e^{\tau \theta\varphi_0}w}\right\|^2_{L^2( (0,T) \times \omega_0)} $.
\item Finally, since $\vartheta(\omega_0)=\omega$ we have
\begin{eqnarray}\left\|{\theta^{\frac{3}{2}}e^{\tau \theta\varphi_0}Du}\right\|^2_{L^2( (0,T) \times \omega_0)}&=&\left\|{{\mathcal{Q}}^{\mathfrak{M}_0}_{\mathfrak{M}}(\theta^{\frac{3}{2}}e^{\tau \theta\varphi}Du)}\right\|^2_{L^2( (0,T) \times \omega_0)}\nonumber\\
&\lesssim&\left\|{\theta^{\frac{3}{2}}e^{\tau \theta\varphi}Du}\right\|^2_{L^2( (0,T) \times \omega}\nonumber
\end{eqnarray}
\end{itemize} 
The proof is complete.
\end{proof}

\section{Controllability results}
\label{sec:Controllability results}
\par The Carleman estimate proved in the previous Section allows to give observability estimate that yields results of controllability to the trajectories for classes of semi-linear heat equations.
\subsection{The linear case} We consider the following semi-discrete parabolic problem with potential

\begin{equation}\label{eq:6.1}\partial_t y+\mathcal{A}^{\mathfrak{M}}y+ay=\textbf{1}_{\omega}v, \quad  t\in (0,T)\quad\quad y|_{\partial \Omega}=0
\end{equation}
\par The adjoint system associated with the controlled system with potential ~\eqref{eq:6.1} is given by
\begin{equation}\label{eq:6.2}-\partial_t q+\mathcal{A}^{\mathfrak{M}}y+ay=0, \quad  t\in (0,T)\quad\quad q|_{\partial \Omega}=0
\end{equation}
We assume that a piecewise $C^1$ diffusion coefficient $c$ satisfies ~\eqref{eq:coefficient} and $\Omega=(0,1)$. From Carleman estimate ~\eqref{Carleman estinmate uniform } we obtain a following observability estimate.
\begin{Proposition}There exists positive constants $C_0, C_1$ and $C_2$ such that for all $T>0$ and all potential fucntion $a$, under the condition $h\le \min(h_0,h_1)$ with
\begin{equation}h_1=C_0\big(1+\frac{1}{T}+\left\|{a}\right\|_{\infty}^{\frac{2}{3}}\big)^{-1}\nonumber
\end{equation}
any solution of ~\eqref{eq:6.2} satisfies
\begin{equation}\label{eq:6.3}\left|{q(0)}\right|_{L^2(\Omega)}\le C_{obs}\left\|{q}\right\|^2_{L^2( (0,T)\times \omega)}+e^{-\frac{C_1}{h}+T\left\|{a}\right\|_{\infty}}\left|{q(T)}\right|^2_{L^2(\Omega)},
\end{equation}
with $C_{obs}=e^{C_2\big(1+\frac{1}{T}+T\left\|{a}\right\|_{\infty}+\left\|{a}\right\|_{\infty}^{\frac{2}{3}}\big)}$.
\end{Proposition}
\begin{Remark}In comparision the observability inequality in continuous case which performed in \cite{BDL07}, we find that the observability inequality obtained here is weak since there is an additional term depending upon $h$ at right-hand-side of inequality ~\eqref{eq:6.3}.
\end{Remark}
\par From the result of Proposition 6.1 we deduce the following controllability result for system ~\eqref{eq:6.1}.

\begin{Proposition}There exists positive constants $C_1, C_2$, $C_3$ and for $T>0$ a map $L_{T,a}:\mathbb{R}^{\mathfrak{M}}\to L^2(0,T;\mathbb{R}^{\mathfrak{M}})$ such that if $h\le \min(h_0,h_2)$ with
\begin{equation}h_1=C_0\big(1+\frac{1}{T}+T\left\|{a}\right\|_{\infty}+\left\|{a}\right\|_{\infty}^{\frac{2}{3}}\big)^{-1}\nonumber
\end{equation}
for all initial data $y_0\in \mathbb{R}^{\mathfrak{M}}$, there exists a semi-discrete control function $v$ given by $v=L_{a}(y_0)$ such that the solution to ~\eqref{eq:6.1} satisfies
$$\left|{y(T)}\right|_{L^2(\Omega)}\le C_0 e^{-C_2/h}\left|{y_0}\right|_{L^2(\Omega)}$$ and $$ \left\|{v}\right\|_{L^2(Q)} \le C_0 \left|{y_0}\right|_{L^2(\Omega)}$$,
with $C_{0}=e^{C_3\big(1+\frac{1}{T}+T\left\|{a}\right\|_{\infty}+\left\|{a}\right\|_{\infty}^{\frac{2}{3}}\big)}$.
\end{Proposition}
\par Note that the final state is of size $ e^{-C/h}\left|{y_0}\right|_{L^2(\Omega)}$. The proof of these proposition are given in \cite{BL12}. 
%
\subsection{The semilinear case}We consider the following semilinear semi-discrete control problem
\begin{equation}\label{eq:6.4}\big(\partial_t +\mathcal{A}^{\mathfrak{M}}\big)y+{G}(y)=\textbf{1}_{\omega}v, \quad y\in (0,T) \quad\quad y|_{\partial \Omega}=0,\quad y(0)=y_0
\end{equation}
where $\omega \subset \Omega$. The function ${G}:\mathbb{R}\to \mathbb{R}$ is assumed of the form
\begin{equation}\label{eq:6.5}{G}(x)=xg(x),\quad\quad x\in \mathbb{R},
\end{equation}
with $g$ Lipschitz continuous. Here, we consider the function $g$ in two cases: $g\in L^{\infty}(\mathbb{R})$ and the more general case as
\begin{equation}\label{eq:6.6}\left|{g(x)}\right|\le K\ln^r(e+\left|{x}\right|), \quad x\in \mathbb{R},\quad\quad\textmd{with} \quad\quad 0\le r<\frac{3}{2}
\end{equation}
\par The results of semi-discrete parabolic with potential above allows one to obtain controllability results for parabolic equation with semi-linear terms whose proofs are given in \cite{BL12}
\begin{Theorem}We assume that $g\in L^{\infty}(\mathbb{R})$ and $c$ satisfies ~\eqref{eq:coefficient}. There exists positive constants $C_0, C_1$ such that for all $T>0$ and $h$ chosen sufficiently small, for all initial data $y_0\in \mathbb{R}^{\mathfrak{M}}$, there exists a semi-discrete control function $v$ with
\begin{equation}\left\|{v}\right\|_{L^2(Q)}\le C\left|{y_0}\right|_{L^2(\Omega)}\nonumber
\end{equation}
such that the solution to the semi-linear parabolic equation ~\eqref{eq:6.4} satisfies
\begin{equation}\left|{y(T)}\right|_{L^2(\Omega)}\le Ce^{-C_0/h}\left|{y_0}\right|_{L^2(\Omega)}\nonumber
\end{equation}
with $C_{0}=e^{C_1\big(1+\frac{1}{T}+T\left\|{g}\right\|_{\infty}+\left\|{g}\right\|_{\infty}^{\frac{2}{3}}\big)}$.
\end{Theorem}
\begin{Theorem}Let $\Omega=(0,1)$, $c$ satisfy ~\eqref{eq:coefficient} and $G$ satisfy ~\eqref{eq:6.5}~-~\eqref{eq:6.6}. There exists $C_0$ such that, for $T>0$ and $M>0$, there exists positive constants $C, h_0$ such that for $0<h\le h_0$ and for all initial data $y_0\in \mathbb{R}^{\mathfrak{M}}$ satisfying $\left|y_0\right|_{L^2(\Omega)}\le M$ there exists a semi-discrete control function $v$ such that the solution to the semi-linear parabolic equation
 \begin{equation}\big(\partial_t -\bar{D}cD\big)y+{G}(y)=\textbf{1}_{\omega}v, \quad y\in (0,T) \quad\quad y|_{\partial \Omega}=0,\quad y(0)=y_0
\end{equation}
satisfies
\begin{equation}\left|{y(T)}\right|_{L^2(\Omega)}\le Ce^{-C_0/h}\left|{y_0}\right|_{L^2(\Omega)}\nonumber
\end{equation}
where $C=C(T,M)$.
\end{Theorem}
\par Observe that the constants are uniform with respect to discretization parameter $h$.

\appendix 
\label{app:appendix}
\section{Proofs of Lemma ~\ref{lem:transmission condition} and intermediate results in Section~\ref{sec: Carleman estimate uniform mesh}}

\subsection{Proof of Lemma ~\ref{lem:transmission condition}}
\par  We have
\begin{equation}(c_dDu)_{n+\frac{3}{2}}-(c_dDu)_{n+\frac{1}{2}}=hf_{n+1}.\nonumber\end{equation}
\par As $Du=\tilde{\rho}Dv+D\rho\tilde{v}$ we obtain
\small{
\begin{eqnarray}\label{eq:3.3}&&r_{n+1}\Big(\tilde{\rho}_{n+\frac{3}{2}}(cDv)_{n+\frac{3}{2}}-\tilde{\rho}_{n+\frac{1}{2}}(cDv)_{n+\frac{1}{2}}+(D\rho)_{n+\frac{3}{2}}(c\tilde{v})_{n+\frac{3}{2}}-(D\rho)_{n+\frac{1}{2}}(c\tilde{v})_{n+\frac{1}{2}}\Big)\nonumber\\
&&=h(rf)_{n+1}.\end{eqnarray}
}
\par We write

\begin{equation}r_{n+1}\tilde{\rho}_{n+\frac{3}{2}}=\frac{r_{n+1}\rho_{n+1}+r_{n+1}\rho_{n+2}}{2}=\frac{1+(((\tau^{+})^2\rho)r)_{n+1}}{2}:=K_{11},\nonumber\end{equation}
\begin{eqnarray}r_{n+1}(c_dD\rho)_{n+\frac{3}{2}}&=&(r\tau^+\rho )_{n+1}(c_drD\rho )_{n+\frac{3}{2}}\nonumber\\
&=&(r\tau^+\rho )_{n+1}\Big({(c_dr\partial\rho)_{n+\frac{3}{2}}+(c_drD\rho )_{n+\frac{3}{2}}-(c_dr\partial\rho)_{n+\frac{3}{2}}}\Big)\nonumber\\
&=&K_{21}\Big({(c_dr\partial\rho)_{n+\frac{3}{2}}+K_{22}}\Big),\nonumber
\end{eqnarray}
where $K_{21}=(r\tau^+\rho )_{n+1}$ and $K_{22}=(c_drD\rho )_{n+\frac{3}{2}}-(c_dr\partial\rho)_{n+\frac{3}{2}}$.
\par Similarly,
\begin{equation}r_{n+1}\tilde{\rho}_{n+\frac{1}{2}}=\frac{r_{n+1}\rho_{n+1}+r_{n+1}\rho_{n}}{2}=\frac{1+(((\tau^{-})^2\rho)r)_{n+1}}{2}:=K_{31},\nonumber
\end{equation}
\begin{eqnarray}
r_{n+1}(c_dD\rho)_{n+\frac{1}{2}}&=&(r\tau^-\rho )_{n+1}(c_drD\rho )_{n+\frac{1}{2}}\nonumber\\
&=&(r\tau^-\rho )_{n+1}\Big({(c_dr\partial\rho)_{n+\frac{1}{2}}+(c_drD\rho )_{n+\frac{1}{2}}-(c_dr\partial\rho)_{n+\frac{1}{2}}}\Big)\nonumber\\
&:=&K_{41}\Big({(c_dr\partial\rho)_{n+\frac{1}{2}}+K_{42}}\Big),\nonumber
\end{eqnarray}
where $K_{41}=(r\tau^-\rho )_{n+1}$ and $(c_drD\rho )_{n+\frac{1}{2}}-(c_dr\partial\rho)_{n+\frac{1}{2}}$.
\par Additionally,
\begin{eqnarray}&&(\tilde{v})_{n+\frac{1}{2}}=v_{n+1}+\frac{v_{n}-v_{n+1}}{2}=v_{n+1}+\mathcal{O}(h)(Dv)_{n+\frac{1}{2}},\nonumber\\
&&(\tilde{v})_{n+\frac{3}{2}}=v_{n+1}+\frac{v_{n+2}-v_{n+1}}{2}=v_{n+1}+\mathcal{O}(h)(Dv)_{n+\frac{3}{2}}.\nonumber
\end{eqnarray}

\par From ~\eqref{eq:3.3} we thus write
\begin{eqnarray}&&K_{11}(c_dDv)_{n+\frac{3}{2}}-K_{31}(c_dDv)_{n+\frac{1}{2}}\nonumber\\
&+&K_{21}\Big({(cr\partial\rho)_{n+\frac{3}{2}}+K_{22}}\Big)\Big({v_{n+1}+\mathcal{O}(h)(Dv)_{n+\frac{3}{2}}}\Big)\nonumber\\
&-&K_{41}\Big({(cr\partial\rho)_{n+\frac{1}{2}}+K_{42}}\Big)\Big({v_{n+1}+\mathcal{O}(h)(Dv)_{n+\frac{1}{2}}}\Big)\nonumber\\
&=&h(rf)_{n+1}.\nonumber
\end{eqnarray}

\par Then
\begin{eqnarray}&&K_{11}\Big((c_dDv)_{n+\frac{3}{2}}-(c_dDv)_{n+\frac{1}{2}}\Big)+(K_{11}-K_{31})(c_dDv)_{n+\frac{1}{2}}\nonumber\\
&+&K_{21}[\star cr\partial \rho]_a v_{n+1}+(K_{21}-K_{41})(cr\partial \rho)_{n+\frac{1}{2}}v_{n+1}+(K_{21}K_{22}-K_{41}K_{42})v_{n+1}\nonumber\\
&+&K_{21}\Big({(cr\partial\rho)_{n+\frac{3}{2}}+K_{22}}\Big)\mathcal{O}(h)(Dv)_{n+\frac{3}{2}}+K_{41}\Big({(cr\partial\rho)_{n+\frac{1}{2}}+K_{42}}\Big)\mathcal{O}(h)(Dv)_{n+\frac{1}{2}}\nonumber\\
&=&h(rf)_{n+1}.\nonumber
\end{eqnarray}

\par Moreover, as $r\partial\rho=-\lambda s\phi\partial\psi=s\mathcal{O_{\lambda}}(1)$ we have
\small{
\begin{eqnarray}&&K_{11}\Big((c_dDv)_{n+\frac{3}{2}}-(c_dDv)_{n+\frac{1}{2}}\Big)+(K_{11}-K_{31})(c_dDv)_{n+\frac{1}{2}}\nonumber\\
&=&K_{21}\lambda s[\star c\phi \partial \psi]_a v_{n+1}-Kv_{n+1}+\Big(K_{21}\mathcal{O_{\lambda}}(sh)+K_{21}K_{22}\mathcal{O}(h)\Big)\Big((c_dDv)_{n+\frac{3}{2}}-(c_dDv)_{n+\frac{1}{2}}\Big)\nonumber\\
&+&\Big({K_{21}\mathcal{O_{\lambda}}(sh)+K_{21}K_{22}\mathcal{O}(h)+K_{41}\mathcal{O_{\lambda}}(sh)+K_{41}K_{42}\mathcal{O}(h)}\Big)(c_dDv)_{n+\frac{1}{2}}\nonumber\\
&+&h(rf)_{n+1},\nonumber
\end{eqnarray}
}
where 
\begin{eqnarray}K&=&(K_{21}-K_{41})(cr\partial\rho)_{n+\frac{1}{2}}+K_{21}K_{22}-K_{41}K_{42}\nonumber\\
&=&(K_{21}-K_{41})s\mathcal{O_{\lambda}}(1)+K_{21}K_{22}-K_{41}K_{42}.\nonumber
\end{eqnarray}
\par From that, we can write
\begin{eqnarray}&&L\Big((c_dDv)_{n+\frac{3}{2}}-(c_dDv)_{n+\frac{1}{2}}\Big)\nonumber\\
&=&K_{21}\lambda s[\star c\phi \partial \psi]_a v_{n+1}-Kv_{n+1}+H(c_dDv)_{n+\frac{1}{2}}+h(rf)_{n+1},\nonumber
\end{eqnarray}
where
\begin{eqnarray}&&L=K_{11}-K_{21}\mathcal{O_{\lambda}}(sh)-K_{21}K_{22}\mathcal{O}(h),\nonumber\\
&&K=(K_{21}-K_{41})s\mathcal{O_{\lambda}}(1)+K_{21}K_{22}-K_{41}K_{42},\nonumber\\
&&H=K_{21}\mathcal{O_{\lambda}}(sh)+K_{21}K_{22}\mathcal{O}(h)+K_{41}\mathcal{O_{\lambda}}(sh)+K_{41}K_{42}\mathcal{O}(h)-K_{11}+K_{31}.\nonumber
\end{eqnarray}
\par As $L=1+\mathcal{O_{\lambda, \mathfrak{K}}}(sh)\neq 0$ (see below) then we read
\begin{eqnarray}&&(c_dDv)_{n+\frac{3}{2}}-(c_dDv)_{n+\frac{1}{2}}\nonumber\\
&=&\Big({L^{-1}K_{21}\lambda s[\star c\phi \partial \psi]_a -L^{-1}K}\Big)v_{n+1}+L^{-1}H(c_dDv)_{n+\frac{1}{2}}+L^{-1}h(rf)_{n+1}.\nonumber
\end{eqnarray}
\par We set 
\begin{eqnarray}&&J_1=L^{-1}K_{21}\lambda s[\star c\phi \partial \psi]_a -L^{-1}K,\nonumber\\
&&J_2=L^{-1}H, \quad\quad\quad\quad\quad\quad\quad\quad\quad\quad\quad\quad\quad\quad\quad\quad\quad\quad\quad J_3=L^{-1}.\nonumber
\end{eqnarray}
\par We thus have
\small{
\begin{equation}\label{eq:3.4}(c_dDv)_{n+\frac{3}{2}}-(c_dDv)_{n+\frac{1}{2}}=J_1v_{n+1}+J_2(c_dDv)_{n+\frac{1}{2}}+J_3h(rf)_{n+1}
\end{equation}
}

\par By using Proposition~\ref{prop:property 1} we find
\begin{eqnarray}&&K_{11}=\frac{1+(((\tau^{+})^2\rho)r)_{n+1}}{2}=1+\mathcal{O_{\lambda, \mathfrak{K}}}(sh),\nonumber\\
&&K_{31}=\frac{1+(((\tau^{-})^2\rho)r)_{n+1}}{2}=1+\mathcal{O_{\lambda, \mathfrak{K}}}(sh),\nonumber\\
&&K_{21}=(r\tau^+\rho)_{n+1}=1+\mathcal{O_{\lambda, \mathfrak{K}}}(sh),\nonumber\\
&&K_{41}=(r\tau^-\rho)_{n+1}=1+\mathcal{O_{\lambda, \mathfrak{K}}}(sh),\nonumber\\
&&K_{22}=(c_drD\rho)_{n+\frac{3}{2}}-(c_dr\partial\rho)_{n+\frac{3}{2}}=s\mathcal{O_{\lambda, \mathfrak{K}}}(sh)^2,\nonumber\\
&&K_{42}=(c_drD\rho)_{n+\frac{1}{2}}-(c_dr\partial\rho)_{n+\frac{1}{2}}=s\mathcal{O_{\lambda, \mathfrak{K}}}(sh)^2.\nonumber
\end{eqnarray}
\par From that we estimate
\begin{eqnarray}&&K=(K_{21}-K_{41})s\mathcal{O_{\lambda}}(1)+K_{21}K_{22}-K_{41}K_{42}=s\mathcal{O_{\lambda, \mathfrak{K}}}(sh),\nonumber\\
&&H=K_{21}\mathcal{O_{\lambda}}(sh)+K_{21}K_{22}\mathcal{O}(h)+K_{41}\mathcal{O_{\lambda}}(sh)+K_{41}K_{42}\mathcal{O}(h)-K_{11}+K_{31}\nonumber\\
&&\ \ \ \ =\mathcal{O_{\lambda, \mathfrak{K}}}(sh),\nonumber\\
&&L=K_{11}-K_{21}\mathcal{O_{\lambda}}(sh)+K_{21}K_{22}\mathcal{O}(h)=1+\mathcal{O_{\lambda, \mathfrak{K}}}(sh).\nonumber
\end{eqnarray}
For $sh$ sufficiently small we have $L^{-1}=1+\mathcal{O_{\lambda, \mathfrak{K}}}(sh)$ and then we obtain
\begin{eqnarray}
&&J_1=L^{-1}K_{21}\lambda s[\star c\phi \partial \psi]_a -L^{-1}K\nonumber\\
&&\ \ \ \ =\Big(1+\mathcal{O_{\lambda, \mathfrak{K}}}(sh)\Big)\lambda s[\star c\phi \partial \psi]_a+s\mathcal{O_{\lambda, \mathfrak{K}}}(sh),\nonumber\\
&&J_2=L^{-1}H=\mathcal{O_{\lambda, \mathfrak{K}}}(sh),\nonumber\\
&&J_3=L^{-1}=1+\mathcal{O_{\lambda, \mathfrak{K}}}(sh).\nonumber
\end{eqnarray}
\par By using Proposition~\ref{prop:property 5}, Lemma~\ref{lem:derivative wrt t} and Lemma~\ref{lem:traslate function} yield
\begin{eqnarray}&&\partial_t K_{11}=\partial_t\Big({(((\tau^{+})^2\rho)r)_{n+1}}\Big)=T\theta(t)\mathcal{O_{\lambda, \mathfrak{K}}}(sh),\nonumber\\
&&\partial_t K_{31}=\partial_t\Big({(((\tau^{-})^2\rho)r)_{n+1}}\Big)=T\theta(t)\mathcal{O_{\lambda, \mathfrak{K}}}(sh),\nonumber\\
&&\partial_t K_{21}=\partial_t(r\tau^+\rho)_{n+1}=T\theta(t)\mathcal{O_{\lambda, \mathfrak{K}}}(sh),\nonumber\\
&&\partial_t K_{41}=\partial_t(r\tau^-\rho)_{n+1}=T\theta(t)\mathcal{O_{\lambda, \mathfrak{K}}}(sh),\nonumber\\
&&\partial_t K_{22}=\partial_t\Big((c_drD\rho)_{n+\frac{3}{2}}-(c_dr\partial\rho)_{n+\frac{3}{2}}\Big)=sT\theta(t)\mathcal{O_{\lambda, \mathfrak{K}}}(sh)^2,\nonumber\\
&&\partial_t K_{42}=\partial_t\Big((c_drD\rho)_{n+\frac{1}{2}}-(c_dr\partial\rho)_{n+\frac{1}{2}}\Big)=sT\theta(t)\mathcal{O_{\lambda, \mathfrak{K}}}(sh)^2,\nonumber
\end{eqnarray} 
which give
\begin{eqnarray}\partial_tL^{-1}=-\frac{\partial_t L}{L^2}&=&\Big(1+\mathcal{O_{\lambda, \mathfrak{K}}}(sh)\Big)\left(\partial_tK_{11}+\partial_tK_{21}\mathcal{O_{\lambda}}(sh)+K_{21}(\partial_t s)\mathcal{O_{\lambda}}(h)\right.\nonumber\\
&&\left.\quad\quad\quad\quad\quad\quad\quad\quad+\partial_tK_{21}K_{22}\mathcal{O}(h)+\partial_tK_{22}K_{21}\mathcal{O}(h)\right)\nonumber\\
&=&T\theta(t)\mathcal{O_{\lambda, \mathfrak{K}}}(sh),\nonumber
\end{eqnarray}
where $sh$ sufficiently small and
\small{
\begin{eqnarray}&&\partial_tH=\partial_tK_{21}\mathcal{O_{\lambda}}(sh)+K_{21}(\partial_t s)\mathcal{O_{\lambda}}(h)+\partial_tK_{21}K_{22}\mathcal{O}(h)+K_{21}\partial_tK_{22}\mathcal{O}(h)\nonumber\\
&+&\partial_tK_{41}\mathcal{O_{\lambda}}(sh)+K_{41}(\partial_t s)\mathcal{O_{\lambda}}(h)+\partial_tK_{41}K_{42}\mathcal{O}(h)+K_{41}\partial_tK_{42}\mathcal{O}(h)-\partial_t K_{11}+\partial_t K_{31}\nonumber\\
&=&T\theta(t)\mathcal{O_{\lambda, \mathfrak{K}}}(sh).\nonumber
\end{eqnarray}
}
\par It follows that we have
\begin{eqnarray}&&\partial_t J_1=sT\theta(t)\mathcal{O_{\lambda, \mathfrak{K}}}(sh),\nonumber\\
&&\partial_t J_2=T\theta(t)\mathcal{O_{\lambda, \mathfrak{K}}}(sh),\quad\quad\quad\quad\quad\quad\partial_t J_3=T\theta(t)\mathcal{O_{\lambda, \mathfrak{K}}}(sh).\nonumber 
\end{eqnarray}
\par Furthermore, we can write ~\eqref{eq:3.4} in the simple form
\begin{eqnarray}&&(c_dDv)_{n+\frac{3}{2}}-(c_dDv)_{n+\frac{1}{2}}\nonumber\\
&=&\lambda s[\star c\phi \partial \psi]_a v_{n+1}+\lambda s\mathcal{O_{\lambda, \mathfrak{K}}}(sh)v_{n+1}\nonumber\\
&+&\mathcal{O_{\lambda, \mathfrak{K}}}(sh)(c_dDv)_{n+\frac{1}{2}}+\Big(1+\mathcal{O_{\lambda, \mathfrak{K}}}(sh)\Big)h(rf)_{n+1},\nonumber
\end{eqnarray}
which yields the conclusion.

\subsection{Proof of Lemma~\ref{lem: estimate I11}}

\par By using Lemma~\ref{lemma: discrete leibnitz} in each domain ${{\Omega}_{01}} $, ${{\Omega}_{02}} $ , we have
\begin{eqnarray} I_{11}&=&2\int_{Q'_0}{cr^2\bar{\tilde{\rho}}\ {\overline{D \rho}}\bar{D}(c_d Dv)\overline{Dv}}\nonumber\\
&=&2\int_{Q_{01}}{cr^2\bar{\tilde{\rho}}\ {\overline{D \rho}}\bar{D}(c_d Dv)\overline{Dv}}+2\int_{Q_{02}}{cr^2\bar{\tilde{\rho}}\ {\overline{D \rho}}\bar{D}(c_d Dv)\overline{Dv}}\nonumber\\
&=&2\int_{Q_{01}}{cr^2\bar{\tilde{\rho}}\ \overline{D \rho}\bar{c}_d\bar{D}(Dv)\overline{Dv}}+2\int_{Q_{01}}{cr^2}\bar{\tilde{\rho}}\ \overline{D \rho}(\bar{D}c_d)(\overline{Dv})^2\nonumber\\
&+&2\int_{Q_{02}}{cr^2\bar{\tilde{\rho}}\ \overline{D \rho}\bar{c}_d\bar{D}(Dv)\overline{Dv}}+2\int_{Q_{02}}{cr^2}\bar{\tilde{\rho}}\ \overline{D \rho}(\bar{D}c_d)(\overline{Dv})^2\nonumber\\
&=&\sum_{i=1}^2\int_{Q_{0i}}{cr^2\bar{\tilde{\rho}}\ \overline{D \rho}\bar{c}_d\bar{D}(Dv)^2}+2\sum_{i=1}^2\int_{Q_{0i}}{cr^2}\bar{\tilde{\rho}}\ \overline{D \rho}(\bar{D}c_d)(\overline{Dv})^2.\nonumber
\end{eqnarray}

\par We then apply a discrete integration by parts (Proposition~\ref{prop:IVP}) in each domain ${{\Omega}_{01}} $, ${{\Omega}_{02}} $ with $\partial\Omega_{01}=\left\{0,a\right\}$ and $\partial\Omega_{02}=\left\{a,1\right\}$ for the first two terms and we obtain
\begin{eqnarray}I_{11}&=&-\sum_{i=1}^2\int_{Q_{0i}}{D(c\bar{c_d}r^2\bar{\tilde{\rho}}\ \overline{D \rho})(Dv)^2}+2\sum_{i=1}^2\int_{Q_{0i}}{cr^2}\bar{\tilde{\rho}}\ \overline{D \rho}(\bar{D}c_d)(\overline{Dv})^2\nonumber\\
&+&\int_0^T{(c\bar{c}_d r^2\bar{\tilde{\rho}}\ \overline{D \rho})(1)(Dv)^2_{n+m+\frac{3}{2}}}-\int_0^T{(c\bar{c}_d r^2\bar{\tilde{\rho}}\ \overline{D \rho})(a^+)(Dv)^2_{n+\frac{3}{2}}}\nonumber\\
&+&\int_0^T(c\bar{c}_dr^2\bar{\tilde{\rho}}\ \overline{D \rho})(a^-)(Dv)^2_{n+\frac{1}{2}}-\int_0^T{(c\bar{c}_d r^2\bar{\tilde{\rho}}\ \overline{D \rho})(0)(Dv)^2_{\frac{1}{2}}}\nonumber\\
&=&-\sum_{i=1}^2\int_{Q_{0i}}{D(c\bar{c}_dr^2\bar{\tilde{\rho}}\ \overline{D \rho})(Dv)^2}+2\sum_{i=1}^2\int_{Q_{0i}}{cr^2\bar{\tilde{\rho}}\ \overline{D \rho}(\bar{D}c_d)(\overline{Dv})^2}+Y_{11}.\nonumber
\end{eqnarray}
where
\begin{eqnarray}&&Y_{11}=Y_{11}^{(1)}+Y_{11}^{(2)}\nonumber\\
&&Y_{11}^{(1)}=\int_0^T{(c\bar{c}_d r^2\bar{\tilde{\rho}}\ \overline{D \rho})(1)(Dv)^2_{n+m+\frac{3}{2}}}-\int_0^T{(c\bar{c}_d r^2\bar{\tilde{\rho}}\ \overline{D \rho})(a^+)(Dv)^2_{n+\frac{3}{2}}}\nonumber\\
&&Y_{11}^{(2)}=\int_0^T(c\bar{c}_dr^2\bar{\tilde{\rho}}\ \overline{D \rho})(a^-)(Dv)^2_{n+\frac{1}{2}}-\int_0^T{(c\bar{c}_d r^2\bar{\tilde{\rho}}\ \overline{D \rho})(0)(Dv)^2_{\frac{1}{2}}}\nonumber
\end{eqnarray}

\begin{Lemma} (see Lemma B.3 in [BHL10a]) Provided $sh\le\mathfrak{K}$ we have
\begin{eqnarray}&&D(c_i\bar{c}_{di}r^2_i\bar{\tilde{\rho}}_i\ \overline{D \rho_i})=-s\lambda^2(c^2_i\phi_i(\psi'_i)^2)_d+s\lambda\phi_{di} \mathcal{O}(1)+s\mathcal{O_{\lambda, \mathfrak{K}}}(sh),\nonumber\\
&&{c_ir^2_i}\bar{\tilde{\rho}}_i\overline{D \rho_i}(\bar{D}c_{di})=s\lambda\phi_i\mathcal{O}(1)+s\mathcal{O_{\lambda, \mathfrak{K}}}((sh)^2),\nonumber\\
&& r^2_i\bar{\tilde{\rho}}_i\overline{D \rho_i}=r_i\partial \rho_i+s\mathcal{O_{\lambda, \mathfrak{K}}}((sh)^2)=-s\lambda\phi_i\psi'_i+s\mathcal{O_{\lambda, \mathfrak{K}}}((sh)^2),\nonumber\\
&&r^2_i\bar{\tilde{\rho}}_i\overline{D \rho_i}=r_i\bar{\tilde{\rho}}_ir_i\overline{D\rho_i}=(1+\mathcal{O_{\lambda, \mathfrak{K}}}(sh))r_i\overline{D\rho_i}.\nonumber\end{eqnarray}
\end{Lemma}

\par Moreover, by Lemma~\ref{lem:average of two function} and Proposition~\ref{prop:IVP} in each domain $\Omega_{01}$, $\Omega_{02}$ we obtain
\begin{equation}\int_{\Omega'_{0}}{s\lambda \phi({\overline{Dv}})^2}\le \int_{\Omega'_{0}}{s\lambda \phi\overline{(Dv)^2}}=\int_{\Omega'_{0}}{s\lambda \tilde{\phi}(Dv)^2}-\frac{h}{2} \sum_{i=1}^2 BT_i\le\int_{\Omega'_{0}}{s\lambda \tilde{\phi}(Dv)^2} \nonumber
\end{equation}
since
\begin{eqnarray}&&BT_1=s\lambda\phi(a)(Dv)^2_{n+\frac{1}{2}}+s\lambda\phi(a)(Dv)^2_{\frac{1}{2}}\ge 0\nonumber\\
&&BT_2=s\lambda\phi(1)(Dv)^2_{n+m+\frac{3}{2}}+s\lambda\phi(a)(Dv)^2_{n+\frac{3}{2}}\ge 0\nonumber
\end{eqnarray}
and $\tilde{\phi}=\phi+h^2\mathcal{O_{\lambda}}(1)$ then we can write
\begin{equation}\int_{\Omega_{0i}}{s\lambda \phi({\overline{Dv}})^2}\le\int_{\Omega_{0i}}{s\lambda \phi(Dv)^2}+\int_{\Omega_{0i}}{s\lambda h^2\mathcal{O_{\lambda}}(1)(Dv)^2}\nonumber
\end{equation}

\par Similarly, we have
\begin{equation}\left|\int_{\Omega_{0i}}{s\mathcal{O_{\lambda, \mathfrak{K}}}(sh)^2({\overline{Dv}})^2}\right|\le \int_{\Omega'}{s\left|\mathcal{O_{\lambda, \mathfrak{K}}}(sh)^2\right|\overline{(Dv)^2}}\le \int_{\Omega'}{s\left|\mathcal{O_{\lambda, \mathfrak{K}}}(sh)^2\right|(Dv)^2}.\nonumber\end{equation}
                                                                                                                                                                                                                                     
\par Thus 
\begin{equation}I_{11}\ge -\int_{Q'}{s\lambda^2}(c^2\phi(\psi')^2)_d (Dv)^2-X_{11}+Y_{11},\nonumber\end{equation}
where $X_{11}=\int_{Q'_0}{\nu_{11}(Dv)^2}$ with $\nu_{11}$ of the form $s\lambda\phi \mathcal{O}(1)+s\mathcal{O_{\lambda, \mathfrak{K}}}(sh)$
and
\begin{eqnarray} &&Y_{11}=Y_{11}^{(1)}+Y_{11}^{(2,1)}+Y_{11}^{(2,2)},\nonumber\\
&& Y_{11}^{(1)}=\int_0^T\big(1+\mathcal{O_{\lambda, \mathfrak{K}}}(sh)\big)(c\bar{c}_d)(1)(r\overline{D\rho})(1)(Dv)^2_{n+m+\frac{3}{2}}\nonumber\\
&&\quad \quad-\int_0^T \big(1+\mathcal{O_{\lambda, \mathfrak{K}}}(sh)\big)(c\bar{c}_d)(0)(r\overline{D\rho})(0)(Dv)^2_{\frac{1}{2}},\nonumber\\
&&Y_{11}^{(2,1)}=\int_0^T{s\lambda\phi(a)\bar{c}_d(a)\Big(-(c\psi')(a^-)(Dv)^2_{n+\frac{1}{2}}+(c\psi')(a^+)(Dv)^2_{n+\frac{3}{2}}\Big)},\nonumber\\
&&Y_{11}^{(2,2)}=\int_0^T{s\mathcal{O_{\lambda, \mathfrak{K}}}(sh)^2(Dv)^2_{n+\frac{1}{2}}}-\int_0^T{s\mathcal{O_{\lambda, \mathfrak{K}}}(sh)^2(Dv)^2_{n+\frac{3}{2}}}.\nonumber\end{eqnarray}

\subsection{Proof of Lemma~\ref{lem: estimate I12}}
\par We set $q=r\bar{\tilde{\rho}}c\phi^{''}$. By using a discrete integrations by parts (Proposition~\ref{prop:IVP}) and Lemma~\ref{lemma: discrete leibnitz} in each domain ${{\Omega}_{01}}$, ${{\Omega}_{02}}$ we have

\begin{eqnarray} I_{12}&=&-2\sum_{i=1}^2\int_{Q_{0i}}{sqv\bar{D}(c_dDv)}\nonumber\\
&=&2\sum_{i=1}^2\int_{Q_{0i}}{s\tilde{q}c_d(Dv)^2}+2\sum_{i=1}^2\int_{Q_{0i}}{sDqc_d\tilde{v}Dv}\nonumber\\
&-&\int_0^T{sq(a^-)v(a)(c_dDv)_{n+\frac{1}{2}}}+\int_0^T{sq(a^+)v(a)(c_dDv)_{n+\frac{3}{2}}}\nonumber\\
&=&2\sum_{i=1}^2\int_{Q_{0i}}{s\tilde{q}c_d(Dv)^2}+2\sum_{i=1}^2\int_{Q_{0i}}{sDqc_d\tilde{v}Dv}dt+Y_{12},\nonumber
\end{eqnarray}\\
since $v|_{\partial \Omega_0}=0$ and with $\partial\Omega_{01}=\left\{{0,a}\right\}$, $\partial\Omega_{02}=\left\{{a,1}\right\}$.
\begin{Lemma} \label{lem:A2}(see the proof as given  in Lemma 4.4 of \cite{BHL10a}) Let $i=1,2$. Provided $sh\le \mathfrak{K}$ we have
\begin{eqnarray} &&\phi''_{i}=\lambda^2\phi_{i}(\psi'_{i})^2+\lambda\phi_{i}\mathcal{O}(1),\nonumber\\
&&q_{i}=r_{i}\bar{\tilde{\rho_{i}}}c \phi^{''}_{i}=\lambda^2c\phi_{i} (\psi'_{i})^2+\lambda\phi_{i}\mathcal{O}(1)+\mathcal{O_{\lambda, \mathfrak{K}}}(sh)^2,\nonumber\\
&&\tilde{q_{i}}=\lambda^2 (c\phi_{i}(\psi'_{i})^2)_d+\lambda\phi_{i}\mathcal{O}(1)+\mathcal{O_{\lambda, \mathfrak{K}}}((sh)^2+h),\nonumber\\
&&Dq_{i}=D(r_{i}\bar{\tilde{\rho_{i}}})\wtilde{c\phi^{''}_i}+\wtilde{(r_{i}\overline{\tilde{\rho_{i}}})}D(c\phi^{''}_{i})=\mathcal{O_{\lambda, \mathfrak{K}}}(1).\nonumber\end{eqnarray}
\end{Lemma} 
\par Note that the proof and the use of Lemma~\ref{lem:A2} are carried out in each domain ${{\Omega}_{01}}$, ${{\Omega}_{02}}$ independently.
\par It follows that
\small{
\begin{equation} I_{12}= 2\sum_{i=1}^2\int_{Q_{0i}}{s\lambda^2(  c^2\phi(\psi')^2)_d (Dv)^2}+\sum_{i=1}^2\int_{Q_{0i}}{\nu_{12}(Dv)^2}+\sum_{i=1}^2\int_{Q_{0i}}{s\mathcal{O_{\lambda, \mathfrak{K}}}(1)\tilde{v}Dv}+Y_{12},\nonumber
\end{equation}
}
\par Then
\begin{equation}I_{12}= 2\int_{Q'_0}{s\lambda^2(c^2\phi(\psi')^2)_d(Dv)^2}-X_{12}+Y_{12},\nonumber
\end{equation}
with 

\begin{eqnarray} Y_{12}&=&\int_0^T{s \lambda^2 \phi(a)v(a)[c(\psi')^2\star c_dDv]_a}\nonumber\\
&+&\int_0^T{\delta_{12}v(a)(cDv)_{n+\frac{3}{2}}+\bar{\delta}_{12}v(a)(cDv)_{n+\frac{1}{2}}},\nonumber
\end{eqnarray}
where $\delta_{12}, \bar{\delta}_{12}$ are of form $s\big(\lambda\phi(a)\mathcal{O}(1)+\mathcal{O_{\lambda, \mathfrak{K}}}(sh)^2\big)$ and
\begin{equation}X_{12}=\int_{Q'_0}{\nu_{12}(Dv)^2}+\int_{Q'_0}{s\mathcal{O_{\lambda, \mathfrak{K}}}(1)\tilde{v}Dv},\nonumber
\end{equation}
where
\begin{equation}\nu_{12}=s\lambda \phi \mathcal{O}(1)+s\mathcal{O_{\lambda, \mathfrak{K}}}(h+(sh)^2).\nonumber
\end{equation}

\subsection{Proof of Lemma~\ref{lem:estimate I13}}
\par  We carry out a discrete integration by parts (Proposition~\ref{prop:IVP}) in each domain ${{\Omega}_{01}}$, ${{\Omega}_{02}}$ with $\partial\Omega_{01}=\left\{{0,a}\right\}$ and $\partial\Omega_{02}=\left\{{a,1}\right\}$ as follows
\begin{eqnarray} I_{13}&=&\int_{Q_{01}}{r\bar{\tilde{\rho}}\bar{D}(c_d Dv)\partial_t v }+\int_{Q_{02}}{r\bar{\tilde{\rho}}\bar{D}(c_d Dv)\partial_t v }\nonumber\\
&=&-\int_{Q_{01}}{D(r\bar{\tilde{\rho}}\partial_t v)c_d Dv }-\int_{Q_{02}}{D(r\bar{\tilde{\rho}}\partial_t v)c_d Dv }\nonumber\\
&+&\int_0^T{(r\bar{\tilde{\rho}})(a^-)\partial_t v(a)(c_d Dv)_{n+\frac{1}{2}}}-\int_0^T{(r\bar{\tilde{\rho}})(0)\partial_t v(0)(c_d Dv)_{\frac{1}{2}}}\nonumber\\
&+&\int_0^T{(r\bar{\tilde{\rho}})(1)\partial_t v(1)(c_d Dv)_{n+m+\frac{3}{2}}}-\int_0^T{(r\bar{\tilde{\rho}})(a^+)\partial_t v(a)(c_d Dv)_{n+\frac{3}{2}}}\nonumber\\
&=&-\int_{Q_{01}}{D(r\bar{\tilde{\rho}}\partial_t v)c_d Dv }-\int_{Q_{02}}{D(r\bar{\tilde{\rho}}\partial_t v)c_d Dv }\nonumber\\
&+&\int_0^T{(r\bar{\tilde{\rho}})(a^-)\partial_t v(a)(c_d Dv)_{n+\frac{1}{2}}}-\int_0^T{(r\bar{\tilde{\rho}})(a^+)\partial_t v(a)(c_d Dv)_{n+\frac{3}{2}}}\nonumber\\
&=&\underbrace{-\sum_{i=1}^2\int_{Q_{0i}}{D(r\bar{\tilde{\rho}})\partial_t \tilde{v}c_dDv}}_{\bar{Q}_1} \underbrace{-\sum_{i=1}^2\int_{Q_{0i}}{\wtilde{r\bar{\tilde{\rho}}}(\partial_tDv)c_dDv}}_{\bar{Q}_2}+Y_{13},\nonumber
\end{eqnarray}
by Lemma~\ref{lemma: discrete leibnitz} and with
\begin{equation}Y_{13}=\int_0^T{(r\bar{\tilde{\rho}})(a^-)\partial_t v(a)(c_d Dv)_{n+\frac{1}{2}}}-\int_0^T{(r\bar{\tilde{\rho}})(a^+)\partial_t v(a)(c_d Dv)_{n+\frac{3}{2}}},\nonumber
\end{equation}
as $v|_{\partial\Omega_0}=0$.
\par By applying Proposition~\ref{prop:property 4} in each domain ${{\Omega}_{01}}$, ${{\Omega}_{02}}$ we find
\begin{eqnarray}&&D(r_{i}\bar{\tilde{\rho_{i}}})=\mathcal{O_{\lambda, \mathfrak{K}}}(sh),\nonumber\\
&&r_{i}\bar{\tilde{\rho_{i}}}=1+\mathcal{O_{\lambda, \mathfrak{K}}}(sh)^2=\mathcal{O_{\lambda, \mathfrak{K}}}(1).\nonumber\end{eqnarray}

\par On the one hand, we have
\begin{eqnarray}\left|{\bar{Q}_1}\right| &\le& \sum_{i=1}^2\int_{Q_{0i}}{s^{-1}\mathcal{O_{\lambda, \mathfrak{K}}}(sh)(\partial_t \tilde{v})^2}+\sum_{i=1}^2\int_{Q_{0i}}{s\mathcal{O_{\lambda, \mathfrak{K}}}(sh)(Dv)^2}\nonumber\\
&\le&  \sum_{i=1}^2\int_{Q_{0i}}{s^{-1}\mathcal{O_{\lambda, \mathfrak{K}}}(sh)\wtilde{(\partial_t v)^2}}+\sum_{i=1}^2\int_{Q_{0i}}{s\mathcal{O_{\lambda, \mathfrak{K}}}(sh)(Dv)^2}\nonumber\\
&=&  \int_{Q_0}{s^{-1}\mathcal{O_{\lambda, \mathfrak{K}}}(sh)(\partial_t v)^2}+\int_{Q'_0}{s\mathcal{O_{\lambda, \mathfrak{K}}}(sh)(Dv)^2},\nonumber\end{eqnarray}
by $({\partial_t \tilde{v}})^2\le \wtilde{({\partial_t \tilde{v}})^2}$ in each domain ${{\Omega}_{01}}$, ${{\Omega}_{02}}$ and $ \small{\sum_{i=1}^2\int_{\Omega_{0i}}{\mathcal{O_{\lambda, \mathfrak{K}}}(1)\wtilde{({\partial_t \tilde{v}})^2}}= \int_{\Omega_0}{\mathcal{O_{\lambda, \mathfrak{K}}}(1)({\partial_t \tilde{v}})^2}}$.
\par On the other hand, by an integrations by parts w.r.t $t$ we write as
\begin{eqnarray} \bar{Q}_2&=&-\frac{1}{2} \sum_{i=1}^2\int_{Q_{0i}}{\wtilde{r\bar{\tilde{\rho}}}c_d.\partial_t(Dv)^2}\nonumber\\
&=&\frac{1}{2} \sum_{i=1}^2\int_{Q_{0i}}{\partial_t(\wtilde{r\bar{\tilde{\rho}}})c_d(Dv)^2}-\frac{1}{2} \sum_{i=1}^2\int_{\Omega_{0i}}{\wtilde{r\bar{\tilde{\rho}}}c_d.(Dv)^2|_{t=0}^{t=T}}.\nonumber\end{eqnarray}
\par We observe that for  $sh\le \epsilon_1(\lambda)$ with $\epsilon_1(\lambda)$ sufficiently small we have $r\bar{\tilde{\rho}}>0$ by Proposition~\ref{prop:property 4}. The sign of the term at $t=T$ and $t=0$ are thus prescribed. Furthermore, Proposition~\ref{prop:property 5} leads to $ \partial_t({r_{i}\bar{\tilde{\rho_{i}}}})=T(sh)^2\theta\mathcal{O_{\lambda, \mathfrak{K}}}(1)$, so that, for $sh\le\mathfrak{K}$ we obtain

\begin{equation} \bar{Q}_2\ge\sum_{i=1}^2\int_{Q_{0i}}{T(sh)^2\theta\mathcal{O_{\lambda, \mathfrak{K}}}(1)(Dv)^2}-\mathcal{C_{\lambda, \mathfrak{K}}}(1)\sum_{i=1}^2\int_{\Omega_{0i}}{(Dv(T))^2}.\nonumber\end{equation}

\par Thus, 
\begin{eqnarray} I_{13} \ge -\int_{\Omega'_0}\mathcal{C_{\lambda, \mathfrak{K}}}(1)(Dv(T))^2-X_{13}+Y_{13}.\nonumber\end{eqnarray}
with
\small{
\begin{eqnarray}&&X_{13}=\int_{Q'_0}{\big(s(sh)+T(sh)^2\theta\big)\mathcal{O_{\lambda, \mathfrak{K}}}(1)(Dv)^2}+\int_{Q_0}{s^{-1}\mathcal{O_{\lambda, \mathfrak{K}}}(sh)(\partial_t v)^2}.\nonumber\\
&&Y_{13}=\int_0^T{(r\bar{\tilde{\rho}})(a^-)\partial_t v(a)(c_d Dv)_{n+\frac{1}{2}}}-\int_0^T{(r\bar{\tilde{\rho}})(a^+)\partial_t v(a)(c_d Dv)_{n+\frac{3}{2}}}\nonumber
\end{eqnarray}
{

\subsection{Proof of Lemma~\ref{lem:estimate I21}}
\par We set $q=c^2r^2(\bar{D}D\rho)\overline{D\rho}$. Observing that $\overline{Dv}=\bar{D}\tilde{v}$ we get
\begin{eqnarray} I_{21}&=&2\int_{Q_{01}}{\underbrace{c^2r^2(\bar{D}D\rho)\overline{{D\rho}}}_{q}\bar{\tilde{v}}\ \overline{Dv}}+2\int_{Q_{02}}{\underbrace{c^2r^2(\bar{D}D\rho)\overline{{D\rho}}}_{q}\bar{\tilde{v}}\ \overline{Dv}}\nonumber\\
&=&\int_{Q_{01}}{q \bar{D}(\tilde{v})^2}+\int_{Q_{02}}{q \bar{D}(\tilde{v})^2}\nonumber\\
&=&-\int_{Q_{01}}{Dq(\tilde{v})^2}-\int_{Q_{02}}{Dq(\tilde{v})^2}\nonumber\\
&+&\int_0^T{q(a^-)(\tilde{v})^2_{n+\frac{1}{2}}}-\int_0^T{q(0)(\tilde{v})^2_{\frac{1}{2}}}\nonumber\\
&+& \int_0^T{q(1)(\tilde{v})^2_{n+m+\frac{3}{2}}}-\int_0^T{q(a^+)(\tilde{v})^2_{n+\frac{3}{2}}}\nonumber\\
&=&-\sum_{i=1}^2\int_{Q_{0i}}{Dq\wtilde{v^2}}+\sum_{i=1}^{2}\frac{h^{2}}{4}\int_{Q_{0i}}{(Dq)(Dv)^2}+Y_{21}^{(1)}\nonumber\\
&=&-\sum_{i=1}^2\int_{Q_{0i}}{\overline{Dq}(v)^2}+\sum_{i=1}^{2}\frac{h^{2}}{4}\int_{Q_{0i}}{(Dq)(Dv)^2}+Y_{21}^{(1)}+Y_{21}^{(2)},\nonumber\end{eqnarray}
by means of  Proposition~\ref{prop:IVP}, Lemma~\ref{lemma: discrete leibnitz}, Lemma~\ref{lem:average of two function} in each domain ${{\Omega}_{01}}$, ${{\Omega}_{02}}$ independently and where 

\begin{eqnarray}&&Y_{21}=Y_{21}^{(1)}+Y_{21}^{(2)}=Y_{21}^{(1,1)}+Y_{21}^{(1,2)}+Y_{21}^{(2)},\nonumber\\
&&Y_{21}^{(1,1)}=\int_0^T{q(1)(\tilde{v})^2_{n+m+\frac{3}{2}}}-\int_0^T{q(0)(\tilde{v})^2_{\frac{1}{2}}},\nonumber\\
&&Y_{21}^{(1,2)}=\int_0^T{q(a^-)(\tilde{v})^2_{n+\frac{1}{2}}}-\int_0^T{q(a^+)(\tilde{v})^2_{n+\frac{3}{2}}},\nonumber\\
&&Y_{21}^{(2)}=-\frac{h}{2}\int_0^T{v^2(a)(Dq)_{n+\frac{1}{2}}}-\frac{h}{2}\int_0^T{v^2(0)(Dq)_{\frac{1}{2}}}\nonumber\\
&&-\frac{h}{2}\int_0^T{v^2(1)(Dq)_{n+m+\frac{3}{2}}}-\frac{h}{2}\int_0^T{v^2(a)(Dq)_{n+\frac{3}{2}}}\nonumber\\
&&=-\frac{h}{2}\int_0^T{v^2(a)(Dq)_{n+\frac{1}{2}}}-\frac{h}{2}\int_0^T{v^2(a)(Dq)_{n+\frac{3}{2}}},\nonumber
\end{eqnarray}
as $v|_{\partial\Omega_0}=0.$

\par We note that $ \tilde{v}_{\frac{1}{2}}=\frac{h}{2}(Dv)_{\frac{1}{2}}, \ \  \tilde{v}_{n+m+\frac{3}{2}}=-\frac{h}{2}(Dv)_{n+m+\frac{3}{2}}$ . On the one hand, by Proposition~\ref{prop:property 1} we have $q=s^2\mathcal{O_{\lambda, \mathfrak{K}}}(1)r\overline{D\rho}$ in each domain ${{\Omega}_{01}}$, ${{\Omega}_{02}}$. It follows that

\begin{eqnarray}Y_{21}^{(1,1)}&=&\int_0^T{s^2\mathcal{O_{\lambda, \mathfrak{K}}}(1)(r\overline{D\rho})(1)(\tilde{v})^2_{n+m+\frac{3}{2}}}+\int_0^T{s^2\mathcal{O_{\lambda, \mathfrak{K}}}(1)(r\overline{D\rho})(0)(\tilde{v})^2_{\frac{1}{2}}}\nonumber\\
&=&\int_0^T{\mathcal{O_{\lambda, \mathfrak{K}}}(sh)^2(r\overline{D\rho})(1)(Dv)^2_{n+m+\frac{3}{2}}}+\int_0^T{\mathcal{O_{\lambda, \mathfrak{K}}}(sh)^2(r\overline{D\rho})(0)(Dv)^2_{\frac{1}{2}}}.\nonumber\end{eqnarray}
\par On the other hand,  by Proposition~\ref{prop:property 6}, Corollary~\ref{cor:multi derivative} we have $q=-c^2(s\phi\lambda)^3(\psi')^3+s^2\mathcal{O_{\lambda}}(1)+s^3\mathcal{O_{\lambda, \mathfrak{K}}}(sh)^2$ in each domain ${{\Omega}_{01}}$, ${{\Omega}_{02}}$. We thus obtain
\begin{eqnarray}Y_{21}^{(1,2)}&=&\int_0^T{\big(s\lambda\phi(a)\big)^3\Big(-(c^2\psi^{'3})(a^-)(\tilde{v})^2_{n+\frac{1}{2}}+(c^2\psi^{'3})(a^+)(\tilde{v})^2_{n+\frac{3}{2}}\Big)}\nonumber\\
&&+\int_0^T\Big(s^2\mathcal{O_{\lambda}}(1)+s^3\mathcal{O_{\lambda, \mathfrak{K}}}(sh)^2\Big)\Big((\tilde{v})^2_{n+\frac{1}{2}}-(\tilde{v})^2_{n+\frac{3}{2}}\Big)\nonumber\\
&=&Y_{21}^{(1,21)}+Y_{21}^{(1,22)},\nonumber\end{eqnarray}
where
\begin{equation}Y_{21}^{(1,21)}=\int_0^T{s^3\lambda^3\phi^3(a)[c^2(\psi')^3\star\tilde{v}^2]_a}.\nonumber
\end{equation}

\begin{Lemma}\label{lem:A3}(see Lemma B.8 in \cite{BHL10a}) Provided $sh\le\mathfrak{K}$ we have
\begin{eqnarray}&&Dq_{i}=s^3\mathcal{O_{\lambda, \mathfrak{K}}}(1),\nonumber\\
&&\overline{Dq_{i}}=-3s^3\lambda^4\phi_{i}^3 c^2(\psi'_{i})^4+(s\lambda \phi_{i})^3\mathcal{O}(1)+s^2\mathcal{O_{\lambda, \mathfrak{K}}}(1)+s^3\mathcal{O_{\lambda, \mathfrak{K}}}(sh)^2.\nonumber
\end{eqnarray}
\end{Lemma}
\par Note that the proof and the use of Lemma~\ref{lem:A3} are done in each domain ${{\Omega}_{01}}$, ${{\Omega}_{02}}$ separately.
\par We then obtain 
\begin{eqnarray}Y_{21}^{(2)}&=&-\frac{h}{2}\int_0^T{v^2(a)(Dq)_{n+\frac{1}{2}}}-\frac{h}{2}\int_0^T{v^2(a)(Dq)_{n+\frac{3}{2}}}\nonumber\\
&=&\int_0^T{s^2\mathcal{O_{\lambda, \mathfrak{K}}}(sh)v^2(a)}.\nonumber\end{eqnarray}

\par We thus write $I_{21}$ 
\begin{equation} I_{21} \ge 3\int_{Q'_0}{\lambda^4s^3\phi^3c^2(\psi')^4(v)^2}-\int_{Q'_0}{\mu_{21}(v)^2}-\int_{Q'_0}{\nu_{21}(Dv)^2}+Y_{21},\nonumber\end{equation}
\ where 
\begin{eqnarray}&&\mu_{21}=(s\lambda\phi)^3\mathcal{O}(1)+s^2\mathcal{O_{\lambda, \mathfrak{K}}}(1)+s^3\mathcal{O_{\lambda, \mathfrak{K}}}(sh)^2,\ \ \nu_{21}=s\mathcal{O_{\lambda, \mathfrak{K}}}(sh)^2,\nonumber\\
&&Y_{21}=Y_{21}^{(1,1)}+Y_{21}^{(1,21)}+Y_{21}^{(1,22)}+Y_{21}^{(2)}.\nonumber\end{eqnarray}

\subsection{ Proof of Lemma~\ref{lem:estimate I22}}
\par We set $q=c^2r(\bar{D}D\rho)\phi''$ and by Lemma~\ref{lem:average double} we have $\overline{\tilde{v}}=v+h^2\bar{D}Dv/4$ in each domain ${{\Omega}_{01}}$, ${{\Omega}_{02}}$. It follows that
\begin{eqnarray}I_{22}&=&-2\int_{Q_{01}}{sq\bar{\tilde{v}}v}-2\int_{Q_{02}}{sq\bar{\tilde{v}}v}\nonumber\\
&=&-2\int_{Q_{01}}{sqv^2}- \int_{Q_{01}}{\frac{sh^{2}}{2}q(\bar{D}Dv)v}\nonumber\\
&&-2\int_{Q_{02}}{sqv^2}-\int_{Q_{02}}{\frac{sh^{2}}{2}q(\bar{D}Dv)v}.\nonumber
\end{eqnarray}

\par Applying a discrete integration by parts (Proposition~\ref{prop:IVP}) and Lemma~\ref{lemma: discrete leibnitz} in each domain ${{\Omega}_{01}}$, ${{\Omega}_{02}}$ yield
\begin{eqnarray}I_{22}&=&-2\sum_{i=1}^2\int_{Q_{0i}}{sqv^2}+\sum_{i=1}^2 \int_{Q_{0i}}{\frac{sh^{2}}{2}D(qv)Dv}+Y_{22}^{(1)}\nonumber\\
&=&-2\sum_{i=1}^2\int_{Q_{0i}}{sqv^2}+\sum_{i=1}^2 \int_{Q_{0i}}{\frac{sh^{2}}{2}\tilde{q}(Dv)^2}+\sum_{i=1}^2 \int_{Q_{0i}}{\frac{sh^{2}}{2}D(q)\tilde{v}Dv}+Y_{22}^{(1)}\nonumber\\
&=&-2\sum_{i=1}^2\int_{Q_{0i}}{sqv^2}+\sum_{i=1}^2 \int_{Q_{0i}}{\frac{sh^{2}}{2}\tilde{q}(Dv)^2}+\sum_{i=1}^2\int_{Q_{0i}}{\frac{sh^{2}}{4}D(q)D(v^2)}+Y_{22}^{(1)}\nonumber\\
&=&-2\sum_{i=1}^2\int_{Q_{0i}}{sqv^2}+\sum_{i=1}^2 \int_{Q_{0i}}{\frac{sh^{2}}{2}\tilde{q}(Dv)^2}-\sum_{i=1}^2 \int_{Q_{0i}}{\frac{sh^{2}}{4}\bar{D}Dq v^2}+Y_{22}^{(1)}+Y_{22}^{(2)},\nonumber
\end{eqnarray}
where
\begin{eqnarray} &&Y_{22}^{(1)}=-\int_0^T{\frac{sh^{2}}{2}q(a^-)v(a)(Dv)_{n+\frac{1}{2}}}+\int_0^T{\frac{sh^{2}}{2}q(a^+)v(a)(Dv)_{n+\frac{3}{2}}},\nonumber\\
&&Y_{22}^{(2)}=\int_0^T{\frac{sh^{2}}{4}v^2(a)(Dq)_{n+\frac{1}{2}}}-\int_0^T{\frac{sh^{2}}{4}v^2(a)(Dq)_{n+\frac{3}{2}}},\nonumber
\end{eqnarray}
as $v|_{\partial\Omega_0}=0.$ 
\par In each domain ${{\Omega}_{01}}$, ${{\Omega}_{02}}$, we have $\phi^{''}=\mathcal{O_{\lambda}}(1)$ and from Proposition~\ref{prop:property 4} we have $q=s^2\mathcal{O_{\lambda, \mathfrak{K}}}(1)$ and $Dq=s^2\mathcal{O_{\lambda, \mathfrak{K}}}(1)$. We thus obtain 
\begin{eqnarray}&&Y_{22}^{(1)}=\int_0^T{s^3\mathcal{O_{\lambda, \mathfrak{K}}}(1)v(a)\frac{h^{2}}{2}(Dv)_{n+\frac{1}{2}}+s^3\mathcal{O_{\lambda, \mathfrak{K}}}(1)v(a)\frac{h^{2}}{2}(Dv)_{n+\frac{3}{2}}},\nonumber\\
&&Y_{22}^{(2)}=\int_0^T{s\mathcal{O_{\lambda, \mathfrak{K}}}(sh)^2v^2(a)}.\nonumber
\end{eqnarray}

\begin{Lemma}\label{lem:A4}(see Lemma B.9 and Lemma B.10 in \cite{BHL10a}) Provided $sh\le \mathfrak{K}$ we have
\begin{eqnarray}&&c^2r_{i}\bar{D}D\rho_{i}=c^2(r_{i}\partial^2\rho_{i}+s^2\mathcal{O_{\lambda, \mathfrak{K}}}(sh)^2)=c^2(s\lambda\phi_{i})^2(\psi'_{i})^2+s\mathcal{O_{\lambda}}(1)+s^2\mathcal{O_{\lambda, \mathfrak{K}}}(sh)^2,\nonumber\\
&&h^2\bar{D}Dq_{i}=s(sh)\mathcal{O_{\lambda, \mathfrak{K}}}(1).\nonumber\end{eqnarray}
\end{Lemma}
Note that the proof and use of above Lemma~\ref{lem:A4} are done in each domain ${{\Omega}_{01}}$, ${{\Omega}_{02}}$ separately.
\par Futhermore, we have $\phi''=\lambda^2(\psi')^2\phi+\lambda\phi\mathcal{O}(1)$ in each domain ${{\Omega}_{01}}$, ${{\Omega}_{02}}$. It follows that
\begin{eqnarray}sq_{i}&=&s\Big(c^2(s\lambda\phi_{i})^2(\psi'_{i})^2+s\mathcal{O_{\lambda}}(1)+s^2\mathcal{O_{\lambda, \mathfrak{K}}}(sh)^2\Big)\Big(\lambda^2(\psi'_{i})^2\phi_{i}+\lambda\phi_{i}\mathcal{O}(1)\Big)\nonumber\\
&=&c^2s^3\lambda^4(\psi'_{i})^4\phi_{i}^3+s^3\lambda^3\phi_{i}^3\mathcal{O}(1)+s^2\mathcal{O_{\lambda}}(1)+s^3\mathcal{O_{\lambda, \mathfrak{K}}}(sh)^2, \nonumber
\end{eqnarray} in each domain ${{\Omega}_{01}}$, ${{\Omega}_{02}}$.
\par We thus write $I_{22}$ as
\begin{equation}I_{22}=-2\int_{Q'_0}{c^2s^3\lambda^4\phi^3(\psi')^4 v^2}+\int_{Q'_0}{\mu_{22}v^2}+\int_{Q'_0}{\nu_{22}(Dv)^2}+Y_{22},\nonumber\end{equation}
\ where
\begin{eqnarray} &&\mu_{22}=(s\lambda\phi)^3\mathcal{O}(1)+s^2\mathcal{O_{\lambda, \mathfrak{K}}}(1)+s^3\mathcal{O_{\lambda, \mathfrak{K}}}(sh)^2,\ \  \nu_{22}=s\mathcal{O_{\lambda, \mathfrak{K}}}(sh)^2,\nonumber\\
&&Y_{22}=Y_{22}^{(1)}+Y_{22}^{(2)}.\nonumber\end{eqnarray}

\subsection{ Proof of Lemma~\ref{lem:estimate I23}}
\par By means of  a discrete integration by parts (Proposition~\ref{prop:IVP}) in each domain ${{\Omega}_{01}}$, ${{\Omega}_{02}}$, we obtain
\begin{eqnarray}I_{23}&=&\sum_{i=1}^2\int_{Q_{0i}}{cr(\bar{D}D\rho)\bar{\tilde{v}}\partial_t v }\nonumber\\
&=&\sum_{i=1}^2\int_{Q_{0i}}{\wtilde{cr(\bar{D}D\rho)\partial_t v}\tilde{v}}\nonumber\\
&-&\frac{h}{2}\int_0^T{(cr(\bar{D}D\rho))(0)\partial_t v(0)\tilde{v}_{\frac{1}{2}}}-\frac{h}{2}\int_0^T{(cr(\bar{D}D\rho))(a^-)\partial_t v(a)\tilde{v}_{n+\frac{1}{2}}}\nonumber\\
&-&\frac{h}{2}\int_0^T{(cr(\bar{D}D\rho))(a^+)\partial_t v(a)\tilde{v}_{n+\frac{3}{2}}}-\frac{h}{2}\int_0^T{(cr(\bar{D}D\rho))(1)\partial_t v(1)\tilde{v}_{n+m+\frac{3}{2}}}\nonumber\\
&=&\bar{Q}_1+\bar{Q}_2+Y_{23}^{(1)},\nonumber
\end{eqnarray}
by Lemma~\ref{lem:average of two function} and where
\begin{eqnarray}&&\bar{Q}_1=\sum_{i=1}^2\int_{Q_{0i}}{\wtilde{(cr(\bar{D}D\rho))}\partial_t \tilde{v} \tilde{v}},\nonumber\\
&&\bar{Q}_2=\sum_{i=1}^2{\frac{h^{2}}{4}}{\int_{Q_{0i}}{D(cr\bar{D}D\rho)(D\partial_t v)\tilde{v}}},\nonumber\\
&&Y_{23}^{(1)}=-\frac{h}{2}\int_0^T{(cr(\bar{D}D\rho))(a^-)\partial_t v(a)\tilde{v}_{n+\frac{1}{2}}}-\frac{h}{2}\int_0^T{(cr(\bar{D}D\rho))(a^+)\partial_t v(a)\tilde{v}_{n+\frac{3}{2}}}\nonumber
\end{eqnarray}
as $\partial_t v|_{\partial\Omega_0}=0$.
\par With an integrations by parts w.r.t $t$ we have
\begin{equation} \bar{Q}_1=\frac{-1}{2}\sum_{i=1}^2\int_{Q_{0i}}{\partial_t{\wtilde{(cr\bar{D}D\rho)}}({\tilde{v}})^2}+\frac{1}{2}\sum_{i=1}^2\int_{\Omega_{0i}}{\wtilde{(cr(\bar{D}D\rho))}({\tilde{v}})^2|_{t=0}^{t=T}}.\nonumber\end{equation}
\par By means of Proposition~\ref{prop:property 4} and Lemma~\ref{lem:derivative wrt x} in each domain ${{\Omega}_{01}}$, ${{\Omega}_{02}}$ we get
\begin{eqnarray}&&\wtilde{cr_{i}(\bar{D}D\rho_{i})}=s^2\mathcal{O_{\lambda, \mathfrak{K}}}(1),\nonumber\\
&&r_{i}\bar{D}D\rho_{i}=s^2\mathcal{O_{\lambda, \mathfrak{K}}}(1),\nonumber\end{eqnarray}
and we further have

\begin{Lemma}\label{lem:A5}(see Lemma A.1 in \cite{BL12})
\begin{equation}\partial_t(\wtilde{cr_{i}\bar{D}D\rho_{i}})=Ts^2\theta\mathcal{O_{\lambda, \mathfrak{K}}}(1).\nonumber
\end{equation}
\end{Lemma}
\par Note that the proof and use of Lemma~\ref{lem:A5} are done in each domain ${{\Omega}_{01}}$, ${{\Omega}_{02}}$ separately.
\par It follows that
\begin{equation}\bar{Q}_1=\sum_{i=1}^2\int_{Q_{0i}}{Ts^2\theta\mathcal{O_{\lambda, \mathfrak{K}}}(1) \wtilde{{v}^2}}+\sum_{i=1}^2\int_{\Omega_{0i}}{s^2\big(\mathcal{O_{\lambda, \mathfrak{K}}}(1)\wtilde{{v}^2}_{|_{t=0}}+\mathcal{O_{\lambda, \mathfrak{K}}}(1)\wtilde{{v}^2}_{|_{t=T}}\big)}\nonumber
\end{equation}
as $\left|{\tilde{v}}\right|^2\le \wtilde{\left|{v}\right|^2}$ in each domain ${{\Omega}_{01}}$, ${{\Omega}_{02}}$.
\par Moreover, we observe that $ \sum_{i=1}^2\int_{\Omega_{0i}}{\mathcal{O_{\lambda, \mathfrak{K}}}(1)\wtilde{{v}^2}}=\int_{\Omega_0}{\mathcal{O_{\lambda, \mathfrak{K}}}(1){v}^2}$. Then,
\begin{equation}
  \label{eq: computation Qtilde1}
  \bar{Q}_1= \int_{Q_0}{Ts^2\theta\mathcal{O_{\lambda, \mathfrak{K}}}(1)v^2}+\int_{\Omega_{0}}{s^2\big(\mathcal{O_{\lambda, \mathfrak{K}}}(1)\wtilde{{v}^2}_{|_{t=0}}+\mathcal{O_{\lambda, \mathfrak{K}}}(1)\wtilde{{v}^2}_{|_{t=T}}\big)}.
\end{equation}
\par We have

\begin{equation}Y_{23}^{(1)}=\int_0^T{s^2\mathcal{O_{\lambda, \mathfrak{K}}}(1)\partial_t v(a)\frac{ h}{2}(\tilde{v}_{n+\frac{1}{2}})+s^2\mathcal{O_{\lambda, \mathfrak{K}}}(1)\partial_t v(a)\frac{h}{2}(\tilde{v}_{n+\frac{3}{2}})}.\nonumber
\end{equation} 
\par By an integration by parts w.r.t $t$ and Lemma~\ref{lemma: discrete leibnitz} in each domain ${{\Omega}_{01}}$, ${{\Omega}_{02}}$ we find
\begin{equation}\bar{Q}_2=\underbrace{-\sum_{i=1}^2 \frac{h^{2}}{4}\int_{Q_{0i}}{\partial_t(D(cr\bar{D}D\rho)\tilde{v})Dv}}_{\bar{Q}_2^1}+\underbrace{\sum_{i=1}^2\frac{h^{2}}{8}\int_{\Omega_{0i}}{D(cr\bar{D}D\rho)D(v)^2}|_{t=0}^{t=T}}_{\bar{Q}_2^2}.\nonumber\end{equation}

\par By means of Lemma~\ref{lemma: discrete leibnitz} and a discrete intergration by parts in space (Proposition~\ref{prop:IVP}) in each domain ${{\Omega}_{01}}$, ${{\Omega}_{02}}$ we see that
\begin{eqnarray}\bar{Q}_2^1&=&\sum_{i=1}^2 \frac{h^{2}}{8}\int_{Q_{0i}}{\partial_t(\bar{D}D(cr\bar{D}D\rho))v^2}-\sum_{i=1}^2 \frac{h^{2}}{4}\int_{Q_{0i}}{D(cr\bar{D}D\rho)(\partial_t{\tilde{v}})Dv}\nonumber\\
&&-\frac{h^{2}}{8}\int_0^T{v^2(a)\partial_t(D(cr\bar{D}D\rho))_{n+\frac{1}{2}}}+\frac{h^{2}}{8}\int_0^T{v^2(a)\partial_t(D(cr\bar{D}D\rho))_{n+\frac{3}{2}}}\nonumber\\
&=&\sum_{i=1}^2 \frac{h^{2}}{8}\int_{Q_{0i}}{\partial_t(\bar{D}D(cr\bar{D}D\rho))v^2}-\sum_{i=1}^2 \frac{h^{2}}{4}\int_{Q_{0i}}{D(cr\bar{D}D\rho)(\partial_t{\tilde{v}})Dv}+Y_{23}^{(2)}\nonumber\
\end{eqnarray}
as $v|_{\partial\Omega_0}=0$.
\begin{Lemma}\label{lem:A6}(Lemma A.2 in \cite{BL12}) Provided $sh\le \mathfrak{K}$ we have
\begin{eqnarray}&&h^2\bar{D}D(c_{i}r_{i}(\bar{D}D\rho_{i}))=s(sh)\mathcal{O_{\lambda, \mathfrak{K}}}(1),\nonumber\\
&&h^2\partial_t (\bar{D}D(c_ir_{i}\bar{D}D\rho_{i}))=Ts^2\theta\mathcal{O_{\lambda, \mathfrak{K}}}(1),\nonumber\\
&&h\partial_t (D(c_ir_{i}\bar{D}D\rho_{i}))=Ts^2\theta\mathcal{O_{\lambda, \mathfrak{K}}}(1),\nonumber\\
&&D(c_ir_{i}\bar{D}D\rho_{i})=s^2\mathcal{O_{\lambda, \mathfrak{K}}}(1).\nonumber\
\end{eqnarray}
\end{Lemma}
Note that all above terms are done in each domain ${{\Omega}_{01}}$, ${{\Omega}_{02}}$ separately.
\par We thus obtain
\begin{equation}Y_{23}^{(2)}=\int_0^T{sT\theta\mathcal{O_{\lambda, \mathfrak{K}}}(sh)v^2(a)}.\nonumber\
\end{equation}

\par Applying the Young's inequality and using that $\left|{\partial_t\tilde{v}}\right|^2\le \wtilde{\left|{\partial_t v}\right|^2} $ in each domain ${{\Omega}_{01}}$, ${{\Omega}_{02}}$, we have 
\small{
\begin{eqnarray}\label{eq:A2}\bar{Q}_2^1&\ge&\int_{Q'_0}{Ts^2\theta\mathcal{O_{\lambda, \mathfrak{K}}}(1)v^2}+\int_{Q'_0}{s^{-1}\mathcal{O_{\lambda, \mathfrak{K}}}(sh)^2\wtilde{\left|{\partial_t v}\right|^2}}+\int_{Q'_0}{s\mathcal{O_{\lambda, \mathfrak{K}}}(sh)^2(Dv)^2}+Y_{23}^{(2)}\nonumber\\
&\ge&\int_{Q'_0}{Ts^2\theta\mathcal{O_{\lambda, \mathfrak{K}}}(1)v^2}+\int_{Q_0}{s^{-1}\mathcal{O_{\lambda, \mathfrak{K}}}(sh)^2{\left|{\partial_t v}\right|^2}}+\int_{Q'_0}{s\mathcal{O_{\lambda, \mathfrak{K}}}(sh)^2(Dv)^2}+Y_{23}^{(2)}\nonumber\\
\end{eqnarray}
}
as $\sum_{i=1}^2\int_{\Omega_{0i}}{\mathcal{O_{\lambda, \mathfrak{K}}}(1)\wtilde{\left|{\partial_t v}\right|^2}}=\int_{\Omega}{\mathcal{O_{\lambda, \mathfrak{K}}}(1)\left|{\partial_t v}\right|^2}$.

\par By using Proposition~\ref{prop:IVP}, Lemma~\ref{lem:A6} in each domain ${{\Omega}_{01}}$, ${{\Omega}_{02}}$ separately yield
 \begin{eqnarray}\label{eq:A3}\bar{Q}_2^2&=&-\sum_{i=1}^2\frac{h^{2}}{8}\int_{\Omega_{0i}}{\bar{D}D(cr\bar{D}D\rho)(v)^2|_{t=0}^{t=T}}\nonumber\\
&&+\frac{h^{2}}{8}v^2(a)\big(D(cr\bar{D}D\rho)\big)_{n+\frac{1}{2}}|_{t=0}^{t=T}-\frac{h^{2}}{8}v^2(a)\big(D(cr\bar{D}D\rho)\big)_{n+\frac{3}{2}}|_{t=0}^{t=T}\nonumber\\
&=&\int_{\Omega'_0}{s\mathcal{O_{\lambda, \mathfrak{K}}}(sh)(v)^2|_{t=T}}+\int_{\Omega'_0}{s\mathcal{O_{\lambda, \mathfrak{K}}}(sh)(v)^2|_{t=0}}+\mathcal{O_{\lambda, \mathfrak{K}}}(sh)^2v^2(a)|_{t=0}^{t=T}\nonumber\\
&=&\int_{\Omega'_0}{s\mathcal{O_{\lambda, \mathfrak{K}}}(sh)(v)^2|_{t=T}}+\int_{\Omega'_0}{s\mathcal{O_{\lambda, \mathfrak{K}}}(sh)(v)^2|_{t=0}}+Y_{23}^{(3)}, 
\end{eqnarray}
as $v|_{\partial\Omega_0}=0$ where 
\begin{equation}Y_{23}^{(3)}=\mathcal{O_{\lambda, \mathfrak{K}}}(sh)^2v^2(a)|_{t=0}^{t=T}.\nonumber\
\end{equation}

\par Collecting~\eqref{eq: computation Qtilde1}, ~\eqref{eq:A2} and ~\eqref{eq:A3} we obtain
\begin{equation}I_{23}\ge \int_{\Omega_0}{s^2\big(\mathcal{O_{\lambda, \mathfrak{K}}}(1)v^2_{|_{t=0}}+\mathcal{O_{\lambda, \mathfrak{K}}}(1)v^2_{|_{t=T}}\big)}-X_{23}+Y_{23},\nonumber\
\end{equation}
where $X_{23}$ and $Y_{23}$ are as given in the statement of Lemma~\ref{lem:estimate I23}.

\subsection{Proof of Lemma~\ref{lem:estimate I31}}

\par By means of a discrete integration by parts (Proposition~\ref{prop:IVP}) in each domain ${{\Omega}_{01}}$, ${{\Omega}_{02}}$ separately, we get
 \begin{eqnarray}I_{31}&=&-2\tau\int_{Q_{01}}{(\partial_t \theta)\varphi cr\overline{D\rho}v\overline{Dv}}-2\tau\int_{Q_{02}}{(\partial_t \theta)\varphi cr\overline{D\rho}v\overline{Dv}}\nonumber\\
&=&-2\tau\int_{Q_{01}}{(\partial_t \theta)\wtilde{\varphi cr\overline{D\rho} v }{Dv}}-2\tau\int_{Q_{02}}{(\partial_t \theta)\wtilde{\varphi cr\overline{D\rho} v }{Dv}}+Y_{31}^{(1)},\nonumber
\end{eqnarray}
with
\begin{equation}Y_{31}^{(1)}=\tau\frac{h}{2}\int_0^T{(\partial_t \theta)(cr\overline{D\rho}\varphi v)(a^-)(Dv)_{n+\frac{1}{2}}}+\tau\frac{h}{2}\int_0^T{(\partial_t \theta)(cr\overline{D\rho}\varphi v)(a^+)(Dv)_{n+\frac{3}{2}}}\nonumber
\end{equation}
as $v|_{\partial\Omega_0}=0$.
\par We have $\wtilde{\varphi cr\overline{D\rho} v }=\wtilde{\varphi cr\overline{D\rho} }\tilde{v}+\frac{h^2}{4}D(\varphi cr\overline{D\rho})Dv$ in each domain ${{\Omega}_{01}}$, ${{\Omega}_{02}}$. It follows that

\begin{eqnarray}I_{31}&=&-\tau\sum_{i=1}^2\int_{Q_{0i}}{(\partial_t\theta)\wtilde{(cr\overline{D\rho}\varphi) }D(v)^2}-\sum_{i=1}^2 \tau\frac{h^{2}}{2}\int_{Q_{0i}}{(\partial_t \theta)D(cr\overline{D\rho}\varphi)(Dv)^2}+Y_{31}^{(1)}\nonumber\\
&=&\tau\sum_{i=1}^2\int_{Q_{0i}}{(\partial_t\theta)\overline{(D(cr\overline{D\rho}\varphi))}v^2}-\sum_{i=1}^2 \tau\frac{h^{2}}{2}\int_{Q_{0i}}{(\partial_t \theta)D(cr\overline{D\rho}\varphi)(Dv)^2}+Y_{31}^{(1)}\nonumber\\
&-&\tau \int_0^T{(\partial_t \theta)v^2(a)\wtilde{(cr\overline{D\rho}\varphi)}_{n+\frac{1}{2}}}+\tau \int_0^T{(\partial_t \theta)v^2(a)\wtilde{(cr\overline{D\rho}\varphi)}_{n+\frac{3}{2}}}\nonumber\\
&=&\tau\sum_{i=1}^2\int_{Q_{0i}}{(\partial_t\theta)\overline{(D(cr\overline{D\rho}\varphi))}v^2}-\sum_{i=1}^2 \tau\frac{h^{2}}{2}\int_{Q_{0i}}{(\partial_t \theta)D(cr\overline{D\rho}\varphi)(Dv)^2}+Y_{31}^{(1)}+Y_{31}^{(2)}\nonumber\end{eqnarray}
by using a discrete integration by parts in each domain ${{\Omega}_{01}}$, ${{\Omega}_{02}}$ separately and
\begin{equation}Y_{31}^{(2)}=-\tau \int_0^T{(\partial_t \theta)v^2(a)\wtilde{(cr\overline{D\rho}\varphi)}_{n+\frac{1}{2}}}+\tau \int_0^T{(\partial_t \theta)v^2(a)\wtilde{(cr\overline{D\rho}\varphi)}_{n+\frac{3}{2}}}\nonumber
\end{equation}
as $v|_{\partial \Omega}=0$.
\par By using the Lipschitz continuity and Proposition~\ref{prop:property 4} we get
\small{
\begin{eqnarray}&&D(cr_{i}\overline{D\rho_{i}}\varphi_{i})=s\mathcal{O_{\lambda, \mathfrak{K}}}(1),\nonumber\\
&&\overline{D(cr_{i}\overline{D\rho_{i}}\varphi_{i})}=s\mathcal{O_{\lambda, \mathfrak{K}}}(1),\nonumber\\
&&\wtilde{cr_{i}\overline{D\rho_{i}}\varphi_{i}}=s\mathcal{O_{\lambda, \mathfrak{K}}}(1),\nonumber\\
&&cr_{i}\overline{D\rho_{i}}=c(r_{i}\partial \rho_{i}+s^2\mathcal{O_{\lambda, \mathfrak{K}}}(sh)^2)=c(-s\lambda\phi_{i}\psi'_{i}+s\mathcal{O_{\lambda, \mathfrak{K}}}(sh)^2)=s\mathcal{O_{\lambda, \mathfrak{K}}}(1).\nonumber\
\end{eqnarray}
}

The proof is done in each domain ${{\Omega}_{01}}$, ${{\Omega}_{02}}$ separately. Note that $\underset{ t}{\max}{\partial_t\theta}=T\theta^2.$
\par It thus follows that
\begin{equation}I_{31}=\int_{Q'_0}{T\theta s^2\mathcal{O_{\lambda, \mathfrak{K}}}(1)v^2}+\int_{Q'_0}{T\theta\mathcal{O_{\lambda, \mathfrak{K}}}(sh)^2 (Dv)^2}+Y_{31},\nonumber\
\end{equation}
where 
\begin{eqnarray}&&Y_{31}=Y_{31}^{(1)}+Y_{31}^{(2)},\nonumber\\
&&Y_{31}^{(1)}=\int_0^T{T\theta s^2\mathcal{O_{\lambda, \mathfrak{K}}}(1)v(a)\frac{h}{2}(Dv)_{n+\frac{1}{2}}}+\int_0^T{T\theta s^2\mathcal{O_{\lambda, \mathfrak{K}}}(1)v(a)\frac{h}{2}(Dv)_{n+\frac{3}{2}}},\nonumber\\
&&Y_{31}^{(2)}=\int_0^T{T\theta s^2\mathcal{O_{\lambda, \mathfrak{K}}}(1)v^2(a)}.\nonumber
\end{eqnarray}

\subsection{Proof of Lemma~\ref{lem:estimate Y11 plus Y21}}
\par We see that
\begin{eqnarray}&&Y_{11}^{(1)}+Y_{21}^{(1,1)}\nonumber\\
&=&\int_0^T{\big(1+\mathcal{O_{\lambda, \mathfrak{K}}}(sh)\big)(c\bar{c}_d)(1)(r\overline{D\rho})_1(Dv)^2_{n+m+\frac{3}{2}}}\nonumber\\
&-&\int_0^T{\big(1+\mathcal{O_{\lambda, \mathfrak{K}}}(sh)\big)(c\bar{c}_d)(0)(r\overline{D\rho})_0(Dv)^2_{\frac{1}{2}}}\nonumber\\
&+&\int_0^T{\mathcal{O_{\lambda, \mathfrak{K}}}(sh)^2(r\overline{D\rho})_0(Dv)^2_{\frac{1}{2}}}+\int_0^T{\mathcal{O_{\lambda, \mathfrak{K}}}(sh)^2(r\overline{D\rho})_1(Dv)^2_{n+m+\frac{3}{2}}}.\nonumber
\end{eqnarray}
\par Moreover, by ~\eqref{constrainted boundary} we have $Y_{11}^{(1)}+Y_{21}^{(1,1)}\ge 0$ for sh sufficiently small.\\

\par We next focus our attention on the trace term at $'a'$ on $Y_{11}^{(2,1)}+Y_{21}^{(1,21)}$ as follows

\subsection{Proof of Lemma~\ref{lem:estimate Y11 plus Y21 again}}
\begin{eqnarray}\label{eq:A4}(\tilde{v})^2_{n+\frac{3}{2}}&=&\big({\frac{v_{n+1}+v_{n+2}}{2}}\big)^2=\big({v_{n+1}+\frac{h}{2}(Dv)_{n+\frac{3}{2}}}\big)^2\nonumber\\
&=&v^2_{n+1}+\frac{h^{2}}{4}(Dv)^2_{n+\frac{3}{2}}+ hv_{n+1}(Dv)_{n+\frac{3}{2}}\nonumber\\
&=&v^2_{n+1}+\frac{h^{2}}{4(c_d)^2_{n+\frac{3}{2}}}(c_dDv)^2_{n+\frac{3}{2}}+ v_{n+1}\frac{h}{(c_d)_{n+\frac{3}{2}}}(c_dDv)_{n+\frac{3}{2}}.\nonumber\\
\end{eqnarray}
\par Similarly, we have

\begin{eqnarray}\label{eq:A5}&&(\tilde{v})^2_{n+\frac{1}{2}}=v^2_{n+1}+\frac{h^{2}}{4(c_d)^2_{n+\frac{1}{2}}}(c_dDv)^2_{n+\frac{1}{2}}- v_{n+1}\frac{h}{(c_d)_{n+\frac{1}{2}}}(c_dDv)_{n+\frac{1}{2}}.
\end{eqnarray}

\par We thus write $Y_{21}^{(1,21)}$ as follows:
\begin{eqnarray}\label{eq:A6}&&Y_{21}^{(1,21)}\nonumber\\
&=&\int_0^T{(s\lambda\phi(a))^3[c^2(\psi')^3\star \left|{\tilde{v}}\right|^2]_a}\nonumber\\
&=&\int_0^T{(s\lambda\phi(a))^3(c^2\psi^{'3})(a^+)\Big(v^2_{n+1}+\frac{h^{2}}{4(c_d)^2_{n+\frac{3}{2}}}(c_dDv)^2_{n+\frac{3}{2}}+v_{n+1}\frac{h}{2(c_d)_{n+\frac{3}{2}}}(c_dDv)_{n+\frac{3}{2}}\Big)}\nonumber\\
&-&\int_0^T{(s\lambda\phi(a))^3(c^2\psi^{'3})(a^-)\Big(v^2_{n+1}+\frac{h^{2}}{4(c_d)^2_{n+\frac{1}{2}}}(cDv)^2_{N+\frac{1}{2}}-v_{n+1}\frac{h}{2(c_d)_{n+\frac{1}{2}}}(c_dDv)_{n+\frac{1}{2}}\Big)}\nonumber\\
&=&\int_0^T{(s\lambda\phi(a))^3[c^2\psi^{'3}\star]_a v^2_{n+1}}\nonumber\\
&+&\int_0^T{(s\lambda\phi(a))^3\Big((c^2\psi'^3)(a^+)\frac{h^{2}}{4(c_d)^2_{n+\frac{3}{2}}}(c_dDv)^2_{n+\frac{3}{2}}-(c^2\psi'^3)(a^-)\frac{h^{2}}{4(c_d)^2_{n+\frac{1}{2}}}(c_dDv)^2_{n+\frac{1}{2}}\Big)}\nonumber\\
&+&\int_0^T{(s\lambda\phi(a))^3\Big((c^2\psi'^3)(a^+)\frac{h}{2(c_d)_{n+\frac{3}{2}}}(c_dDv)_{n+\frac{3}{2}}-(c^2\psi'^3)(a^-)\frac{h}{2(c_d)_{n+\frac{1}{2}}}(c_dDv)_{n+\frac{1}{2}}\Big)v(a)}.\nonumber\\
\end{eqnarray}

\par Moreover, the term $Y_{11}^{(2,1)}$ is given by

\begin{eqnarray}Y_{11}^{(2,1)}&=&\int_0^T{s\lambda\phi(a)\Big(-\psi'(a^-)c(a^-)\bar{c}_d(a)(Dv)^2_{n+\frac{1}{2}}+\psi'(a^+)c(a^+)\bar{c}_d(a)(Dv)^2_{n+\frac{3}{2}}\Big)}\nonumber\\
&=&\int_0^T{s\lambda\phi(a)\Big(-\psi'(a^-)\frac{c(a^-)\bar{c}_d(a)}{(c_d)^2_{n+\frac{1}{2}}}(c_dDv)^2_{n+\frac{1}{2}}+\psi'(a^+)\frac{c(a^+)\bar{c}_d(a)}{(c_d)^2_{n+\frac{3}{2}}}(c_dDv)^2_{n+\frac{3}{2}}\Big)}.\nonumber
\end{eqnarray}

\par We estimate as
\begin{eqnarray}\frac{c(a^-)\bar{c}_d(a)}{(c_d)^2_{n+\frac{1}{2}}}&=&\frac{\big( (c_d)_{n+\frac{1}{2}}+\mathcal{O}(h)\big)\big( (c_d)_{n+\frac{1}{2}}+(c_d)_{n+\frac{3}{2}}\big)}{2(c_d)^2_{n+\frac{1}{2}}}\nonumber\\
&=&\frac{\big( (c_d)_{n+\frac{1}{2}}+\mathcal{O}(h)\big)\big(2 (c_d)_{n+\frac{1}{2}}+\mathcal{O}(h)\big)}{2(c_d)^2_{n+\frac{1}{2}}}\nonumber\\
&=&1+h\mathcal{O}(1).\nonumber
\end{eqnarray}
\par Similarly,
\begin{equation}\frac{c(a^+)\bar{c}_d(a)}{(c_d)^2_{n+\frac{3}{2}}}=1+h\mathcal{O}(1).\nonumber\end{equation}

\par We thus obtain $Y_{11}^{(2,1)}$ 
\begin{eqnarray}\label{eq:A.7}Y_{11}^{(2,1)}&=&\int_0^T{s\lambda\phi(a)\Big(-\psi'(a^-)\big(1+h\mathcal{O}(1)\big)(c_dDv)^2_{n+\frac{1}{2}}+\psi'(a^+)\big(1+h\mathcal{O}(1)\big)(c_dDv)^2_{n+\frac{3}{2}}\Big)}\nonumber\\
&=&\int_0^T{s\lambda\phi(a)[\psi'\star (c_dDv)^2]_a}\nonumber\\
&+&\int_0^T{s\lambda\phi(a)\psi'(a^+)\mathcal{O}(h)(c_dDv)^2_{n+\frac{3}{2}}}+\int_0^T{s\lambda\phi(a)\psi'(a^-)\mathcal{O}(h)(c_dDv)^2_{n+\frac{1}{2}}}.
\end{eqnarray}

\par Combining ~\eqref{eq:A6} with ~\eqref{eq:A.7} we obtain
\begin{eqnarray}&&Y_{11}^{(2,1)}+Y_{21}^{(1,21)}\nonumber\\
&=&\int_0^T{s\lambda\phi(a)[\psi'\star (c_dDv)^2]_a}+\int_0^T{s^3\lambda^3\phi^3(a)[(\psi')^3c^2\star]_a v^2_{n+1}}\nonumber\\
&+&\int_0^T{s\lambda\phi(a)\psi'(a^+)\mathcal{O}(h)(c_dDv)^2_{n+\frac{3}{2}}}+\int_0^T{s\lambda\phi(a)\psi'(a^-)\mathcal{O}(h)(c_dDv)^2_{n+\frac{1}{2}}}\nonumber\\
&+&\int_0^T{(s\lambda\phi(a))^3\Big((c^2\psi'^3)(a^+)\frac{h^2}{4(c_d)^2_{n+\frac{3}{2}}}(c_dDv)^2_{n+\frac{3}{2}}-(c^2\psi'^3)(a^-)\frac{h^2}{4(c_d)^2_{n+\frac{1}{2}}}(c_dDv)^2_{n+\frac{1}{2}}\Big)}\nonumber\\
&+&\int_0^T{(s\lambda\phi(a))^3\Big((c^2\psi'^3)(a^+)\frac{h}{2(c_d)_{n+\frac{3}{2}}}(c_dDv)_{n+\frac{3}{2}}-(c^2\psi'^3)(a^-)\frac{h}{2(c_d)_{n+\frac{1}{2}}}(c_dDv)_{n+\frac{1}{2}}\Big)v(a)}\nonumber\\
&=&\mu+\mu_1,\nonumber
\end{eqnarray}
where 
\begin{equation}\mu=\int_0^T{s\lambda\phi(a)[\psi'\star (c_dDv)^2]_a}+\int_0^T{s^3\lambda^3\phi^3(a)[c^2(\psi')^3\star]_a v^2_{n+1}}.\nonumber
\end{equation}
and $\mu_1$ can be written as
\small{
\begin{eqnarray}
\mu_1&=&\int_0^T{s\mathcal{O_{\lambda, \mathfrak{K}}}(sh)(c_dDv)^2_{n+\frac{3}{2}}}+\int_0^T{s\mathcal{O_{\lambda, \mathfrak{K}}}(sh)(c_dDv)^2_{n+\frac{1}{2}}}\nonumber\\
&+&\int_0^T{s^2\mathcal{O_{\lambda, \mathfrak{K}}}(sh)(c_dDv)_{n+\frac{3}{2}}v_{n+1}}+\int_0^T{s^2\mathcal{O_{\lambda, \mathfrak{K}}}(sh)(c_dDv)_{n+\frac{1}{2}}v_{n+1}}.\nonumber
\end{eqnarray}
}
\par We can write
\begin{eqnarray}&&[(\psi')\star(c_dDv)^2]_a\nonumber\\
&=&[(\psi')\star]_a(c_dDv)^2_{n+\frac{1}{2}}+[\star (c_dDv)]^2_a\psi'(a^+)+2[\star(c_dDv)]_a\psi'(a^+)(c_dDv)_{n+\frac{1}{2}}.\nonumber
\end{eqnarray}
\par Indeed, we have
\begin{equation}
 [(\psi')\star(c_dDv)^2]_a=(c_dDv)^2_{n+\frac{3}{2}}\psi'(a^+)-(c_dDv)^2_{n+\frac{1}{2}}\psi'(a^-),\nonumber
\end{equation}
and

\begin{eqnarray}
 &&[(\psi')\star]_a(c_dDv)^2_{n+\frac{1}{2}}+[\star (c_dDv)]^2_a\psi'(a^+)+2[\star(c_dDv)]_a\psi'(a^+)(c_dDv)_{n+\frac{1}{2}}\nonumber\\
&=&(c_dDv)^2_{n+\frac{1}{2}}\psi'(a^+)-(c_dDv)^2_{n+\frac{1}{2}}\psi'(a^-)\nonumber\\
&+&(c_dDv)^2_{n+\frac{3}{2}}\psi'(a^+)+(c_dDv)^2_{n+\frac{1}{2}}\psi'(a^+)-2(c_dDv)_{n+\frac{3}{2}}(c_dDv)_{n+\frac{1}{2}}\psi'(a^+)\nonumber\\
&+&2(c_dDv)_{n+\frac{3}{2}}(c_dDv)_{n+\frac{1}{2}}\psi'(a^+)-2(c_dDv)^2_{n+\frac{1}{2}}\psi'(a^+)\nonumber\\
&=&(c_dDv)^2_{n+\frac{3}{2}}\psi'(a^+)-(c_dDv)^2_{n+\frac{1}{2}}\psi'(a^-).\nonumber
\end{eqnarray}

\par Moreover, by using Lemma~\ref{lem:transmission condition}, we obtain
\begin{eqnarray}&&[(\psi')\star(c_dDv)^2]_a\nonumber\\
&=&[(\psi')\star]_a(c_dDv)^2_{n+\frac{1}{2}}+[\star (c_dDv)]^2_a\psi'(a^+)+2[\star(c_dDv)]_a\psi'(a^+)(c_dDv)_{n+\frac{1}{2}}\nonumber\\
&=&[(\psi')\star]_a(c_dDv)^2_{n+\frac{1}{2}}+\Big(\lambda^2 s^2[\star(c\phi\psi')]^2_a v^2_{n+1}+r_0^2+2\lambda s r_0[\star c\phi \psi']_a v_{n+1}\Big)\psi'(a^+)\nonumber\\
&& \ \ \ \ \ \ \ \ \ \ \ \ \ \ \ +2\Big(s\lambda[\star c\phi\psi']_a v_{n+1}+r_0\Big)\psi'(a^+)(c_dDv)_{n+\frac{1}{2}},\nonumber
\end{eqnarray}
which gives
\begin{eqnarray}\mu&=&\int_0^T{s\lambda\phi(a)[\psi'\star]_a(c_dDv)^2_{n+\frac{1}{2}}}\nonumber\\
&+&\int_0^T{2s^2\lambda^2\phi(a)[\star c \phi\psi']_a\psi'(a^+)v_{N+1}(c_dDv)_{n+\frac{1}{2}}}\nonumber\\
&+&\int_0^T{s^3\lambda^3\phi(a)\Big( [\star c\phi\psi']^2_a\psi'(a^+)+[c^2(\psi')^3 \star]_a\phi^2(a)\Big)v^2_{n+1} }\nonumber\\
&+&\int_0^T{s\lambda\phi(a)\psi'(a^+)r_0^2}+2\int_0^T{s^2\lambda^2\phi(a)[\star c \phi\psi']_a\psi'(a^+)r_0v_{n+1}}\nonumber\\
&+&2\int_0^T{s\lambda\phi(a)\psi'(a^+)r_0(c_dDv)_{n+\frac{1}{2}}}\quad.\nonumber
\end{eqnarray}

\par Moreover, we have:
\begin{eqnarray}&&[\star c \phi\psi']_a=c\phi\psi' |_{n+\frac{3}{2}}-c\phi\psi'|_{n+\frac{1}{2}}=c\phi\psi'|_{a^+}-c\phi\psi'|_{a^-}+h\mathcal{O_{\lambda}}(1)=\phi(a)[c\psi'\star]_a+h\mathcal{O_{\lambda}}(1),\nonumber\\
&&[\star c\phi\psi']^2_a=[c\phi\psi' \star]^2_a+2[c\phi\psi' \star]_a h\mathcal{O_{\lambda}}(1)+h^2\mathcal{O_{\lambda}}(1)=\phi^2(a)[c\psi' \star]^2_a+h\mathcal{O_{\lambda}}(1).\nonumber\
\end{eqnarray}

\par We thus write $\mu $ as
\begin{eqnarray}\mu&=&\int_0^T{s\lambda\phi(a)[\psi'\star]_a(c_dDv)^2_{n+\frac{1}{2}}}\nonumber\\
&+&\int_0^T{2s^2\lambda^2\phi^2(a)[ c\psi'\star]_a\psi'(a^+)v_{n+1}(c_dDv)_{n+\frac{1}{2}}}\nonumber\\
&+&\int_0^T{s^3\lambda^3\phi^3(a)\Big( [ c\psi'\star]^2_a\psi'(a^+)+[c^2(\psi')^3 \star]_a\Big)v^2_{n+1} } \nonumber\\
&+&\int_0^T{s\lambda\phi(a)\psi'(a^+)r_0^2}+2\int_0^T{s^2\lambda^2\phi^2(a)[ c\psi'\star]_a\psi'(a^+)r_0v_{n+1}}\nonumber\\
&+&2\int_0^T{s\lambda\phi(a)\psi'(a^+)r_0(c_dDv)_{n+\frac{1}{2}}}+\int_0^T{s^2\mathcal{O_{\lambda, \mathfrak{K}}}(sh)v^2_{n+1}}\nonumber\\
&+&\int_0^T{s\mathcal{O_{\lambda, \mathfrak{K}}}(sh)v_{n+1}(c_dDv)_{n+\frac{1}{2}}}+\int_0^T{s\mathcal{O_{\lambda, \mathfrak{K}}}(sh)r_0v_{n+1}}\nonumber\\
&=&\int_0^T{s\lambda\phi(a)[\psi'\star]_a(c_dDv)^2_{n+\frac{1}{2}}}\nonumber\\
&+&\int_0^T{2s^2\lambda^2\phi^2(a)[c\psi'\star]_a\psi'(a^+)v_{n+1}(c_dDv)_{n+\frac{1}{2}}}\nonumber\\
&+&\int_0^T{s^3\lambda^3\phi^3(a)\Big( [c\psi'\star]^2_a\psi'(a^+)+[c^2(\psi')^3 \star]_a\Big)v^2_{n+1} }\ \ +\mu_r \nonumber
\end{eqnarray}
where $\mu_r$ can be written as
\small{
\begin{eqnarray}
\mu_r&=&\int_0^T{s\mathcal{O_{\lambda}}(1)r_0^2}+\int_0^T{s^2\mathcal{O_{\lambda}}(1)r_0v_{n+1}}+\int_0^T{s\mathcal{O_{\lambda}}(1)r_0(c_dDv)_{n+\frac{1}{2}}}\nonumber\\
&+&\int_0^T{s^2\mathcal{O_{\lambda, \mathfrak{K}}}(sh)v^2_{n+1}}+\int_0^T{s\mathcal{O_{\lambda, \mathfrak{K}}}(sh)v_{n+1}(c_dDv)_{n+\frac{1}{2}}}+\int_0^T{s\mathcal{O_{\lambda, \mathfrak{K}}}(sh)r_0v_{n+1}}.\nonumber
\end{eqnarray}
}

\par We have thus achieved

\begin{equation}\mu=\int_0^T{s\lambda\phi(a)\big(\textbf{A}u(t,a),u(t,a)\big)}+\mu_r,\nonumber
\end{equation}
with $u(t,a)=\big( (c_dDv)_{n+\frac{1}{2}},s\lambda\phi(a)v_{n+1}\big)^t$ and the symmetric matrix $\textbf{A}$ defined in Lemma~\ref{lem:weight function}.

\par From the choice made for the weight function $\beta$ in Lemma~\ref{lem:weight function} we find that:
\begin{equation}\mu \ge C\alpha_0 \int_0^T{s\lambda\phi(a)(c_dDv)^2_{n+\frac{1}{2}}}+C\alpha_0\int_0^T{s^3\lambda^3\phi^3(a)v^2_{n+1}}+\mu_r,\nonumber\
\end{equation}
with $\alpha_0>0$.


\subsection{Proof of Lemma~\ref{lem:estimate Y13}}
\par By using Lemma~\ref{lem:transmission condition} we have
\small{
\begin{eqnarray}Y_{13}&=&-\int_0^T{r\bar{\tilde{\rho}}(a^+)\partial_t v(a)(c_dDv)_{n+\frac{3}{2}}}+\int_0^T{r\bar{\tilde{\rho}}(a^-)\partial_t v(a)(c_dDv)_{n+\frac{1}{2}}}\nonumber\\
&=&-\int_0^Tr\bar{\tilde{\rho}}(a^+)\partial_t v(a)\Big({(c_dDv)_{n+\frac{1}{2}}+J_1 v_{n+1}}+J_2(c_dDv)_{n+\frac{1}{2}}+J_3 h(rf)_{n+1}\Big)\nonumber\\
&&+\int_0^T{r\bar{\tilde{\rho}}(a^-)\partial_t v(a)(c_dDv)_{n+\frac{1}{2}}},\nonumber
\end{eqnarray}
}
where $J_1, J_2$ and $J_3$ are given as in Lemma~\ref{lem:transmission condition}.

\par Since $J_2=\mathcal{O_{\lambda, \mathfrak{K}}}(sh)$ and $r\bar{\tilde{\rho}}=1+\mathcal{O_{\lambda, \mathfrak{K}}}(sh)$ we can write
\begin{eqnarray}
Y_{13}&=&\int_0^T \mathcal{O_{\lambda, \mathfrak{K}}}(sh)\partial_t v(a)(c_dDv)_{n+\frac{1}{2}}-\int_0^Tr\bar{\tilde{\rho}}(a^+)J_1 v(a)\partial_t v(a)\nonumber\\
&-&\int_0^Tr\bar{\tilde{\rho}}(a^+)J_3\partial_t v(a)h(rf)_{n+1} .\nonumber
\end{eqnarray}

\par Futhermore, as $f=f_1-\partial_t (\rho v)$ we thus find 
\begin{eqnarray}
Y_{13}&=&\int_0^T \mathcal{O_{\lambda, \mathfrak{K}}}(sh)\partial_t v(a)(c_dDv)_{n+\frac{1}{2}}-\int_0^Tr\bar{\tilde{\rho}}(a^+)J_1 v(a)\partial_t v(a)\nonumber\\
&-&\int_0^Tr\bar{\tilde{\rho}}(a^+)J_3\partial_t v(a)h\big(rf_1-r\partial_t(\rho v)\big)_{n+1} .\nonumber
\end{eqnarray}

\par With an integration by parts w.r.t $t$ for the second term above we obtain

\begin{eqnarray}Y_{13}&=&\int_0^T{\mathcal{O_{\lambda, \mathfrak{K}}}(sh)\partial_tv(a)(c_dDv)_{n+\frac{1}{2}}}+\frac{1}{2}\int_0^T{\partial_t\big(r\bar{\tilde{\rho}}(a^+)J_1\big) v^2(a)}\nonumber\\
&-&\frac{1}{2} r\bar{\tilde{\rho}}(a^+)J_1v^2(a)|_{t=0}^{t=T}-\int_0^T{r\bar{\tilde{\rho}}(a^+)J_3\partial_t v(a)h(rf_1)_{n+1}}\nonumber\\
&+&\int_0^T{r\bar{\tilde{\rho}}(a^+)J_3\partial_t v(a)hr_{n+1}\big(\rho \partial_t v+\partial_t\rho v\big)_{n+1}}\nonumber\\
&=&\int_0^T{\mathcal{O_{\lambda, \mathfrak{K}}}(sh)\partial_tv(a)(c_dDv)_{n+\frac{1}{2}}}+\frac{1}{2}\int_0^T{\partial_t\big(r\bar{\tilde{\rho}}(a^+)J_1\big) v^2(a)}\nonumber\\
&+& s\mathcal{O_{\lambda, \mathfrak{K}}}(1)v^2(a)|_{t=0}^{t=T}+\int_0^T{\mathcal{O_{\lambda, \mathfrak{K}}}(1)\partial_t v(a)h(rf_1)_{n+1}}\nonumber\\
&+&\int_0^T{\big(1+\mathcal{O_{\lambda, \mathfrak{K}}}(sh)\big)h\big({\partial_tv(a)}\big)^2}+\frac{1}{2}\int_0^T{r\bar{\tilde{\rho}}(a^+)J_3h(r\partial_t \rho)_{n+1}\partial_t \big(v^2(a)\big)},\nonumber
\end{eqnarray}
where $r\bar{\tilde{\rho}}, J_{3} $ are of the form $ 1+\mathcal{O_{\lambda, \mathfrak{K}}}(sh)$ and $J_1$ of the form $s\mathcal{O_{\lambda, \mathfrak{K}}}(1)$.

\par We apply an integration by parts in time for the last term 
\begin{eqnarray}Y_{13}&=&\int_0^T{\mathcal{O_{\lambda, \mathfrak{K}}}(sh)\partial_tv(a)(c_dDv)_{n+\frac{1}{2}}}+\frac{1}{2}\int_0^T{\partial_t\big(r\bar{\tilde{\rho}}(a^+)J_1\big) v^2(a)}\nonumber\\
&+&s\mathcal{O_{\lambda, \mathfrak{K}}}(1)v^2(a)|_{t=0}^{t=T}+\int_0^T{\mathcal{O_{\lambda, \mathfrak{K}}}(1)\partial_t v(a)h(rf_1)_{n+1}}\nonumber\\
&+&\int_0^T{\big(1+\mathcal{O_{\lambda, \mathfrak{K}}}(sh)\big)h\big({\partial_tv(a)}\big)^2}-\frac{1}{2}\int_0^T{\partial_t\big(r\bar{\tilde{\rho}}(a^+)J_3(r\partial_t \rho)_{n+1}\big)h v^2(a)}\nonumber\\
&+&\frac{1}{2}r\bar{\tilde{\rho}}(a^+)J_3(r\partial_t \rho)_{n+1}hv^2(a,.)|_{t=0}^{t=T}.\nonumber
\end{eqnarray}

\par Moreover, we have
\begin{eqnarray}\label{eq:A.8}&&\partial_t s=s(2t-T)\theta=sT\theta\mathcal{O}(1),\nonumber\\
&&\partial_t \rho=-\varphi(x)(\partial_t s)\rho=-\varphi(x)s(2t-T)\theta\rho,\nonumber\\
&&r\partial_t\rho=-\varphi(x)s(2t-T)\theta\\
&&\partial_t(r\partial_t \rho)=sT^2\theta^2\mathcal{O}(1),\nonumber
\end{eqnarray} 
by using ~\eqref{eq:theta}- ~\eqref{eq:derivative of theta}.
\par Now we estimate the terms $\partial_t\big(r\bar{\tilde{\rho}}(a^+)J_1\big)$ and $\partial_t\big(r\bar{\tilde{\rho}}(a^+)J_3(r\partial_t \rho)_{n+1}\big)$. By recalling $\partial_tJ_1= sT\theta\mathcal{O_{\lambda, \mathfrak{K}}}(sh)$, $\partial_tJ_3=T\theta\mathcal{O_{\lambda, \mathfrak{K}}}(sh)$ as well as using Proposition~\ref{prop:property 5} and ~\eqref{eq:A.8} we obtain
\begin{eqnarray}
\partial_t\big(r\bar{\tilde{\rho}}(a^+)J_1\big)&=&\partial_t\big(r\bar{\tilde{\rho}}(a^+)\big)J_1+r\bar{\tilde{\rho}}(a^+)\partial_t J_1\nonumber\\
&=&sT\theta\mathcal{O_{\lambda, \mathfrak{K}}}(sh),\nonumber
\end{eqnarray} and
\begin{eqnarray}
&&\partial_t\big(r\bar{\tilde{\rho}}(a^+)J_3(r\partial_t \rho)_{n+1}\big)\nonumber\\
&=&\partial_t\big(r\bar{\tilde{\rho}}(a^+)\big)J_3(r\partial_t \rho)_{n+1}+r\bar{\tilde{\rho}}(a^+)\partial_tJ_3 (r\partial_t \rho)_{n+1}+r\bar{\tilde{\rho}}(a^+)J_3\partial_t\big((r\partial_t \rho)_{n+1}\big)\nonumber\\
&=&sT^2\theta^2\mathcal{O_{\lambda, \mathfrak{K}}}(1).\nonumber
\end{eqnarray}

\par Thus $Y_{13}$ can be written

\begin{eqnarray}Y_{13}&=&\int_0^T{\mathcal{O_{\lambda, \mathfrak{K}}}(sh)\partial_t v(a)(c_dDv)_{n+\frac{1}{2}}}+\int_0^T{sT\theta\mathcal{O_{\lambda, \mathfrak{K}}}(sh)v^2(a)}\nonumber\\
&+&s\mathcal{O_{\lambda, \mathfrak{K}}}(1)v^2(a)|_{t=0}^{t=T}+\int_0^T{\mathcal{O_{\lambda, \mathfrak{K}}}(1)\partial_t v(a)h(rf_1)_{n+1}}\nonumber\\
&+&\int_0^T{\big(1+\mathcal{O_{\lambda, \mathfrak{K}}}(sh)\big)h\big({\partial_tv(a)}\big)^2}+\int_0^T{T^2\theta^2\mathcal{O_{\lambda, \mathfrak{K}}}(sh)v^2(a) }\nonumber\\
&+&\frac{1}{2}(1+\mathcal{O_{\lambda, \mathfrak{K}}}(s(t)h)) h\Big(-\varphi(a)s(t)(2t-T)\theta(t) v^2(a,.)\Big)|_{t=0}^{t=T}.\nonumber
\end{eqnarray}

\par We observe that for $0<sh<\epsilon_3(\lambda)$ with $\epsilon_3(\lambda)$ sufficiently small we have 
\begin{equation}r\bar{\tilde{\rho}}= 1+\mathcal{O_{\lambda, \mathfrak{K}}}(sh)>0.\nonumber
\end{equation}

\par Additionally, $\varphi(x)<0$ then the last term of $Y_{13}$ are non-negative. From that, we estimate $Y_{13}$ as follows

\small{
\begin{eqnarray}&&Y_{13}\ge\int_0^T{\mathcal{C}_{\lambda, \mathfrak{K}}h(\partial_t v(a))^2}+ \int_0^T{\big(s T\theta\mathcal{O_{\lambda, \mathfrak{K}}}(sh)+T^2\theta^2\mathcal{O_{\lambda, \mathfrak{K}}}(sh) \big)v^2(a)}\nonumber\\
&+&s\mathcal{O_{\lambda, \mathfrak{K}}}(1)v^2(a)|_{t=0}^{t=T}+\int_0^T{\mathcal{O_{\lambda, \mathfrak{K}}}(sh)\partial_t v(a)(c_dDv)_{n+\frac{1}{2}}}+\int_0^T{\mathcal{O_{\lambda, \mathfrak{K}}}(1)\partial_t v(a)h(rf_1)_{n+1}}.\nonumber
\end{eqnarray}
}

\subsection{ Proof of Lemma~\ref{lem:Young ineq}}
\par On the one hands, as $f=f_1-\partial_t(\rho v) $  we write

\begin{eqnarray}(rf)_{n+1}&=&(rf_1)_{n+1}-(r\partial_t (\rho v))_{n+1}\nonumber\\
&=&(rf_1)_{n+1}-\Big((r\rho)\partial_t v+(r\partial_t \rho)v\Big)_{n+1}\nonumber\\
&=&(rf_1)_{n+1}-(\partial_t v)_{n+1}-sT\theta\mathcal{O_{\lambda}}(1)v_{n+1}\nonumber.
\end{eqnarray}

\par We thus obtain
\begin{equation}\label{eq:A.9}\left|{(rf)_{n+1}}\right|^2\le C\Big((rf_1)^2_{n+1}+(\partial_t v)^2_{n+1}+s^2T^2\theta^2\mathcal{O_{\lambda}}(1)v^2_{n+1}\Big).
\end{equation}

\par On the other hands,
\begin{equation}\label{eq:A10}[\rho_1\star \rho_2]=[\rho_1\star]_a(\rho_2)_{n+\frac{1}{2}}+\rho_1(a^+)[\star \rho_2]_a,
\end{equation}
and we recall
\begin{equation}
(c_dDv)_{n+\frac{3}{2}}-(c_dDv)_{n+\frac{1}{2}}=[\star c_dDv]_a=\lambda s[\star c\phi\psi']_av_{n+1}+r_0, \nonumber
\end{equation}
where $r_0$ is given in Lemma~\ref{lem:transmission condition} as
\begin{equation}r_0=s\mathcal{O_{\lambda, \mathfrak{K}}}(sh)v_{n+1}+\mathcal{O_{\lambda, \mathfrak{K}}}(sh)(c_dDv)_{n+\frac{1}{2}}+h\mathcal{O_{\lambda, \mathfrak{K}}}(1)(rf)_{n+1},\nonumber
\end{equation}
\par We then have
\begin{eqnarray}\label{eq:A11}(c_dDv)^2_{n+\frac{3}{2}}&=&(c_dDv)^2_{n+\frac{1}{2}}+[\star c_dDv]^2_a+2[\star c_dDv]_a(c_dDv)_{n+\frac{1}{2}}\nonumber\\
&=&(c_dDv)^2_{n+\frac{1}{2}}+\lambda^2s^2[\star c\phi\psi']^2_av^2_{n+1}+r_0^2+2\lambda s[\star c\phi\psi']_a r_0v_{n+1}\nonumber\\
&&+ 2\lambda s[\star c\phi\psi']_a v_{n+1}(c_dDv)_{n+\frac{1}{2}}+2r_0(c_dDv)_{n+\frac{1}{2}}.
\end{eqnarray}

and we compute
\begin{eqnarray}r^2_0&=& s^2\mathcal{O_{\lambda, \mathfrak{K}}}(sh)^2v^2_{n+1}+\mathcal{O_{\lambda, \mathfrak{K}}}(sh)^2(c_dDv)^2_{n+\frac{1}{2}}+h^2\mathcal{O_{\lambda, \mathfrak{K}}}(1)(rf)^2_{n+1}\nonumber\\
&&+ s\mathcal{O_{\lambda, \mathfrak{K}}}(sh)^2(c_dDv)_{n+\frac{1}{2}}v_{n+1}+ s\mathcal{O_{\lambda, \mathfrak{K}}}(sh)h(rf)_{n+1}v_{n+1}\nonumber\\
&&+\mathcal{O_{\lambda, \mathfrak{K}}}(sh)(c_dDv)_{n+\frac{1}{2}}h(rf)_{n+1}.\nonumber
\end{eqnarray}
\par By applying Cauchy-Schwartz inequality we have

\begin{eqnarray}\label{eq:A18}(c_dDv)^2_{n+\frac{3}{2}}\le\mathcal{O}(1)(c_dDv)^2_{n+\frac{1}{2}}+s^2\mathcal{O_{\lambda}}(1)v^2_{n+1}+\mathcal{O}(1)r_0^2
\end{eqnarray}

\begin{equation}\label{eq:A12}
r^2_0\le s^2\mathcal{O_{\lambda, \mathfrak{K}}}(sh)^2v^2_{n+1}+\mathcal{O_{\lambda, \mathfrak{K}}}(sh)^2(c_dDv)^2_{n+\frac{1}{2}}+h^2\mathcal{O_{\lambda, \mathfrak{K}}}(1)(rf)^2_{n+1},
\end{equation}

\begin{equation}\label{eq:A13}
sr_0v_{n+1}\le \big(s^2\mathcal{O_{\lambda, \mathfrak{K}}}(sh)+s\mathcal{O_{\lambda, \mathfrak{K}}}(1)\big)v^2_{n+1}+\mathcal{O_{\lambda, \mathfrak{K}}}(sh)(c_dDv)^2_{n+\frac{1}{2}}+h\mathcal{O_{\lambda, \mathfrak{K}}}(sh)(rf)^2_{n+1},
\end{equation}

\begin{equation}\label{eq:A14}
s^2r_0v_{n+1}\le \big(s^3\mathcal{O_{\lambda, \mathfrak{K}}}(sh)+\epsilon s^3\mathcal{O_{\lambda, \mathfrak{K}}}(1)\big)v^2_{n+1}+s\mathcal{O_{\lambda, \mathfrak{K}}}(sh)(c_dDv)^2_{n+\frac{1}{2}}+C_{\epsilon}h\mathcal{O_{\lambda, \mathfrak{K}}}(sh)(rf)^2_{n+1},
\end{equation}
\small{
\begin{eqnarray}\label{eq:A15}
sT\theta r_0v_{n+1}&\le& \big(s^2T\theta\mathcal{O_{\lambda, \mathfrak{K}}}(sh)+sT^2\theta^2\mathcal{O_{\lambda, \mathfrak{K}}}(sh)\big)v^2_{n+1}\nonumber\\
&+&s\mathcal{O_{\lambda, \mathfrak{K}}}(sh)(c_dDv)^2_{n+\frac{1}{2}}+h\mathcal{O_{\lambda, \mathfrak{K}}}(sh)(rf)^2_{n+1},
\end{eqnarray}
}
\begin{equation}\label{eq:A16}
sr_0(cDv)_{n+\frac{1}{2}}\le s^3\mathcal{O_{\lambda, \mathfrak{K}}}(sh)v^2_{n+1}+\big( s\mathcal{O_{\lambda, \mathfrak{K}}}(sh)+\epsilon s\mathcal{O_{\lambda, \mathfrak{K}}}(1)\big) (c_dDv)^2_{n+\frac{1}{2}}+C_{\epsilon}\mathcal{O_{\lambda, \mathfrak{K}}}(sh)h(rf)^2_{n+1},
\end{equation}
\small{
\begin{equation}\label{eq:A17}{(\partial_t v(a))r_0}\le{\mathcal{O_{\lambda, \mathfrak{K}}}(1)h(\partial_t v(a))^2}+{s\mathcal{O_{\lambda, \mathfrak{K}}}(sh)(c_dDv)^2_{n+\frac{1}{2}}}+{s^3\mathcal{O_{\lambda, \mathfrak{K}}}(sh)v^2_{n+1}}+{\mathcal{O_{\lambda, \mathfrak{K}}}(1)h(rf)^2_{n+1}}.
\end{equation}
}

\par We estimate following terms \\
\par The first term, by using  ~\eqref{eq:A18} we have
\small{
\begin{eqnarray}\left| Y_{11}^{(2,2)}\right| &=&\int_0^T{s\mathcal{O_{\lambda, \mathfrak{K}}}(sh)^2(c_dDv)^2_{n+\frac{1}{2}}}+{s\mathcal{O_{\lambda,\mathfrak{K}}}(sh)^2(c_dDv)^2_{n+\frac{3}{2}}}\nonumber\\
&\le&\int_0^T{s\mathcal{O_{\lambda, \mathfrak{K}}}(sh)^2(c_dDv)^2_{n+\frac{1}{2}}}+\int_0^T{s^3\mathcal{O_{\lambda, \mathfrak{K}}}(sh)^2v^2_{n+1}}+\int_0^T{s\mathcal{O_{\lambda, \mathfrak{K}}}(sh)^2r_0^2}.\nonumber
\end{eqnarray} 
}
\par Moreover, by using ~\eqref{eq:A12} we obtain

\begin{eqnarray}&&\int_0^T{s\mathcal{O_{\lambda, \mathfrak{K}}}(sh)^2r_0^2}\nonumber\\
&\le&\int_0^T{s^3\mathcal{O_{\lambda, \mathfrak{K}}}(sh)^4v^2_{n+1}}+\int_0^T{s\mathcal{O_{\lambda, \mathfrak{K}}}(sh)^4(c_dDv)^2_{n+\frac{1}{2}}}+\int_0^T{h\mathcal{O_{\lambda, \mathfrak{K}}}(sh)^3(rf)^2_{n+1}}.\nonumber
\end{eqnarray}



\par Then, by using ~\eqref{eq:A.9} we estimate ${Y_{11}^{(2,2)}}$ 
\small{
\begin{eqnarray}\left|{Y_{11}^{(2,2)}}\right|&\le&\int_0^T{s^3\mathcal{O_{\lambda, \mathfrak{K}}}(sh)^2v^2_{n+1}}+\int_0^T{s\mathcal{O_{\lambda, \mathfrak{K}}}(sh)^2(c_dDv)^2_{n+\frac{1}{2}}}+\int_0^T{h\mathcal{O_{\lambda, \mathfrak{K}}}(sh)^3(rf)^2_{n+1}}\nonumber\\
&\le&\int_0^T{\Big(s^3\mathcal{O_{\lambda, \mathfrak{K}}}(sh)^2+sT^2\theta^2\mathcal{O_{\lambda, \mathfrak{K}}}(sh)^4\Big)v^2_{n+1}}+\int_0^T{\mathcal{O_{\lambda, \mathfrak{K}}}(sh)^3h (\partial_t v)^2_{n+1}}\nonumber\\
&+&\int_0^T{s\mathcal{O_{\lambda, \mathfrak{K}}}(sh)^2(c_dDv)^2_{n+\frac{1}{2}}}+\int_0^T{h\mathcal{O_{\lambda, \mathfrak{K}}}(sh)^3(rf_1)^2_{n+1}}.\nonumber
\end{eqnarray}
}

\par For the next term, using ~\eqref{eq:A10} and Lemma~\ref{lem:transmission condition}  we obtain
\small{
\begin{eqnarray}Y_{12}&=&\int_0^T{s\mathcal{O_{\lambda}}(1)v(a)[c\psi^{'2}\star (c_dDv)]_a}\nonumber\\
&+&\int_0^T{s\mathcal{O_{\lambda, \mathfrak{K}}}(1)v(a)(cDv)_{n+\frac{1}{2}}+s\mathcal{O_{\lambda, \mathfrak{K}}}(1)v(a)(cDv)_{n+\frac{3}{2}}}\nonumber\\
&=&\int_0^T{s\mathcal{O_{\lambda, \mathfrak{K}}}(1)v(a)(c_dDv)_{n+\frac{1}{2}}+s\mathcal{O_{\lambda, \mathfrak{K}}}(1)v(a)\Big((c_dDv)_{n+\frac{1}{2}}+ s\mathcal{O_{\lambda}}(1)v(a)+r_0\Big)}\nonumber\\
&=&\int_0^T{s\mathcal{O_{\lambda, \mathfrak{K}}}(1)v(a)(c_dDv)_{n+\frac{1}{2}}+s^2\mathcal{O_{\lambda, \mathfrak{K}}}(1)v^2(a)+s\mathcal{O_{\lambda, \mathfrak{K}}}(1)v(a)r_0}.\nonumber
\end{eqnarray}
 }
\par Using ~\eqref{eq:A13} yields
\small{
\begin{eqnarray}&&\int_0^T{s\mathcal{O_{\lambda, \mathfrak{K}}}(1)v(a)r_0}\nonumber\\
&\le&\int_0^T{\Big(s^2\mathcal{O_{\lambda, \mathfrak{K}}}(sh)+s\mathcal{O_{\lambda, \mathfrak{K}}}(1)\Big)v^2_{n+1}}+\int_0^T{\mathcal{O_{\lambda, \mathfrak{K}}}(sh)(c_dDv)^2_{n+\frac{1}{2}}}+\int_0^T{ h\mathcal{O_{\lambda, \mathfrak{K}}}(sh)(rf)^2_{n+1}}.\nonumber
\end{eqnarray}
}
\par By using ~\eqref{eq:A.9} we obtain

\begin{eqnarray}\left|{Y_{12}}\right|&\le&\int_0^T{\big(s^2\mathcal{O_{\lambda, \mathfrak{K}}}(1)+sT^2\theta^2\mathcal{O_{\lambda, \mathfrak{K}}}(sh)^2\big)v^2_{n+1}}+\int_0^T{\mathcal{O_{\lambda, \mathfrak{K}}}(sh)h (\partial_t v)^2_{n+1}}\nonumber\\
&+&\int_0^T{\mathcal{O_{\lambda, \mathfrak{K}}}(1)(c_dDv)^2_{n+\frac{1}{2}}}+\int_0^T{ \mathcal{O_{\lambda, \mathfrak{K}}}(sh)h(rf_1)^2_{n+1}}.\nonumber
\end{eqnarray}

\par Moreover, we have
\begin{eqnarray}\label{eq:A19}&& \tilde{v}_{n+\frac{1}{2}}=v_{n+1}-\frac{h}{2(c_d)_{n+\frac{1}{2}}}(c_dDv)_{n+\frac{1}{2}}=v_{n+1}+\mathcal{O}(h)(c_dDv)_{n+\frac{1}{2}},\nonumber\\
&& \tilde{v}_{n+\frac{3}{2}}=v_{n+1}+\mathcal{O}(h)(c_dDv)_{n+\frac{3}{2}}.
\end{eqnarray}

\par  By using ~\eqref{eq:A19}, ~\eqref{eq:A18} we obtain
\begin{eqnarray}
(\tilde{v})^2_{n+\frac{1}{2}}+(\tilde{v})^2_{n+\frac{3}{2}}&\le& \mathcal{O}(1)v^2_{n+1}+\mathcal{O}(h^2)(cDv)^2_{n+\frac{3}{2}}+\mathcal{O}(h^2)(cDv)^2_{n+\frac{1}{2}}\nonumber\\
&\le&\big(\mathcal{O}(1)+\mathcal{O}(sh)^2\big)v^2_{n+1}+\mathcal{O}(h^2)(cDv)^2_{n+\frac{1}{2}}+\mathcal{O}(h^2)r^2_0\nonumber
\end{eqnarray}

\par Thus, we have the following estimate

\small{
\begin{eqnarray}\left| Y_{21}^{(1,22)}\right| &=&\int_0^T{\Big(s^2\mathcal{O}(1)+s^3\mathcal{O_{\lambda, \mathfrak{K}}}(sh)^2\Big)\Big((\tilde{v})^2_{n+\frac{1}{2}}+(\tilde{v})^2_{n+\frac{3}{2}}\Big)}\nonumber\\
&=&\int_0^T{s^3\mathcal{O_{\lambda, \mathfrak{K}}}(1)\Big((\tilde{v})^2_{n+\frac{1}{2}}+(\tilde{v})^2_{n+\frac{3}{2}}\Big)}\nonumber\\
&\le&\int_0^T{s^3\mathcal{O_{\lambda, \mathfrak{K}}}(1)v^2_{n+1}}+\int_0^T{s\mathcal{O_{\lambda, \mathfrak{K}}}(sh)^2(c_dDv)^2_{n+\frac{1}{2}}}+\int_0^T{s\mathcal{O_{\lambda, \mathfrak{K}}}(sh)^2r^2_0}.\nonumber
\end{eqnarray}
}

\par Futhermore, using ~\eqref{eq:A12} we have

\begin{eqnarray}&&\int_0^T{s\mathcal{O_{\lambda, \mathfrak{K}}}(sh)^2r^2_0}\nonumber\\
&\le&\int_0^T{s^3\mathcal{O_{\lambda, \mathfrak{K}}}(sh)^4v^2_{n+1}}+\int_0^T{s\mathcal{O_{\lambda, \mathfrak{K}}}(sh)^4(c_dDv)^2_{n+\frac{1}{2}}}+\int_0^T{h\mathcal{O_{\lambda, \mathfrak{K}}}(sh)^3(rf)^2_{n+1}}.\nonumber
\end{eqnarray}

\par By using ~\eqref{eq:A.9} we get
\small{
\begin{eqnarray}\left|{Y_{21}^{(1,22)}}\right|&\le&\int_0^T{\Big(s^3\mathcal{O_{\lambda, \mathfrak{K}}}(1)+sT^2\theta^2\mathcal{O_{\lambda, \mathfrak{K}}}(sh)^4\Big)v^2_{n+1}}+\int_0^T{\mathcal{O_{\lambda, \mathfrak{K}}}(sh)^3h (\partial_t v)^2_{n+1}}\nonumber\\
&+&\int_0^T{s\mathcal{O_{\lambda, \mathfrak{K}}}(sh)^3(c_dDv)^2_{n+\frac{1}{2}}}+\int_0^T{ h\mathcal{O_{\lambda, \mathfrak{K}}}(sh)^3(rf_1)^2_{n+1}}.\nonumber
\end{eqnarray}
}

\par For the term $Y_{22}^{(1)}$ we have
\small{
\begin{eqnarray}Y_{22}^{(1)}&=&\int_0^T{s^3\mathcal{O_{\lambda, \mathfrak{K}}}(1)v(a)\frac{h^2}{2}(Dv)_{n+\frac{1}{2}}+s^3\mathcal{O_{\lambda, \mathfrak{K}}}(1)v(a)\frac{h^2}{2}(Dv)_{n+\frac{3}{2}}}\nonumber\\
&=&\int_0^T{\Big(s\mathcal{O_{\lambda, \mathfrak{K}}}(sh)^2(c_dDv)_{n+\frac{1}{2}}+s\mathcal{O_{\lambda, \mathfrak{K}}}(sh)^2(c_dDv)_{n+\frac{3}{2}}\Big)v(a)}\nonumber\\
&=&\int_0^T{\Big(s\mathcal{O_{\lambda, \mathfrak{K}}}(sh)^2(c_dDv)_{n+\frac{1}{2}}+ s^2\mathcal{O_{\lambda, \mathfrak{K}}}(sh)^2v(a)+s\mathcal{O_{\lambda, \mathfrak{K}}}(sh)^2r_0\Big)v(a)}\nonumber\\
&=&\int_0^T{s\mathcal{O_{\lambda, \mathfrak{K}}}(sh)^2v(a)(c_dDv)_{n+\frac{1}{2}}}+\int_0^T{s^2\mathcal{O_{\lambda, \mathfrak{K}}}(sh)^2v^2(a)}+\int_0^T{s\mathcal{O_{\lambda, \mathfrak{K}}}(sh)^2v(a)r_0}.\nonumber
\end{eqnarray}
}

\par Using ~\eqref{eq:A13} we achieve
\small{
\begin{eqnarray}&&\int_0^T{s\mathcal{O_{\lambda, \mathfrak{K}}}(sh)^2v(a)r_0}\nonumber\\
&\le&\int_0^T{s^2\mathcal{O_{\lambda, \mathfrak{K}}}(sh)^2v^2_{n+1}}+\int_0^T{\mathcal{O_{\lambda, \mathfrak{K}}}(sh)^3(c_dDv)^2_{n+\frac{1}{2}}}+\int_0^T{h\mathcal{O_{\lambda, \mathfrak{K}}}(sh)^3(rf)^2_{n+1}}.\nonumber
\end{eqnarray}
}

\par Using ~\eqref{eq:A.9}, we estimate $Y_{22}^{(1)}$ as:
\begin{eqnarray}\left|{Y_{22}^{(1)}}\right|&\le&\int_0^T{\Big(s^2\mathcal{O_{\lambda, \mathfrak{K}}}(sh)^2+sT^2\theta^2\mathcal{O_{\lambda, \mathfrak{K}}}(sh)^4\Big)v^2_{n+1}}+\int_0^T{\mathcal{O_{\lambda, \mathfrak{K}}}(sh)^3h (\partial_t v)^2_{n+1}}\nonumber\\
&+&\int_0^T{\mathcal{O_{\lambda, \mathfrak{K}}}(sh)^3(c_dDv)^2_{n+\frac{1}{2}}}+\int_0^T{h\mathcal{O_{\lambda, \mathfrak{K}}}(sh)^3(rf_1)^2_{n+1}}.\nonumber
\end{eqnarray}

\par And, using ~\eqref{eq:A19} and Lemma~\ref{lem:transmission condition} we obtain
\begin{eqnarray}&&Y_{23}^{(1)}=\int_0^T{s^2\mathcal{O_{\lambda, \mathfrak{K}}}(1)(\partial_t v(a))\frac{h}{2}\tilde{v}_{n+\frac{1}{2}}}+\int_0^T{s^2\mathcal{O_{\lambda, \mathfrak{K}}}(1)(\partial_t v(a))\frac{h}{2}\tilde{v}_{n+\frac{3}{2}}}\nonumber\\
&=&\int_0^T{s^2\mathcal{O_{\lambda, \mathfrak{K}}}(1)(\partial_t v(a))\frac{h}{2}\Big(v_{n+1}+\mathcal{O}(h)(c_dDv)_{n+\frac{1}{2}}\Big)}\nonumber\\
&+&\int_0^T{s^2\mathcal{O_{\lambda, \mathfrak{K}}}(1)(\partial_t v(a))\frac{h}{2}\Big(v_{n+1}+\mathcal{O}(h)(c_dDv)_{n+\frac{3}{2}}\Big)}\nonumber\\
&=&\int_0^T{s\mathcal{O_{\lambda, \mathfrak{K}}}(sh)(\partial_t v(a))v(a)}+\int_0^T{\mathcal{O_{\lambda, \mathfrak{K}}}(sh)^2(\partial_t v(a))(c_dDv)_{n+\frac{1}{2}}}\nonumber\\
&+&\int_0^T{\mathcal{O_{\lambda, \mathfrak{K}}}(sh)^2(\partial_t v(a))(c_dDv)_{n+\frac{3}{2}}}\nonumber\\
&=&\int_0^T{s\mathcal{O_{\lambda, \mathfrak{K}}}(sh)(\partial_t v(a))v(a)}+\int_0^T{\mathcal{O_{\lambda, \mathfrak{K}}}(sh)^2(\partial_t v(a))(c_dDv)_{n+\frac{1}{2}}}\nonumber\\
&+&\int_0^T{\mathcal{O_{\lambda, \mathfrak{K}}}(sh)^2(\partial_t v(a))\Big((c_dDv)_{n+\frac{1}{2}}+\lambda s[\star c\phi\psi']_av_{n+1}+r_0\Big)}.\nonumber
\end{eqnarray}

\par In addition, with $s, \lambda$ enough large, $sh$ enough small and with applying Young's inequality and ~\eqref{eq:A17} yield

\begin{equation}\int_0^T{s\mathcal{O_{\lambda, \mathfrak{K}}}(sh)(\partial_t v(a))v(a)}\le \int_0^T{\mathcal{O_{\lambda, \mathfrak{K}}}(sh)h(\partial_t v(a))^2}+\int_0^T{s^3\mathcal{O_{\lambda, \mathfrak{K}}}(1)v^2(a)}.\nonumber
\end{equation}

\begin{equation}\int_0^T{\mathcal{O_{\lambda, \mathfrak{K}}}(sh)^2(\partial_t v(a))(c_dDv)_{n+\frac{1}{2}}}\le \int_0^T{\mathcal{O_{\lambda, \mathfrak{K}}}(sh)h(\partial_t v(a))^2}+\int_0^T{s\mathcal{O_{\lambda, \mathfrak{K}}}(sh)^2(c_dDv)^2_{n+\frac{1}{2}}}.\nonumber
\end{equation}

\begin{eqnarray}\int_0^T{\mathcal{O_{\lambda, \mathfrak{K}}}(sh)^2(\partial_t v(a))r_0}&\le& \int_0^T{\mathcal{O_{\lambda, \mathfrak{K}}}(sh)^2h(\partial_t v(a))^2}+\int_0^T{s\mathcal{O_{\lambda, \mathfrak{K}}}(sh)^3(c_dDv)^2_{n+\frac{1}{2}}}\nonumber
\\
&+&\int_0^T{s^3\mathcal{O_{\lambda, \mathfrak{K}}}(sh)^3(v)^2_{n+1}}+\int_0^T{\mathcal{O_{\lambda, \mathfrak{K}}}(sh)^2h(rf)^2_{n+1}}.\nonumber
\end{eqnarray}

\par Using ~\eqref{eq:A.9}, we estimate $Y_{23}^{(1)}$ 

\begin{eqnarray}Y_{23}^{(1)}&\le& \int_0^T{\Big(s^3\mathcal{O_{\lambda, \mathfrak{K}}}(1)+sT^2\theta^2\mathcal{O_{\lambda, \mathfrak{K}}}(sh)^3\Big)v^2_{n+1}}+ \int_0^T{\mathcal{O_{\lambda, \mathfrak{K}}}(sh)h(\partial_t v(a))^2}\nonumber\\
&&+\int_0^T{s\mathcal{O_{\lambda, \mathfrak{K}}}(sh)^2(c_dDv)^2_{n+\frac{1}{2}}}+\int_0^T{\mathcal{O_{\lambda, \mathfrak{K}}}(sh)^2h(rf_1)^2_{n+1}}.\nonumber
\end{eqnarray}

\par Applying Lemma~\ref{lem:transmission condition} we have

\begin{eqnarray}Y_{31}^{(1)}&=&\int_0^T{T\theta s^2\mathcal{O_{\lambda, \mathfrak{K}}}(1)v(a)\frac{h}{2}(Dv)_{n+\frac{1}{2}}}+\int_0^T{T\theta s^2\mathcal{O_{\lambda, \mathfrak{K}}}(1)v(a)\frac{h}{2}(Dv)_{n+\frac{3}{2}}}\nonumber\\
&=&\int_0^T{sT\theta\mathcal{O_{\lambda, \mathfrak{K}}}(sh)v(a)(c_dDv)_{n+\frac{1}{2}}}\nonumber\\
&+&\int_0^T{sT\theta\mathcal{O_{\lambda, \mathfrak{K}}}(sh)v(a)\Big( s\mathcal{O_{\lambda, \mathfrak{K}}}(1)v(a)+r_0\Big)}\nonumber
\end{eqnarray}

\par By using ~\eqref{eq:A15} we obtain
\small{
\begin{eqnarray}\int_0^T{sT\theta\mathcal{O_{\lambda, \mathfrak{K}}}(sh)v_{n+1}r_0}&\le&\int_0^T{\Big( s^2T\theta\mathcal{O_{\lambda, \mathfrak{K}}}(sh)^2+sT^2\theta^2\mathcal{O_{\lambda, \mathfrak{K}}}(sh)^2\Big)v^2_{n+1}}\nonumber\\
&+&\int_0^T{s\mathcal{O_{\lambda, \mathfrak{K}}}(sh)^2(c_dDv)^2_{n+\frac{1}{2}}}+\int_0^T{\mathcal{O_{\lambda, \mathfrak{K}}}(sh)^2h(rf)^2_{n+1}}.\nonumber
\end{eqnarray}
}
\par We have
\begin{equation}\int_0^T{sT\theta\mathcal{O_{\lambda, \mathfrak{K}}}(sh)(c_dDv)_{n+\frac{1}{2}}}v_{n+1}\le\int_0^T{sT^2\theta^2\mathcal{O_{\lambda, \mathfrak{K}}}(sh)v^2_{n+1}}+\int_0^T{s\mathcal{O_{\lambda, \mathfrak{K}}}(sh)(c_dDv)^2_{n+\frac{1}{2}}}.\nonumber
\end{equation}

\par With ~\eqref{eq:A.9} we thus estimate $Y_{31}^{(1)}$ as

\begin{eqnarray}\left|{Y_{31}^{(1)}}\right|&\le& \int_0^T{\Big( s^2T\theta\mathcal{O_{\lambda, \mathfrak{K}}}(sh)+sT^2\theta^2\mathcal{O_{\lambda, \mathfrak{K}}}(sh)\Big)v^2_{n+1}}+\int_0^T{\mathcal{O_{\lambda, \mathfrak{K}}}(sh)h(\partial_t v)^2_{n+1}}\nonumber\\
&+&\int_0^T{s\mathcal{O_{\lambda, \mathfrak{K}}}(sh)(c_dDv)^2_{n+\frac{1}{2}}}+\int_0^T{\mathcal{O_{\lambda, \mathfrak{K}}}(sh)h(rf_1)^2_{n+1}}.\nonumber
\end{eqnarray}

\par Next, by using ~\eqref{eq:A18} we estimate $\mu_1^{(1)}$ as
\begin{eqnarray}
\mu^{(1)}&=&\int_0^T{s\mathcal{O_{\lambda, \mathfrak{K}}}(sh)(c_dDv)^2_{n+\frac{3}{2}}}+\int_0^T{s\mathcal{O_{\lambda, \mathfrak{K}}}(sh)(c_dDv)^2_{n+\frac{1}{2}}}\nonumber\\
&\le&\int_0^T{s\mathcal{O_{\lambda, \mathfrak{K}}}(sh)(c_dDv)^2_{n+\frac{1}{2}}}+\int_0^T{s^3\mathcal{O_{\lambda, \mathfrak{K}}}(sh)v^2_{n+1}}+\int_0^T{s\mathcal{O_{\lambda, \mathfrak{K}}}(sh)r^2_0}\nonumber
\end{eqnarray}

\par By making use of ~\eqref{eq:A12}  we have
\small{
\begin{eqnarray}&&\int_0^T{s\mathcal{O_{\lambda, \mathfrak{K}}}(sh)r_0^2}\nonumber\\
&\le&\int_0^T{s^3\mathcal{O_{\lambda, \mathfrak{K}}}(sh)^3v^2_{n+1}}+\int_0^T{s\mathcal{O_{\lambda, \mathfrak{K}}}(sh)^3(c_dDv)^2_{n+\frac{1}{2}}}+\int_0^T{\mathcal{O_{\lambda, \mathfrak{K}}}(sh)^2h(rf)^2_{n+1}}.\nonumber
\end{eqnarray}
}

%
%

\par Using ~\eqref{eq:A.9} we obtain
\begin{eqnarray}\mu_1^{(1)}&\le& \int_0^T{\Big(s^3\mathcal{O_{\lambda, \mathfrak{K}}}(sh)+sT^2\theta^2\mathcal{O_{\lambda, \mathfrak{K}}}(sh)^3\Big)v^2_{n+1}}+\int_0^T{\mathcal{O_{\lambda, \mathfrak{K}}}(sh)^2h(\partial_t v)^2_{n+1}}\nonumber\\
&+&\int_0^T{s\mathcal{O_{\lambda, \mathfrak{K}}}(sh)(c_dDv)^2_{n+\frac{1}{2}}}+\int_0^T{\mathcal{O_{\lambda, \mathfrak{K}}}(sh)^2h(rf_1)^2_{n+1}}.\nonumber
\end{eqnarray}

\par By making use Lemma~\ref{lem:transmission condition} we have
\small{
\begin{eqnarray}\mu_1^{(2)}&=&\int_0^T{s^2\mathcal{O_{\lambda, \mathfrak{K}}}(sh)v_{n+1}(c_dDv)_{n+\frac{3}{2}}}+\int_0^T{s^2\mathcal{O_{\lambda, \mathfrak{K}}}(sh)v_{n+1}(c_dDv)_{n+\frac{1}{2}}}\nonumber\\
&=&\int_0^T{s^2\mathcal{O_{\lambda, \mathfrak{K}}}(sh)v_{n+1}(c_dDv)_{n+\frac{1}{2}}}+\int_0^T{s^3\mathcal{O_{\lambda, \mathfrak{K}}}(sh)v^2_{n+1}}+\int_0^T{s^2\mathcal{O_{\lambda, \mathfrak{K}}}(sh)r_0v_{n+1}}.\nonumber
\end{eqnarray}
}
\par Applying Young's inequality and using ~\eqref{eq:A14} yield
\begin{equation}\int_0^T{s^2\mathcal{O_{\lambda, \mathfrak{K}}}(sh)v_{n+1}(c_dDv)_{n+\frac{1}{2}}}\le\int_0^T{s^3\mathcal{O_{\lambda, \mathfrak{K}}}(sh)v^2_{n+1}}+\int_0^T{s\mathcal{O_{\lambda, \mathfrak{K}}}(sh)(c_dDv)^2_{n+\frac{1}{2}}}.\nonumber
\end{equation}

\small{
\begin{eqnarray}&&\int_0^T{s^2\mathcal{O_{\lambda, \mathfrak{K}}}(sh)r_0v_{n+1}}\nonumber\\
&\le&\int_0^T{s^3\mathcal{O_{\lambda, \mathfrak{K}}}(sh)^2v^2_{n+1}}+\int_0^T{s\mathcal{O_{\lambda, \mathfrak{K}}}(sh)^2(c_dDv)^2_{n+\frac{1}{2}}}+\int_0^T{\mathcal{O_{\lambda, \mathfrak{K}}}(sh)^2h(rf)^2_{n+1}}.\nonumber
\end{eqnarray}
}

\par Using ~\eqref{eq:A.9} we have
\small{
\begin{eqnarray}\mu_1^{(2)}&\le& \int_0^T{\Big(s^3\mathcal{O_{\lambda, \mathfrak{K}}}(sh)+sT^2\theta^2\mathcal{O_{\lambda, \mathfrak{K}}}(sh)^3\Big)v^2_{n+1}}+\int_0^T{\mathcal{O_{\lambda, \mathfrak{K}}}(sh)^2h(\partial_t v)^2_{n+1}}\nonumber\\
&+&\int_0^T{s\mathcal{O_{\lambda, \mathfrak{K}}}(sh)(c_dDv)^2_{n+\frac{1}{2}}}+\int_0^T{\mathcal{O_{\lambda, \mathfrak{K}}}(sh)^2h(rf_1)^2_{n+1}}.\nonumber
\end{eqnarray}
}

\par We thus obtain
\begin{eqnarray}\mu_1&\le& \int_0^T{\Big(s^3\mathcal{O_{\lambda, \mathfrak{K}}}(sh)+sT^2\theta^2\mathcal{O_{\lambda, \mathfrak{K}}}(sh)^3\Big)v^2_{n+1}}+\int_0^T{\mathcal{O_{\lambda, \mathfrak{K}}}(sh)^2h(\partial_t v)^2_{n+1}}\nonumber\\
&+&\int_0^T{s\mathcal{O_{\lambda, \mathfrak{K}}}(sh)(c_dDv)^2_{n+\frac{1}{2}}}+\int_0^T{\mathcal{O_{\lambda, \mathfrak{K}}}(sh)^2h(rf_1)^2_{n+1}}.\nonumber
\end{eqnarray}

\par Now, we estimate some terms of $\mu_r$. By using ~\eqref{eq:A12}- ~\eqref{eq:A16} we have
\small{
\begin{eqnarray}&&\int_0^T{s\mathcal{O_{\lambda}}(1)r_0^2}\nonumber\\
&\le&\int_0^T{s^3\mathcal{O_{\lambda, \mathfrak{K}}}(sh)^2v^2_{n+1}}+\int_0^T{s\mathcal{O_{\lambda, \mathfrak{K}}}(sh)^2(c_dDv)^2_{n+\frac{1}{2}}}+\int_0^T{\mathcal{O_{\lambda, \mathfrak{K}}}(sh)h(rf)^2_{n+1}}.\nonumber
\end{eqnarray}
}

\small{
\begin{eqnarray}\int_0^T{s^2\mathcal{O_{\lambda}}(1)r_0v_{n+1}}&\le&\int_0^T{\Big(s^3\mathcal{O_{\lambda, \mathfrak{K}}}(sh)+\epsilon s^3\mathcal{O_{\lambda, \mathfrak{K}}}(1)\Big)v^2_{n+1}}\nonumber\\
&+&\int_0^T{s\mathcal{O_{\lambda, \mathfrak{K}}}(sh)(c_dDv)^2_{n+\frac{1}{2}}}+C_{\epsilon}\int_0^T{\mathcal{O_{\lambda, \mathfrak{K}}}(sh)h(rf)^2_{n+1}}.\nonumber\end{eqnarray}
}

\small{
\begin{eqnarray}\int_0^T{s\mathcal{O_{\lambda}}(1)r_0(c_dDv)_{n+\frac{1}{2}}}&\le&\int_0^T{s^3\mathcal{O_{\lambda, \mathfrak{K}}}(sh)v^2_{n+1}}+C_{\epsilon}\int_0^T{\mathcal{O_{\lambda, \mathfrak{K}}}(sh)h(rf)^2_{n+1}}\nonumber\\
&+&\int_0^T{\Big({s\mathcal{O_{\lambda, \mathfrak{K}}}(sh)+\epsilon s\mathcal{O_{\lambda, \mathfrak{K}}}(1)}\Big)(c_dDv)^2_{n+\frac{1}{2}}}.\nonumber
\end{eqnarray}
}
\begin{equation}\int_0^T{s\mathcal{O_{\lambda, \mathfrak{K}}}(sh)v_{n+1}(c_dDv)_{n+\frac{1}{2}}}\le\int_0^T{s\mathcal{O_{\lambda, \mathfrak{K}}}(sh)(c_dDv)^2_{n+\frac{1}{2}}}+\int_0^T{s\mathcal{O_{\lambda, \mathfrak{K}}}(sh)v^2_{n+1}}.\nonumber
\end{equation}

\par Using ~\eqref{eq:A.9} we have:
\begin{eqnarray}\mu_r&\le& \int_0^T{\Big(s^3\mathcal{O_{\lambda, \mathfrak{K}}}(sh)+sT^2\theta^2\mathcal{O_{\epsilon, \lambda, \mathfrak{K}}}(sh)^2+\epsilon s^3\mathcal{O_{\lambda, \mathfrak{K}}}(1)\Big)v^2_{n+1}}\nonumber\\
&+&\int_0^T{\mathcal{O_{\epsilon, \lambda, \mathfrak{K}}}(sh)h(\partial_t v)^2_{n+1}}+\int_0^T{\mathcal{O_{\epsilon, \lambda, \mathfrak{K}}}(sh)h(rf_1)^2_{n+1}}\nonumber\\
&+&\int_0^T{\Big(s\mathcal{O_{\lambda, \mathfrak{K}}}(sh)+\epsilon s\mathcal{O_{\lambda, \mathfrak{K}}}(1)\Big)}(c_dDv)^2_{n+\frac{1}{2}}.\nonumber
\end{eqnarray}

\section*{Acknowledgement}
The author warmly thanks J\'er\^{o}me Le Rousseau for his helps and remarks. The author also wishes to acknowledge Region Centre for its financial support.

\newpage

\renewcommand{\refname}{References}

\end{document}